\documentclass[11pt,a4paper]{article}
\usepackage{amsmath}
\usepackage{amsfonts}
\usepackage{amssymb}
\usepackage{setspace}
\usepackage{amsthm}
\usepackage{rotating}
\usepackage[a4paper, left=2.5cm, right=2.5cm, top=2.8cm, bottom=2.8cm]{geometry}
\usepackage{float}
\usepackage{amsmath}
\usepackage{graphicx}
\usepackage{subfigure,bm,enumerate}
\usepackage{natbib}
\usepackage{bbm}
\usepackage{color}

\newcommand{\bol}[1]{\mbox{\boldmath$#1$}}
\newcommand{\bmu}{\bol{\mu}}
\newcommand{\btheta}{\bol{\theta}}

\newcommand{\bnu}{\bol{\nu}}
\newcommand{\bpsi}{\bol{\psi}}
\newcommand{\bl}{\mathbf{l}}
\newcommand{\bI}{\mathbf{I}}
\newcommand{\bx}{\mathbf{x}}

\newcommand{\bu}{\mathbf{u}}
\newcommand{\by}{\mathbf{y}}
\newcommand{\bS}{\mathbf{S}}
\newcommand{\bB}{\mathbf{B}}
\newcommand{\bSigma}{\bol{\Sigma}}
\newcommand{\bOmega}{\bol{\Omega}}
\newcommand{\bomega}{\bol{\omega}}
\newcommand{\bX}{\mathbf{X}}
\newcommand{\bY}{\mathbf{Y}}
\newcommand{\bV}{\mathbf{V}}
\newcommand{\bzero}{\mathbf{0}}

\usepackage{color}

\usepackage{ulem}

\newtheorem{theorem}{Theorem}
\newtheorem{proposition}{Proposition}

\newtheorem{corollary}{Corollary}

\begin{document}
\begin{center}
\vspace*{2cm} \noindent
{\bf \large
  Central limit theorems for functionals of large sample covariance matrix and mean vector in matrix-variate location mixture of normal distributions
}\\
\vspace{1cm} \noindent {\sc Taras Bodnar$^{a,}$\footnote{Corresponding author. E-mail address: taras.bodnar@math.su.se.
The second author appreciates the financial support of the Swedish Research Council Grant Dnr: 2013-5180 and Riksbankens Jubileumsfond
Grant Dnr: P13-1024:1}, Stepan Mazur$^{b}$, Nestor Parolya$^{c}$
}\\
\vspace{1cm}
{\it \footnotesize  $^a$
 Department of Mathematics, Stockholm University, SE-10691 Stockholm, Sweden}\\
{\it \footnotesize  $^b$Department of Statistics, School of Business, \"Orebro University, SE-70182 \"Orebro, Sweden
} \\
{\it \footnotesize  $^c$
Institute of Statistics, Leibniz University of Hannover, D-30167 Hannover, Germany}\\
\end{center}


\begin{abstract}
In this paper we consider the asymptotic distributions of functionals of the sample covariance matrix and the sample mean vector obtained under the assumption that the matrix of observations has a matrix-variate location mixture of normal distributions. The central limit theorem is derived for the product of the sample covariance matrix and the sample mean vector. Moreover, we consider the product of the inverse sample covariance matrix and the mean vector for which the central limit theorem is established as well. All results are obtained under the large-dimensional asymptotic regime where the dimension $p$ and the sample size $n$ approach to infinity such that $p/n\to c\in[0 , +\infty)$ when the sample covariance matrix does not need to be invertible and $p/n\to c\in [0, 1)$ otherwise.
\end{abstract}

\noindent ASM Classification: 62H10, 62E15, 62E20, 60F05, 60B20   

\newpage
\section{Introduction}
The functions of the sample covariance matrix and the sample mean vector appear in various statistical applications.
The classical improvement techniques for the mean estimation have already been discussed by \citet{Stein1956} and \citet{Jorion1986}.
In particular, \citet{Efron2006} constructed confidence regions of smaller volume than the standard spheres for the mean vector of a multivariate normal distribution.
\citet{fan2008}, \citet{baishi2011}, \citet{BodnarGupta2011}, \citet{cai20121}, \citet{cai20122}, \citet{fan2013}, \citet{bodguppar2014}, \citet{wang2015}, \citet{bodguppar2015} among others suggested improved techniques for the estimation of covariance matrix and precision matrix (the inverse of covariance matrix).

In our work we introduce the family of matrix-variate location mixture of normal distributions (MVLMN) which is a generalization of the models considered by \citet{AzzaliniDallaValle1996}, \citet{AzzaliniCapitanio1999}, \citet{Azzalini2005}, \citet{LiseoLoperfido2003,LiseoLoperfido2006}, \citet{BartolettiLoperfido2010}, \citet{Loperfido2010}, \citet{ChristiansenLoperfido2014}, \citet{Adcock2015}, \citet{DeLucaLoperfido2015} among others.
Under the assumption of MVLMN we consider the expressions for the sample mean vector $\overline{\mathbf{x}}$ and the sample covariance matrix $\mathbf{S}$.
In particulary, we deal with two products $\bl^\top\mathbf{S}\overline{\mathbf{x}}$ and $\bl^\top\mathbf{S}^{-1}\overline{\mathbf{x}}$ where $\mathbf{l}$ is a non-zero vector of constants. It is noted that this kind of expressions has not been intensively considered in the literature, although they are present in numerous important applications.
The first application of the products arises in the portfolio theory, where the vector of optimal portfolio weights is proportional to $\mathbf{S}^{-1}\overline{\mathbf{x}}$. The second application is in the discriminant analysis where the coefficients of the discriminant function are expressed as a product of the inverse sample covariance matrix and the difference of the  sample mean vectors.

\citet{BodnarOkhrin2011} derived the exact distribution of the product of the inverse sample covariance matrix and the sample mean vector under the assumption of normality, while \citet{KotsiubaMazur2015} obtained its asymptotic distribution as well as its approximate density based on the Gaussian integral and the third order Taylor series expansion. Moreover, \citet{BodnarMazurOkhrin2013,BodnarMazurOkhrin2014} analyzed the product of the sample (singular) covariance matrix and the sample mean vector. In the present paper, we contribute to the existing literature by deriving the central limit theorems (CLTs) under the introduced class of matrix-variate distributions in the case of the high-dimensional observation matrix. Under the considered family of distributions, the columns of the observation matrix are not independent anymore and, thus, the CLTs cover a more general class of random matrices.

Nowadays, modern scientific data include large number of sample points which is often comparable to the number of features (dimension) and so the sample covariance matrix and the sample mean vector are not the efficient estimators anymore. For example, stock markets include a large number of companies which is often close to the  number of available time points. In order to {understand better} the statistical properties of the traditional estimators and tests based on high-dimensional settings, it is of interest to study the asymptotic distribution of the above mentioned bilinear forms involving the sample covariance matrix and the sample mean vector.

The appropriate central limit theorems, which do not suffer from the ``curse of dimensionality'' and do not reduce the number of dimensions, are of  great interest for high-dimensional statistics because more efficient estimators and tests may be constructed and applied in practice. The classical multivariate procedures are based on the central limit theorems assuming that the dimension $p$ is fixed and the sample size $n$ increases. However, numerous authors provide quite reasonable proofs that this assumption does not lead to precise  distributional approximations for commonly used statistics, and that under increasing dimension asymptotics the better approximations can be obtained [see, e.g., \cite{baisil2004} and references therein]. Technically speaking, under the high-dimensional asymptotics we understand the case when the sample size $n$ and the dimension $p$ tend to infinity, such that their ratio $p/n$ converges to some positive constant $c$. Under this condition the well-known Mar\v{c}enko-Pastur and Silverstein equations were derived [see, \cite{marpas1967}, \cite{silverstein1995}].

The rest of the paper is structured as follows. In Section 2 we introduce a semi-parametric matrix-variate location mixture of normal distributions. Main results are given in Section 3, where we derive the central limit theorems under high-dimensional asymptotic regime of the (inverse) sample covariance matrix and the sample mean vector under the MVLMN. Section 4 presents a short numerical study in order to verify the obtained analytic results.


\section{Semi-parametric family of matrix-variate location mixture of normal distributions}

In this section we introduce the family of MVLMN which generalizes the existent families of skew normal distributions.

Let
\begin{eqnarray*}
\mathbf{X}
=
\left(
\begin{array}{ccc}
x_{11}&  \ldots   &  x_{1n}
\\
\vdots  &  \ddots  &  \vdots
\\
x_{p1}   &   \ldots   &x_{pn}
\end{array}
\right)
=
\left(\mathbf{x}_{1},...,\mathbf{x}_{n}\right),
\end{eqnarray*}
be the $p\times n$ observation matrix where $\mathbf{x}_{j}$ is the $j^{th}$ observation vector.
In the following, we assume that the random matrix $\mathbf{X}$ possesses a stochastic representation given by
\begin{eqnarray}
\label{model}
\mathbf{X}
\stackrel{d}{=}
\mathbf{Y}+\mathbf{B}\bnu\mathbf{1}_{n}^\top,
\end{eqnarray}
where $\mathbf{Y}\sim \mathcal{N}_{p,n}(\bmu\mathbf{1}_{n}^\top,\mathbf{\Sigma}\otimes\mathbf{I}_{n})$ ($p\times n$-dimensional matrix-variate normal distribution with mean matrix $\bmu\mathbf{1}_{n}^\top$ and covariance matrix $\mathbf{\Sigma}\otimes\mathbf{I}_{n}$), $\bnu$ is a $q$-dimensional random vector with continuous density function $f_{\bnu}(\cdot)$, $\mathbf{B}$ is a $p\times q$ matrix of constants.
Further, it is assumed that $\mathbf{Y}$ and $\bnu$ are independently distributed.
If random matrix $\mathbf{X}$ follows model (\ref{model}), then we say that $\mathbf{X}$ is MVLMN distributed with parameters $\bmu$, $\mathbf{\Sigma}$, $\mathbf B$, and $f_{\bnu}(\cdot)$.
The first three parameters are finite dimensional, while the fourth parameter is infinite dimensional. This makes model (\ref{model}) to be of a semi-parametric type.
The assertion we denote by $\mathbf{X}  \sim \mathcal{LMN}_{p,n;q} ( \bmu ,   \mathbf{\Sigma},  \mathbf B ;  f_{\bnu})$.
If $f_{\bnu}$ can be parametrized by finite dimensional parameter $\btheta$, then (\ref{model}) reduces to a parametrical model which is denoted by $\mathbf{X} \sim \mathcal{LMN}_{p,n;q}  ( \bmu,  \mathbf{\Sigma}, \mathbf B; \btheta)$.
If $n=1$, then we use the notation $\mathcal{LMN}_{p;q}(\cdot,\cdot,\cdot;\cdot)$ instead of $\mathcal{LMN}_{p,1;q}(\cdot,\cdot,\cdot;\cdot)$.

From (\ref{model}) the density function of $\mathbf{X}$ is expressed as
\begin{eqnarray}
\label{density of X}
f_{\mathbf{X}}(\mathbf{Z})
=
\int_{\mathbb{R}^{q}}
 f_{N_{p,n}(\bmu \mathbf 1_n^\top,\mathbf{\Sigma}\otimes\mathbf{I}_{n})}(\mathbf{Z}-\mathbf{B}\bnu^*\mathbf{1}_{n}^\top)
 f_{\bnu}(\bnu^*)
\mathbf{d\bnu}^*.
\end{eqnarray}
In a special case when $\bnu=|\bpsi|$ is the vector formed by the absolute values of every element in $\bpsi$ where $\bpsi\sim \mathcal{N}_q(\mathbf{0},\mathbf{\Omega})$, i.e. $\bnu$ has a $q$-variate truncated normal distribution, we get

\begin{proposition}\label{kapec}
Assume model (\ref{model}). Let $\bnu=|\bpsi|$ with $\bpsi\sim \mathcal{N}_q(\mathbf{0},\mathbf{\Omega})$. Then the density function of $\mathbf{X}$ is given by
\begin{eqnarray}
  f_{\mathbf{X}}(\mathbf{Z}) = \widetilde{C}^{-1}\Phi_{q}\left(\mathbf{0};-\mathbf{D}\mathbf{E}\mathbf{vec}(\mathbf{Z}-\bmu\mathbf{1}_n^\top),\mathbf{D}\right)
  \phi_{pn}\left(\mathbf{vec}(\mathbf{Z}-\bmu\mathbf{1}_n^\top);\mathbf{0},\mathbf{F}\right)
\end{eqnarray}
where
$\mathbf{D} = (n\mathbf{B}^\top\mathbf{\Sigma}^{-1}\mathbf{B}+\mathbf{\Omega}^{-1})^{-1}$,
$\mathbf{E}=\mathbf{1}_n^\top\otimes\mathbf{B}^\top\mathbf{\Sigma}^{-1}$, $\mathbf{F}=(\mathbf{I}_n\otimes\mathbf{\Sigma}^{-1}-\mathbf{E}^\top\mathbf{D}\mathbf{E})^{-1}$, and
$$\widetilde{C}^{-1}=C^{-1}\frac{|\mathbf{F}|^{1/2}|\mathbf{D}|^{1/2}}{|\mathbf{\Omega}|^{1/2}|\mathbf{\Sigma}|^{n/2}}
~~\text{with} ~~
C=\Phi_{q}(\mathbf{0};\mathbf{0},\mathbf{\Omega}).$$
\end{proposition}

The proof of Proposition \ref{kapec} is given in the Appendix.

It is remarkable that model (\ref{model}) includes several skew-normal distributions considered by \citet{AzzaliniDallaValle1996}, \citet{AzzaliniCapitanio1999}, \citet{Azzalini2005}.
For example, in case of $n=1$, $q=1$, $\bmu=\mathbf{0}$, $\mathbf{B}=\mathbf{\Delta} \mathbf{1}_p$, and $\mathbf{\Sigma}=( \mathbf{I}_p - \mathbf{\Delta} ^2 ) ^{1/2}\mathbf{\Psi}( \mathbf{I}_p - \mathbf{\Delta} ^2 ) ^{1/2}$ we get
\begin{eqnarray}
\label{spcase}
\mathbf{X}
\stackrel{d}{=}
( \mathbf{I}_p - \mathbf{\Delta} ^2 ) ^{1/2}   \mathbf{v}_0  +  \mathbf{\Delta} \mathbf{1}_p  |v_1|,
\end{eqnarray}
where $\mathbf{v}_0\sim \mathcal{N}_p(\mathbf{0},\mathbf{\Psi})$ and $v_1\sim \mathcal{N}(0,1)$ are independently distributed; $\mathbf{\Psi}$ is a correlation matrix and $\mathbf{\Delta}=diag(\delta_1,...,\delta_p)$ with $\delta_j\in (-1,1)$.
Model (\ref{spcase}) was previously introduced by \citet{Azzalini2005}.

Moreover, model (\ref{model}) also extends the classical random coefficient growth-curve model (see, \cite{potthoff1964}, \cite{amemiya1994} among others), i.e., the columns of $\bX$ can be rewritten in the following way
  \begin{eqnarray}
    \bx_i\overset{d}{=}\bmu+\bB\bnu+\boldsymbol{\varepsilon}_i,~~i=1,\ldots,n\,,
  \end{eqnarray}
where $\boldsymbol{\varepsilon}_1,\ldots,\boldsymbol{\varepsilon}_n$ are i.i.d. $\mathcal{N}_p(\bzero, \bSigma)$, $\bnu\sim f_{\bnu}$ and $\boldsymbol{\varepsilon}_1,\ldots,\boldsymbol{\varepsilon}_n$, $\bnu$ are independent. In the random coefficient growth-curve model, it is typically assumed that $\boldsymbol{\varepsilon}_i\sim\mathcal{N}_p(\bzero, \sigma^2\bI)$ and $\bnu\sim\mathcal{N}_q(\bzero, \bol{\Omega})$ (see, e.g., \cite{rao1965}). As a result, the suggested MVLMN model may be useful for studying the robustness of the random coefficients against the non-normality.


\section{CLTs for expressions involving the sample covariance matrix and the sample mean vector}

The sample estimators for the mean vector and the covariance matrix are given by
\begin{eqnarray*}
\overline{\mathbf{x}}
=
\frac{1}{n}  \sum_{i=1}^{n}  \mathbf{x}_{i}
=
\frac{1}{n}  \mathbf{X}  \mathbf{1}_{n}
\ \ \ \  \mathrm{and}\ \ \ \ \
\mathbf{S}
=
\frac{1}{n-1} \sum_{i=1}^{n} (\mathbf{x}_{i}-\overline{\mathbf{x}}) (\mathbf{x}_{i}-\overline{\mathbf{x}})^\top
=\mathbf{X} \mathbf{VX} ^\top,
\end{eqnarray*}
where $\mathbf{V}=\mathbf{I}_{n}-\frac{1}{n}\mathbf{1}_{n}\mathbf{1}_{n}^\top$ is a symmetric idempotent matrix, i.e., $\mathbf{V}=\mathbf{V}^\top$ and $\mathbf{V}^2=\mathbf{V}$.

The following proposition shows that $\overline{\mathbf{x}}$ and $\mathbf{S}$ are independently distributed and presents their marginal distributions under model (\ref{model}).
Moreover, its results lead to the conclusion that the independence of $\overline{\mathbf{x}}$ and $\mathbf{S}$ could not be used as a characterization property of a multivariate normal distribution if the observation vectors in data matrix are dependent.

\begin{proposition}
\label{th1}
Let $\mathbf{X}\sim \mathcal{LMN}_{p,n;q} (\bmu,\mathbf{\Sigma} , \mathbf B ;  f_{\bnu} )$.
Then
\begin{enumerate}[(a)]
\item
$(n-1)\mathbf{S} \sim \mathcal{W}_{p}(n-1,\mathbf{\Sigma})$ ($p$-dimensional Wishart distribution for $p \le n-1$ and $p$-dimensional singular Wishart distribution for $p >n-1$ with $(n-1)$ degrees of freedom and covariance matrix $\mathbf{\Sigma}$),
\item
$\overline{\mathbf{x}}\sim \mathcal{LMN}_{p;q} \left( \bmu, \frac{1}{n}\mathbf{\Sigma}, \mathbf B; f_{\bnu} \right)$,
\item
$\mathbf{S}$ and $\overline{\mathbf{x}}$ are independently distributed.
\end{enumerate}
\end{proposition}


\begin{proof}
The statements of the proposition follow immediately from the fact that
    \begin{eqnarray}\label{proof of theorem 1}
      \bar{\bx}=\frac{1}{n}\sum\limits_{i=1}^n\bx_i=\bar{\by}+\bB\bnu~~\text{and}~~\bS=\bX\bV\bX^\top=\bY\bV\bY^\top\,.
    \end{eqnarray}
Indeed, from (\ref{proof of theorem 1}), the multivariate normality of $\bY$ and the independence of $\bY$ and $\bnu$, we get that $\overline{\mathbf{x}}$ and $\mathbf{S}$ are independent; $\mathbf{S}$ is (singular) Wishart distributed with $(n-1)$ degrees of freedom and covariance matrix $\mathbf{\Sigma}$; $\overline{\mathbf{x}}$ has a location mixture of normal distributions with parameters $\bmu$, $\frac{1}{n}\mathbf{\Sigma}$, $\mathbf B$ and $f_{\bnu}$.
\end{proof}

For the validity of the asymptotic results presented in Sections 3.1 and 3.2 we need the following two conditions

\begin{itemize}
  \item[(A1)] Let $(\lambda_i,\mathbf{u}_i)$ denote the set of eigenvalues and eigenvectors of $\mathbf{\Sigma}$. We assume that there exist $m_1$ and $M_1$ such that
  \[0<m_1\le \lambda_1 \le \lambda_2 \le ... \le \lambda_p \le M_1<\infty\]
  uniformly in $p$.
  \item[(A2)] There exists $M_2<\infty$ such that
  \[|\mathbf{u}_i^\top \bmu| \le M_2 ~~ \text{and} ~~ |\mathbf{u}_i^\top \mathbf{b}_j| \le M_2 ~~\text{for all}~~i=1,...,p ~~\text{and}~~j=1,...,q \]
  uniformly in $p$ where $\mathbf{b}_j$, $j=1,...,q$, are the columns of $\mathbf{B}$.
\end{itemize}
Generally, we say that an arbitrary $p$-dimensional vector $\bl$ satisfies the condition (A2) if $|\mathbf{u}_i^\top \bl| \le M_2< \infty$ for all $i=1,\ldots,p$.

Assumption $\mathbf{(A1)}$ is a classical condition in random matrix theory (see, \citet{baisil2004}), which bounds the spectrum of $\bSigma$ from below as well as from above. Assumption $\mathbf{(A2)}$ is a technical one. In combination with $\mathbf{(A1)}$ this condition ensures that $p^{-1}\bmu^\top \bSigma \bmu$, $p^{-1}\bmu^\top \bSigma^{-1}\bmu$, $p^{-1}\bmu^\top \bSigma^3 \bmu$, $p^{-1}\bmu^\top \bSigma^{-3} \bmu$, as well as that all the diagonal elements of $\mathbf{B}^\top \bSigma \mathbf{B}$, $\mathbf{B}^\top \bSigma^3 \mathbf{B}$, and $\mathbf{B}^\top \bSigma^{-1} \mathbf{B}$ are uniformly bounded. All these quadratic forms are used in the statements and the proofs of our results.  Note that the constants appearing in the inequalities will be denoted by $M_2$ and may vary from one expression to another.
%
  We further note, that assumption (A2) is automatically fulfilled if $\bmu$ and $\bB$ are sparse (not $\bSigma$ itself). More precisely, a stronger condition, which verifies (A2) is $||\bmu||<\infty$ and $||\mathbf{b}_j||<\infty$ for $j=1,\ldots, q$ uniformly in $p$. Indeed, using the Cauchy-Schwarz inequality to $(\bu_i^\top\bmu)^2$ we get
  \begin{eqnarray}
(\mathbf{u}_i^\top \bmu)^2 \leq ||\bu_i||^2 ||\bmu||^2=||\bmu||^2<\infty\,.
  \end{eqnarray}
It is remarkable that $||\bmu||<\infty$ is fulfilled if $\bmu$ is indeed a sparse vector or if all its elements are of order $O\left(p^{-1/2}\right)$. Thus, a sufficient condition for (A2) to hold would be either sparsity of $\bmu$ and $\bB$ or their elements are reasonably small but not exactly equal to zero. Note that the assumption of sparsity for $\bB$ is quite natural in the context of high-dimensional random coefficients regression models. It implies that if $\bB$ is large dimensional, then it may have many zeros and there exists a small set of highly significant random coefficients which drive the random effects in the model (see, e.g., \cite{bhlmann2011}). Moreover, it is hard to estimate $\bmu$ in a reasonable way when $||\bmu||\to\infty$ as $p \to \infty$ (see, e.g., \cite{BodnarOkhrinParolya2016}). In general, however, $||\bmu||$ and $||\mathbf{b}_j||$ do not need to be bounded, in such case the sparsity of eigenvectors of matrix $\bSigma$ may guarantee the validity of (A2). Consequently, depending on the properties of the given data set the proposed model may cover many practical problems.


\subsection{CLT for the product of sample covariance matrix and sample mean vector}

In this section we present the central limit theorem for the product of the sample covariance matrix and the sample mean vector.

\begin{theorem}
\label{th2}
Assume $\mathbf X\sim  \mathcal{LMN}_{p,n;q} (\bmu,\mathbf{\Sigma} , \mathbf B ;  f_{\bnu} )$ with $\bSigma$ positive definite and let $p/n=c +o(n^{-1/2})$, $c\in[0, +\infty)$ as $n\to\infty$. Let $\bl$ be a $p$-dimensional vector of constants that satisfies condition $\mathbf{(A2)}$. Then, under $\mathbf{(A1)}$ and $\mathbf{(A2)}$ it holds that
\begin{eqnarray}
\label{th2eq1}
\sqrt{n}\sigma^{-1}_{\bnu}\left(\bl^\top\bS\overline{\bx}-\bl^\top\bSigma\bmu_{\bnu}\right)  \overset{\mathcal{D}}{\longrightarrow} \mathcal{N}(0,1)~~\text{for}~p/n\to c\in[0, +\infty) ~~\text{as}~~n\to\infty\,,
\end{eqnarray}

\vspace{5mm}

where

\vspace{-15mm}

\begin{eqnarray}
 \bmu_{\bnu}& =& \bmu+\bB\bnu,\\[2mm]
\sigma^2_{\bnu}&=&\left[\bmu^\top_{\bnu}\bSigma\bmu_{\bnu}+ c\frac{\text{tr}(\bSigma^2)}{p}\right]\bl^\top\bSigma\bl+(\bl^\top\bSigma\bmu_{\bnu})^2 +\bl^\top \bSigma^3 \bl\,.
\end{eqnarray}
\end{theorem}

\begin{proof}
First, we consider the case of $p \le n-1$, i.e., $\mathbf{S}$ has a Wishart distribution. Let $\mathbf L = (\bl,\overline{\mathbf{x}})^\top$ and define
$\widetilde{\mathbf{S}}=\mathbf L \mathbf{S}\mathbf L^\top = \{\widetilde{\mathbf{S}}_{ij}\}_{i,j=1,2}$
with
$\widetilde{S}_{11}=\bl^\top\mathbf{S}\bl$,
$\widetilde{S}_{12}=\bl^\top\mathbf{S}\overline{\mathbf{x}}$,
$\widetilde{S}_{21}=\overline{\mathbf{x}}^\top\mathbf{S}\bl$,
and
$\widetilde{S}_{22}=\overline{\mathbf{x}}^\top\mathbf{S}\overline{\mathbf{x}}$.
Similarly, let
$\widetilde{\mathbf{\Sigma}}=\mathbf{L}\mathbf{\Sigma}\mathbf{L}^\top=\{\widetilde{\mathbf{\Sigma}}_{ij}\}_{i,j=1,2}$ with
$\widetilde{\Sigma}_{11}=\bl^\top\mathbf{\Sigma}\bl$,
$\widetilde{\Sigma}_{12}=\bl^\top\mathbf{\Sigma}\overline{\mathbf{x}}$, $\widetilde{\Sigma}_{21}=\overline{\mathbf{x}}^\top \mathbf{\Sigma}\bl$, and $\widetilde{\Sigma}_{22}=\overline{\mathbf{x}}^\top\mathbf{\Sigma}\overline{\mathbf{x}}$.

Since $\bS$ and $\overline{\mathbf{x}}$ are independently distributed, $\mathbf{S}\sim \mathcal{W}_{p}\left(n-1,{\dfrac{1}{n-1}}\mathbf{\Sigma}\right)$ and $rank~\mathbf{L}=2\leq p$ with probability one, we get from Theorem 3.2.5 of \citet{Muirhead1982} that $ \widetilde{\mathbf{S}}|\overline{\mathbf{x}} \sim \mathcal{W}_{2}\left(n-1,{\dfrac{1}{n-1}}\widetilde \bSigma\right)$. As a result, the application of Theorem 3.2.10 of \citet{Muirhead1982} leads to
\begin{eqnarray*}
\widetilde{S}_{12}|\widetilde{S}_{22},\overline{\mathbf{x}}
\sim
 \mathcal{N}\left(\widetilde{\Sigma}_{12}\widetilde{\Sigma}_{22}^{-1}\widetilde{S}_{22},{\dfrac{1}{n-1}}\widetilde{\Sigma}_{11\cdot 2}\widetilde{S}_{22}\right),
\end{eqnarray*}
where $\widetilde{\Sigma}_{11\cdot2} = \widetilde{\Sigma}_{11} - \widetilde{\Sigma}_{12}^2/ \widetilde{\Sigma}_{22}$ is the Schur complement.

Let $\xi=(n-1)\widetilde{S}_{22}/\widetilde{\Sigma}_{22}$, then
\begin{eqnarray*}
\bl^\top\mathbf{S}\overline{\mathbf{x}}|\xi,\overline{\mathbf{x}}\sim \mathcal{N}\left({\dfrac{\xi}{n-1}}\bl^\top\mathbf{\Sigma}\overline{\mathbf{x}},{\dfrac{\xi}{(n-1)^2}}[\overline{\mathbf{x}}^\top\mathbf{\Sigma}\overline{\mathbf{x}}\bl^\top\mathbf{\Sigma }\bl -(\overline{\mathbf{x}}^\top\mathbf{\Sigma}\bl)^2]\right).
\end{eqnarray*}

From Theorem 3.2.8 of \citet{Muirhead1982} it follows that $\xi$ and $\overline{\mathbf{x}}$ are independently distributed and $\xi \sim \chi^2 _{n-1}$.
Hence, the stochastic representation of $\bl^\top\mathbf{S}\overline{\mathbf{x}}$ is given by
\begin{eqnarray}\label{stoch1}
\bl^\top\mathbf{S}\overline{\mathbf{x}}
\stackrel{d}{=}
{\dfrac{\xi}{n-1}}\bl^\top\mathbf{\Sigma}\overline{\mathbf{x}}
+
{\sqrt{\dfrac{\xi}{n-1}}}
(\overline{\mathbf{x}}^\top\mathbf{\Sigma}\overline{\mathbf{x}}\bl^\top\mathbf{\Sigma}\bl-(\bl^\top\mathbf{\Sigma}\overline{\mathbf{x}})^2)^{1/2}
{\dfrac{z_{0}}{\sqrt{n-1}}},
\end{eqnarray}
where
$\xi \sim  \chi^{2}_{{n-1}}$,
$ z_0\sim \mathcal{N} (0, 1)$,
$\overline{\mathbf{x}} \sim  \mathcal{LMN}_{p;q}\left(\bmu,\frac{1}{n}\mathbf{\Sigma}, \mathbf B ;  f_{\bnu}\right)$;
 $\xi$, $z_{0}$ and $\overline{\mathbf{x}}$ are mutually independent. The symbol ''$\stackrel{d}{=}$'' stands for the equality in distribution.
 
It is remarkable that the stochastic representation \eqref{stoch1} for $\bl^\top\mathbf{S}\overline{\mathbf{x}}$ remains valid also in the case of $p>n-1$ following the proof of Theorem 4 in \citet{BodnarMazurOkhrin2014}. Hence, we will make no difference between these two cases in the remaining part of the proof.

From the properties of $\chi^2$-distribution and using the fact that $n/(n-1) \to 1$ as $n \to \infty$, we immediately receive
\begin{eqnarray}\label{chi_prop}
  \sqrt{n}\left(\frac{\xi}{n}-1\right)& \overset{\mathcal{D}}{\longrightarrow}& \mathcal{N}(0, 2) ~~\text{as}~~ n\to\infty\,.
\end{eqnarray}
We further get that $\sqrt{n}(z_0/\sqrt{n})\sim \mathcal{N}(0, 1)$ for all $n$ and, consequently, it is also its asymptotic distribution.

Next, we show that $\bl^\top\mathbf{\Sigma}\overline{\mathbf{x}}$ and $\overline{\mathbf{x}}^\top\mathbf{\Sigma}\overline{\mathbf{x}}$ are jointly asymptotically normally distributed given $\bnu$. For any $a_1$ and $a_2$, we consider
\begin{eqnarray*}
&&a_1\overline{\mathbf{x}}^\top\mathbf{\Sigma}\overline{\mathbf{x}}+2a_2\bl^\top\mathbf{\Sigma}\overline{\mathbf{x}}
= a_1 \left(\overline{\mathbf{x}}+\dfrac{a_2}{a_1}\bl\right)^\top\mathbf{\Sigma}\left(\overline{\mathbf{x}}+\dfrac{a_2}{a_1}\bl\right)
-\dfrac{a_2^2}{a_1}\bl^\top\mathbf{\Sigma}\bl=a_1 \tilde{\mathbf{x}}^\top\mathbf{\Sigma}\tilde{\mathbf{x}}-\dfrac{a_2^2}{a_1}\bl^\top\mathbf{\Sigma}\bl \,,
\end{eqnarray*}
where $\tilde{\mathbf{x}}=\bar{\mathbf{x}}+\frac{a_2}{a_1}\mathbf{1}$ and $\tilde{\mathbf{x}}|\bnu \sim \mathcal{N}_p \left(\bmu_{\mathbf{a},\bnu},\frac{1}{n}\mathbf{\Sigma}\right)$ with $\bmu_{\mathbf{a},\bnu}=\bmu+\mathbf{B}\bnu+\dfrac{a_2}{a_1}\bl$.
By \cite{provost1996} the stochastic representation of $\tilde{\mathbf{x}}^\top\mathbf{\Sigma}\tilde{\mathbf{x}}$ is given by
  \begin{eqnarray*}
    \tilde{\mathbf{x}}^\top\mathbf{\Sigma}\tilde{\mathbf{x}} \overset{d}{=} \frac{1}{n}\sum\limits_{i=1}^p\lambda^2_i\xi_i\,,
  \end{eqnarray*}
where $\xi_1,...,\xi_p$ given $\bnu$ are independent with $\xi_i|\bnu \sim \chi^2_1(\delta^2_i)$, $\delta_i=\sqrt{n}\lambda^{-1/2}_i\bu_i^\top\bmu_{\mathbf{a},\bnu}$. Here, the symbol $\chi^2_d(\delta^2_i)$ denotes a chi-squared distribution with $d$ degrees of freedom and non-centrality parameter $\delta^2_i$.

Now, we apply the Lindeberg CLT to the conditionally independent random variables $V_i=\lambda^2_i\xi_i/n$. For that reason, we need first to verify the Lindeberg's condition. Denoting $\sigma^2_n=\mathbbm{V}(\sum_{i=1}^p V_i|\bnu)$ we get
\begin{eqnarray}\label{variance_clt}
  \sigma^2_n = \sum\limits_{i=1}^p \mathbbm{V}\left[\left.\frac{\lambda^2_i}{n}\xi_i\right|\bnu\right]=\sum\limits_{i=1}^p\frac{\lambda_i^4}{n^2}2(1+2\delta_i^2)= \frac{1}{n^2}\left(2\text{tr}(\bSigma^4) + 4n\bmu_{\mathbf{a},\bnu}^\prime\bSigma^{3}\bmu_{\mathbf{a},\bnu} \right)
\end{eqnarray}
We need to check if for any small $\varepsilon>0$ it holds that
\begin{eqnarray*}
  \lim_{{n\to\infty}}\frac{1}{\sigma^2_n}\sum\limits_{i=1}^p\mathbbm{E}\left[\left.\left(V_i-\mathbbm{E}(V_i)\right)^2\mathbbm{1}_{\{|V_i-\mathbbm{E}(V_i)|>\varepsilon\sigma_n \}}\right|\bnu\right] \longrightarrow0\,.
\end{eqnarray*}
First, we get
\begin{eqnarray*}
 &&\sum_{i=1}^p\mathbbm{E}\left[\left.\left(V_i-\mathbbm{E}(V_i|\bnu)\right)^2\mathbbm{1}_{\{|V_i-\mathbbm{E}(V_i|\bnu)|>\varepsilon\sigma_n \}}\right|\bnu\right]\\
 & \overset{Cauchy-Schwarz}{\leq}&
 \sum_{i=1}^p\mathbbm{E}^{1/2}\left[\left.\left(V_i-\mathbbm{E}(V_i|\bnu)\right)^4\right|\bnu\right]
 \mathbbm{P}^{1/2}\{\left.|V_i-\mathbbm{E}(V_i|\bnu)|>\varepsilon\sigma_n\right|\bnu\}\\
&\overset{Chebyshev}{\leq}& \sum_{i=1}^p\frac{\lambda_i^4}{n^2} \sqrt{12(1+2\delta^2_i)^2+48(1+4\delta_i^2)}\frac{\sigma_i}{\varepsilon\sigma_n}
\end{eqnarray*}
with $\sigma^2_i=\mathbbm{V}(V_i|\bnu)$ and, thus,
\begin{eqnarray*}
 && \frac{1}{\sigma^2_n}\sum\limits_{i=1}^p\mathbbm{E}\left[\left.\left(V_i-\mathbbm{E}(V_i|\bnu)\right)^2
 \mathbbm{1}_{\{|V_i-\mathbbm{E}(V_i|\bnu)|>\varepsilon\sigma_n \}}\right|\bnu\right]\\
 & \leq& \frac{1}{\varepsilon}\frac{\sum_{i=1}^p\lambda_i^4 \sqrt{12(1+2\delta^2_i)^2+48(1+4\delta_i^2)}\frac{\sigma_i}{\sigma_n}}{2\text{tr}(\bSigma^4) + 4n\bmu_{\mathbf{a},\bnu}^\top\bSigma^{3}\bmu_{\mathbf{a},\bnu}}\\
&=& \frac{\sqrt{3}}{\varepsilon}\frac{\sum_{i=1}^p\lambda_i^4 \sqrt{(5+2\delta^2_i)^2-20}\frac{\sigma_i}{\sigma_n}}{\text{tr}(\bSigma^4) + 2n\bmu_{\mathbf{a},\bnu}^\top\bSigma^{3}\bmu_{\mathbf{a},\bnu}   }\\
&\leq&\frac{\sqrt{3}}{\varepsilon}\frac{\sum_{i=1}^p\lambda_i^4 (5+2\delta^2_i)\frac{\sigma_i}{\sigma_n}}{\text{tr}(\bSigma^4) + 2n\bmu_{\mathbf{a},\bnu}^\top\bSigma^{3}\bmu_{\mathbf{a},\bnu}}\\
&\le& \frac{\sqrt{3}}{\varepsilon}\frac{5\text{tr}(\bSigma^4)+2n\bmu_{\mathbf{a},\bnu}^\top\bSigma^3\bmu_{\mathbf{a},\bnu}}{\text{tr}(\bSigma^4) + 2n\bmu_{\mathbf{a},\bnu}^\top\bSigma^{3}\bmu_{\mathbf{a},\bnu}}\frac{\sigma_{max}}{\sigma_n}\\
&\le& \frac{\sqrt{3}}{\varepsilon}\left(\frac{4}{1+ 2n\bmu_{\mathbf{a},\bnu}^\top\bSigma^{3}\bmu_{\mathbf{a},\bnu}/\text{tr}(\bSigma^4)}+1\right)\frac{\sigma_{max}}{\sigma_n}\\
&\le& \frac{5\sqrt{3}}{\varepsilon}\frac{\sigma_{max}}{\sigma_n}\,.
\end{eqnarray*}
Finally, Assumptions $\mathbf{(A1)}$ and $\mathbf{(A2)}$ yield
\begin{eqnarray}\label{lindeberg}
  \frac{\sigma^2_{max}}{\sigma^2_n} = \frac{\sup_i\sigma^2_i}{\sigma_n^2} = \frac{\sup_i\lambda_i^4(1+2\delta_i^2)}{\text{tr}(\bSigma^4) + 2n\bmu_{\mathbf{a},\bnu}^\top\bSigma^{3}\bmu_{\mathbf{a},\bnu}}=
  \frac{\sup_i \lambda_i^4+2n\lambda_i^3(\bu_i^\top\bmu_{\mathbf{a},\bnu})^2}
  {\text{tr}(\bSigma^4) + 2n\bmu_{\mathbf{a},\bnu}^\top\bSigma^{3}\bmu_{\mathbf{a},\bnu}}\longrightarrow0\,,
\end{eqnarray}
which verifies the Lindeberg condition since
\begin{eqnarray}
(\bu_i^\top\bmu_{\mathbf{a},\bnu})^2&=&\left(\bu_i^\top\bmu+\bu_i^\top\bB\bnu+\bu_i^\top\bl\frac{a_2}{a_1}\right)^2\nonumber\\
&=& (\bu_i^\prime\bmu)^2+\left( \bu_i^\top\bl\frac{a_2}{a_1}\right)^2 +(\bu_i^\top\bB\bnu)^2 +2\bu_i^\top\bmu\cdot\bu_i^\top\bB\bnu +2\frac{a_2}{a_1}\bu_i^\top\bl(\bu_i^\top\bmu+\bu_i^\top\bB\bnu)    \nonumber  \\
&\overset{(A2)}{\le}& M^2_2+qM_2^2\bnu^{\top}\bnu+  M^2_2\frac{a_2^2}{a^2_1}+ 2M^2_2\sqrt{q\bnu^{\top}\bnu}+  2M^2_2\frac{|a_2|}{|a_1|}(1+\sqrt{q\bnu^\top\bnu})\nonumber\\
&=& M_2^2\left(1+\sqrt{q\bnu^{\top}\bnu}+\frac{|a_2|}{|a_1|} \right)^2 < \infty.
\end{eqnarray}

Thus, using \eqref{variance_clt} and
\begin{eqnarray}\label{mean_clt}
\sum_{i=1}^p\mathbbm{E}(V_i|\bnu)=\sum_{i=1}^p\frac{\lambda^2_i}{n}(1+\delta^2_i) = \text{tr}(\bSigma^2)/n + \bmu_{\mathbf{a},\bnu}^\top\bSigma\bmu_{\mathbf{a},\bnu}
\end{eqnarray}
we get that
\begin{eqnarray*}
\sqrt{n} \frac{ \tilde{\mathbf{x}}^\top\mathbf{\Sigma}\tilde{\mathbf{x}} - \text{tr}(\bSigma^2)/n -\bmu_{\mathbf{a},\bnu}^\top\bSigma\bmu_{\mathbf{a},\bnu}}{\sqrt{\text{tr}(\bSigma^4)/n + 2\bmu_{\mathbf{a},\bnu}^\prime\bSigma^{3}\bmu_{\mathbf{a},\bnu}}} \Bigg|\bnu \overset{\mathcal{D}}{\longrightarrow} \mathcal{N}(0, 2)
\end{eqnarray*}
and for $a_1\overline{\mathbf{x}}^\top\mathbf{\Sigma}\overline{\mathbf{x}}+2a_2\bl^\top\mathbf{\Sigma}\overline{\mathbf{x}}$ we have
\begin{eqnarray*}
 \sqrt{n}\frac{ a_1\overline{\mathbf{x}}^\top\mathbf{\Sigma}\overline{\mathbf{x}}+2a_2\bl^\top\mathbf{\Sigma}\overline{\mathbf{x}}-a_1\left(\text{tr}(\bSigma^2)/n +\bmu_{\mathbf{a},\bnu}^\top\bSigma\bmu_{\mathbf{a},\bnu}\right)+\dfrac{a_2^2}{a_1}\bl^\top\mathbf{\Sigma}\bl}{ \sqrt{a_1^2\left(\text{tr}(\bSigma^4)/n + 2\bmu_{\mathbf{a},\bnu}^\prime\bSigma^{3}\bmu_{\mathbf{a},\bnu}\right)}}\Bigg|\bnu \overset{\mathcal{D}}{\longrightarrow} \mathcal{N}(0, 2)\,.
\end{eqnarray*}
Denoting $\mathbf{a}=(a_1, 2a_2)^\top$ and $\bmu_{\bnu}=\bmu+\bB\bnu$ we can rewrite it as
{\footnotesize
\begin{eqnarray}\label{mult_norm3}
\sqrt{n}\left[\mathbf{a}^\top
\left(
\begin{array}{c}
   \overline{\mathbf{x}}^\top\mathbf{\Sigma}\overline{\mathbf{x}} \\
   \bl^\top\mathbf{\Sigma}\overline{\mathbf{x}}
  \end{array}
\right) -
 \mathbf{a}^\top
\left(
\begin{array}{c}
   \bmu_{\bnu}^\top\mathbf{\Sigma}\bmu_{\bnu} + c\frac{\text{tr}(\bSigma^2)}{p}\\
   \bl^\top\mathbf{\Sigma}\bmu_{\bnu}
  \end{array}
\right)\right] \Bigg|\bnu \overset{\mathcal{D}}{\longrightarrow} \mathcal{N}\left(\mathbf{0}, \mathbf{a}^\top\left(
  \begin{array}{cc}
    2c\frac{\text{tr}(\bSigma^4)}{p}+4\bmu_{\bnu}^\top\mathbf{\Sigma}^3\bmu_{\bnu} &  2 \bl^\top\mathbf{\Sigma}^3\bmu_{\bnu}\\
 2 \bl^\top\mathbf{\Sigma}^3\bmu_{\bnu} & \bl^\top\bSigma^3\bl
  \end{array}
\right)    \mathbf{a}   \right)
\end{eqnarray}
}
which implies that the vector $\sqrt{n}\left(\overline{\mathbf{x}}^\top\mathbf{\Sigma}\overline{\mathbf{x}} -\bmu_{\bnu}^\top\mathbf{\Sigma}\bmu_{\bnu} - c\frac{\text{tr}(\bSigma^2)}{p},~  \bl^\top\mathbf{\Sigma}\overline{\mathbf{x}}-  \bl^\top\mathbf{\Sigma}\bmu_{\bnu}\right)^\top$ has asymptotically multivariate normal distribution conditionally on $\bnu$ because the vector $\mathbf{a}$ is arbitrary.

Taking into account \eqref{mult_norm3}, \eqref{chi_prop} and the fact that $\xi$, $z_0$ and $\overline{\bx}$ are mutually independent we get 
{\footnotesize
\begin{eqnarray*}
\sqrt{n}\left.\left[
\left(
\begin{array}{c}
  \frac{\xi}{n} \\
   \overline{\mathbf{x}}^\top\mathbf{\Sigma}\overline{\mathbf{x}} \\
   \bl^\top\mathbf{\Sigma}\overline{\mathbf{x}}\\
  \frac{z_0}{\sqrt{n}}
\end{array}
\right)
-
\left(
\begin{array}{c}
  1 \\
   \bmu_{\bnu}^{\top}\mathbf{\Sigma}\bmu_{\bnu} + c\frac{\text{tr}(\bSigma^2)}{p}\\
   \bl^\top\mathbf{\Sigma}\bmu_{\bnu}\\
 0
\end{array}
\right)
\right] \right|\bnu
 \overset{\mathcal{D}}{\longrightarrow} \mathcal{N}\left(\mathbf{0}, \left(
  \begin{array}{cccc}
2 & 0 & 0 & 0\\
 0 &   2c\frac{\text{tr}(\bSigma^4)}{p}+4\bmu_{\bnu}^\top\mathbf{\Sigma}^3\bmu_{\bnu} &  2 \bl^\top\mathbf{\Sigma}^3\bmu_{\bnu} & 0\\
 0& 2 \bl^\top\mathbf{\Sigma}^3\bmu_{\bnu} & \bl^\top\bSigma^3\bl &0\\
0 &0 &0 & 1\,
  \end{array}
\right)  \right)\,.
\end{eqnarray*}
}
The application of the multivariate delta method leads to
\begin{eqnarray}\label{clt1}
  \sqrt{n}\sigma^{-1}_{\bnu}\left(\bl^\top\bS\overline{\bx} -  \bl^\top\bSigma\bmu_{\bnu}\right)\Big|\bnu\overset{\mathcal{D}}{\longrightarrow}\mathcal{N}\left(0,1\right)
\end{eqnarray}
where
\begin{eqnarray*}
  \sigma^2_{\bnu} = (\bl^\top\bSigma\bmu_{\bnu})^2 +\bl^\top \bSigma^3 \bl +\bl^\top\bSigma\bl\left[ \bmu_{\bnu}^\top\mathbf{\Sigma}\bmu_{\bnu} + c\frac{\text{tr}(\bSigma^2)}{p}\right]
\end{eqnarray*}
The asymptotic distribution does not depend on $\bnu$ and, thus, it is also the unconditional asymptotic distribution.
\end{proof}

Theorem \ref{th2} shows that properly normalized bilinear form $\bl^\top\bS\overline{\bx}$ itself can be accurately approximated by a mixture of normal distributions with both mean and variance depending on $\bnu$. Moreover, this central limit theorem delivers the following approximation for the distribution of $\bl^\top\bS\overline{\bx}$, namely for large $n$ and $p$ we have
\begin{eqnarray}
  {p^{-1}\bl^\top\bS\overline{\bx}}|\bnu \approx \mathcal{CN}\left(p^{-1}\bl^\top\bSigma\bmu_{\bnu}, \frac{p^{-2}\sigma^2_{\bnu}}{n}\right)\,,
\end{eqnarray}
i.e., it has a compound normal distribution with random mean and variance.

 The proof of Theorem \ref{th2} shows, in particular, that its key point is a stochastic representation of the product $\bl^\top\mathbf{S}\overline{\mathbf{x}}$ which can be presented using a $\chi^{2}$ distributed random variable, a standard normally distributed random variable, and a random vector which has a location mixture of normal distributions. Both assumptions (A1) and (A2) guarantee that the asymptotic mean $p^{-1}\bl^\top\bSigma\bmu_{\bnu}$ and the asymptotic variance $p^{-2}\sigma^2_{\bnu}$ are bounded with probability one and the covariance matrix $\bSigma$ is invertible as the dimension $p$ increases. Note that the case of standard asymptotics can be easily recovered from our result if we set $c\to0$.

Although Theorem 1 presents a result similar to the classical central limit theorem, it is not a proper CLT because of $\bnu$ which is random and, hence, it provides an asymptotic equivalence only. In the current settings it is not possible to provide a proper CLT without additional assumptions imposed on $\bnu$. The reason is the finite dimensionality of $\bnu$ denoted by $q$ which is fixed and independent of $p$ and $n$. This assures that the randomness in $\bnu$ will not vanish asymptotically. So, in order to justify a classical CLT we need to have an increasing value of $q=q(n)$ with the following additional assumptions:
\begin{itemize}
\item[(A3)] Let $\mathbbm{E}(\bnu)=\bomega$ and $\mathbbm{Cov}(\bnu)=\bOmega$. For any vector $\bl$ it holds that
      \begin{eqnarray}
      \sqrt{q}\frac{\bl^\top\bnu - \bl^\top\bomega}{\sqrt{\bl^\top\bOmega\bl}}\overset{\mathcal{D}}{\longrightarrow}  \tilde{z}_0\,
      \end{eqnarray}
where the distribution of $\tilde{z}_0$ may in general depend on $\bl$.
\item[(A4)] There exists $M_2<\infty$ such that
  \[|\mathbf{u}_i^\top \bB \bomega| \le M_2 ~~ \text{for all}~~i=1,...,p~~\text{uniformly in $p$.}  \]
  Moreover, $\bB\bOmega\bB^\top$ has an uniformly bounded spectral norm.
\end{itemize}
Assumption $\mathbf{(A3)}$ assures that the vector of random coefficients $\bnu$ satisfies itself a sort of concentration property in the sense that any linear combination of its elements asymptotically concentrates around some random variable. In a special case of $\tilde{z}_0$ being standard normally distributed (A3) will imply a usual CLT for a linear combination of vector $\bnu$.
As a result, assumption $\mathbf{(A3)}$ is a pretty general and natural condition on $\bnu$ taking into account that it is assumed $\bnu\sim\mathcal{N}_q(\bomega, \bOmega)$ (see, e.g., \cite{rao1965}) in many practical situations. Assumption $\mathbf{(A4)}$ is similar to $\mathbf{(A1)}$ and  $\mathbf{(A2)}$ and it ensures that $\bmu$ and $\bB\bomega$ as well as $\bSigma$ and $\bB\bOmega \bB^\top $ have the same behaviour as $p \to \infty$.

\begin{corollary}[CLT]\label{cor1}
   Under the assumptions of Theorem 1, assume $\mathbf{(A3)}$ and $\mathbf{(A4)}$. Let $\mathbf{l}^\top \bSigma \mathbf{l}=O(p)$ and $\bmu^\top \bSigma \bmu=O(p)$. Then it holds that 
  \begin{eqnarray*}
    \sqrt{n}\sigma^{-1}\left(\bl^\top\bS\bar\bx - \bl^\top\bSigma(\bmu+\bB\bomega)\right)\overset{\mathcal{D}}{\longrightarrow}\mathcal{N}(0, 1)
  \end{eqnarray*}
  for $p/n=c+o\left(n^{-1/2}\right)$ with $c\ge 0$ and $q/n\to \gamma$ with $\gamma> 0$ as $n\to\infty$ where
  \begin{eqnarray*}
    \sigma^2&=&
             (\bl^\top\bSigma(\bB\bomega+\bmu))^2 +\bl^\top \bSigma^3 \bl
     +\bl^\top\bSigma\bl\left[ (\bmu+\bB\bomega)^\top\mathbf{\Sigma}(\bmu+\bB\bomega) + c\frac{\text{tr}(\bSigma^2)}{p}\right]\,.
  \end{eqnarray*}
\end{corollary}
\begin{proof}
Since 
\begin{eqnarray*}
(\bu_i^\top\bmu_{\mathbf{a},\bnu})^2&=&\left(\bu_i^\top\bmu+\bu_i^\top\bB\bomega+\bu_i^\top\bB(\bnu-\bomega)+\bu_i^\top\bl\frac{a_2}{a_1}\right)^2\nonumber\\
&\overset{(A2), (A4)}{\le}& M^2_2+ M^2_2+M_2^2 q(\bnu-\bomega)^{\top}(\bnu-\bomega)+  M^2_2\frac{a_2^2}{a^2_1}\\
&+&2M^2_2\left(1+2\frac{|a_2|}{|a_1|}\right)+2M^2_2\left(2+\frac{|a_2|}{|a_1|}\right)\sqrt{q(\bnu-\bomega)^{\top}(\bnu-\bomega)} < \infty
\end{eqnarray*}
with probability one, we get from the proof of Theorem \ref{th2} and assumption (A3) that
  \begin{eqnarray}\label{stochT1}
    \sqrt{n}\sigma^{-1}\left(\bl^\top\bS\bar\bx - \bl^\top\bSigma(\bmu+\bB\bomega)\right)
    =\frac{\sigma_{\bnu}}{\sigma}z_0+\frac{\sqrt{n}}{\sqrt{q}}\frac{\sqrt{\bl^\top\bSigma \bB\bOmega\bB^\top \bSigma \bl}}{\sigma} \tilde{z}_0+o_P(1)\,,
  \end{eqnarray}
where $z_0\sim\mathcal{N}(0, 1)$ and $\tilde{z}_0$ are independently distributed. 

Let $\lambda_{max}(\mathbf A)$ denotes the largest eigenvalue of a symmetric positive definite matrix $\mathbf A$. First, we note that for any vector $\mathbf{m}$ that satisfies $\mathbf{(A2)}$, we get
\begin{eqnarray*}
\mathbf{m}^\top\bSigma \bB\bOmega\bB^\top \bSigma \mathbf{m}&\le& \lambda_{max}(\bSigma)^2 \lambda_{max}(\mathbf{U}^T\bB\bOmega\bB^\top\mathbf{U}) \sum_{i=1}^{p} |\mathbf m^\top \bu_i| =O(p),
\end{eqnarray*}
where $\mathbf{U}$ is the eigenvector matrix of $\bSigma$ which is a unitary matrix and, consequently, (see, e.g., Chapter 8.5 in \citet{Lutkepohl1996}) $$\lambda_{max}(\mathbf{U}^T\bB\bOmega\bB^\top\mathbf{U})=\lambda_{max}(\bB\bOmega\bB^\top)=O(1).$$
As a result, we get $\bl^\top\bSigma \bB(\bnu-\bomega) =O_P(1)$ and $(\bmu+\bB\bomega)^\top\bSigma \bB(\bnu-\bomega) =O_P(1)$. 

Furthermore, it holds that
\begin{eqnarray*}
(\bnu-\bomega)^\top\bB^\top \bSigma \bB (\bnu-\bomega) &\le& \lambda_{max}(\bSigma) \sum_{i=1}^{p} (\mathbf{u}_i^\top \bB (\bnu-\bomega))^2
= O_P(1) \frac{1}{q}\sum_{i=1}^{p} \mathbf{u}_i^\top \bB\bOmega\bB^\top \mathbf{u}_i \\
&\le&  O_P(1) \lambda_{max}(\bB\bOmega\bB^\top) \frac{1}{q}\sum_{i=1}^{p} \mathbf{u}_i^\top \mathbf{u}_i= O_P(1).
\end{eqnarray*}

Hence, we have
\[ \frac{\sqrt{\bl^\top\bSigma \bB\bOmega\bB^\top \bSigma \bl}}{\sigma} \to 0\]
and, consequently, together with $\tilde{z}_0=O_P(1)$ and $q/n\to \gamma>0$ it ensures that the second summand in \eqref{stochT1} will disappear.
At last,
\[\frac{\sigma_{\bnu}^2}{\sigma^2} \overset{a.s.}{\to} 1\]
for $p/n=c+o\left(n^{-1/2}\right)$ with $c\ge 0$ and $q/n\to \gamma$ with $\gamma> 0$ as $n\to\infty$.
\end{proof}

Note that the asymptotic regime $q/n\to \gamma>0$ ensures that the dimension of the random coefficients is not increasing much faster than the sample size, i.e., $q$ and $n$ are comparable. It is worth mentioning that due to Corollary 1 the asymptotic distribution of the bilinear form $\bl^\top\bS\bar{\bx}$ does not depend on the covariance matrix of the random effects $\bOmega$. This knowledge may be further useful in constructing a proper pivotal statistic.

\subsection{CLT for the product of inverse sample covariance matrix and sample mean vector }

In this section we consider the distributional properties of the product of the inverse sample covariance matrix $\mathbf{S}^{-1}$ and the sample mean vector $\overline{\mathbf{x}}$. Again we prove that proper weighted bilinear forms involving $\bS^{-1}$ and $\overline{\bx}$ have asymptotically a normal distribution.
This result is summarized in Theorem \ref{th4}.
\begin{theorem}
\label{th4}
Assume $\mathbf X\sim  \mathcal{LMN}_{p,n;q} (\bmu,\mathbf{\Sigma} , \mathbf B ;  f_{\bnu} )$, $p<n-1$, with $\bSigma$ positive definite and let $p/n=c +o(n^{-1/2})$, $c\in[0, 1)$ as $n\to\infty$. Let $\bl$ be a $p$-dimensional vector of constants that satisfies $\mathbf{(A2)}$. Then, under $\mathbf{(A1)}$ and $\mathbf{(A2)}$ it holds that
\begin{eqnarray} \label{th3eq1}
 \sqrt{n}\tilde{\sigma}^{-1}_{\bnu}\left(\bl^\top\mathbf{S}^{-1}\overline{\mathbf{x}}-\frac{1}{1-c}\bl^\top\mathbf {\Sigma}^{-1} \bmu_{\bnu} \right)\overset{\mathcal{D}}{\longrightarrow} \mathcal{N}(0,1)
\end{eqnarray}
where $\bmu_{\bnu} = \bmu+\bB\bnu$ and
\begin{eqnarray*}
\tilde{\sigma}^{2}_{\bnu}&=&\frac{1}{(1-c)^3}\left(\left(\bl^\top\mathbf{\Sigma}^{-1} \bmu_{\bnu}\right)^2 +\bl^\top\bSigma^{-1}\bl(1+\bmu_{\bnu}^\top\mathbf{\Sigma}^{-1} \bmu_{\bnu})\right).
\end{eqnarray*}
\end{theorem}
\begin{proof}
From Theorem 3.4.1 of \citet{GuptaNagar2000} and Proposition \ref{th1} we get
\begin{eqnarray*}
\mathbf S ^{-1}
\sim
 \mathcal{IW}_p \left(n+p,(n-1)\mathbf{\Sigma}^{-1}\right).
\end{eqnarray*}

It holds that
\begin{eqnarray*}
\bl^\top\mathbf{S}^{-1}\overline{\mathbf{x}}
=
(n-1)\overline{\mathbf{x}}^\top\mathbf{\Sigma}^{-1}\overline{\mathbf{x}} \frac{\bl^\top\mathbf{S}^{-1}\overline{\mathbf{x}}}{\overline{\mathbf{x}}^\top\mathbf{S}^{-1}\overline{\mathbf{x}}}
\frac{\overline{\mathbf{x}}^\top\mathbf{S}^{-1}\overline{\mathbf{x}}}{(n-1)\overline{\mathbf{x}}^\top\mathbf{\Sigma}^{-1}\overline{\mathbf{x}}}.
\end{eqnarray*}

Since $\mathbf{S}^{-1}$ and $\overline{\mathbf{x}}$ are independently distributed, we get from Theorem 3.2.12 of \citet{Muirhead1982} that
\begin{eqnarray}\label{chi-squared distr}
\widetilde{\xi}=(n-1)\frac{\overline{\mathbf{x}}^\top\mathbf{\Sigma}^{-1}\overline{\mathbf{x}}}{\overline{\mathbf{x}}^\top\mathbf{S}^{-1}\overline{\mathbf{x}}}
\sim
\chi^{2}_{{n-p}}
\end{eqnarray}
and it is independent of $\overline{\mathbf{x}}$. Moreover, the application of Theorem 3 in \citet{BodnarOkhrin2008} proves that $\overline{\mathbf{x}}^\top\mathbf{S}^{-1}\overline{\mathbf{x}}$ is independent of $\bl^\top\mathbf{S}^{-1}\overline{\mathbf{x}}/\overline{\mathbf{x}}^\top\mathbf{S}^{-1}\overline{\mathbf{x}}$ for given $\overline{\mathbf{x}}$.
As a result, it is also independent of $\overline{\mathbf{x}}^\top\mathbf{\Sigma}^{-1}\overline{\mathbf{x}}\cdot\bl^\top\mathbf{S}^{-1}\overline{\mathbf{x}}/\overline{\mathbf{x}}^\top\mathbf{S}^{-1}\overline{\mathbf{x}}$ for given $\overline{\mathbf{x}}$ and, consequently, 
\begin{eqnarray*}
(n-1)\overline{\mathbf{x}}^\top\mathbf{\Sigma}^{-1}\overline{\mathbf{x}} \frac{\bl^\top\mathbf{S}^{-1}\overline{\mathbf{x}}}{\overline{\mathbf{x}}^\top\mathbf{S}^{-1}\overline{\mathbf{x}}} ~~ \text{and} ~~
\frac{\overline{\mathbf{x}}^\top\mathbf{S}^{-1}\overline{\mathbf{x}}}{(n-1)\overline{\mathbf{x}}^\top\mathbf{\Sigma}^{-1}\overline{\mathbf{x}}}.
\end{eqnarray*}
are independent.

From the proof of Theorem 1 of \citet{BodnarSchmid2008} we obtain
\begin{eqnarray}
\label{t-distr}
(n-1)\overline{\mathbf{x}}^\top\mathbf{\Sigma}^{-1}\overline{\mathbf{x}}\frac{\mathbf{l}^\top\mathbf{S}^{-1}\overline{\mathbf{x}}}
{\overline{\mathbf{x}}^\top\mathbf{S}^{-1}\overline{\mathbf{x}}} \Bigg|\overline{\mathbf{x}}
\sim
t\left({n-p+1}; (n-1)\bl^\top\mathbf {\Sigma}^{-1} \overline{\mathbf x};(n-1)^2\frac{\overline{\mathbf{x}}^\top\mathbf{\Sigma}^{-1}\overline{\mathbf{x}}}{{n-p+1}}\bl^\top\mathbf{R}_{\overline{\mathbf{x}}}\bl\right),
\end{eqnarray}
where {$\mathbf{R}_{\mathbf{a}}=\mathbf{\Sigma}^{-1}-\mathbf{\Sigma}^{-1}\mathbf{a}\mathbf{a}^\top\mathbf{\Sigma}^{-1}/\mathbf{a}^\top\mathbf{\Sigma}^{-1}\mathbf{a}$, $\mathbf{a} \in I\!\!R^p$,} and the symbol $t(k,\mu,\tau^2)$ denotes a $t$-distribution with $k$ degrees of freedom, location parameter $\mu$ and scale parameter $\tau^2$.

Combining (\ref{chi-squared distr}) and (\ref{t-distr}), we get the stochastic representation of $\bl^\top\mathbf{S}^{-1}\overline{\mathbf x}$ given by
\begin{eqnarray*}
\bl^\top\mathbf{S}^{-1}\overline{\mathbf{x}}
&\stackrel{d}{=}&
\widetilde{\xi}^{-1} (n-1)
\left(\bl^\top\mathbf {\Sigma}^{-1} \overline {\mathbf x}
+t_{0}\sqrt{\frac{\overline{\mathbf{x}}^\top\mathbf{\Sigma}^{-1}\overline{\mathbf{x}}}{n-p+1}\cdot\bl^\top\mathbf{R}_{\overline{\mathbf{x}}}\bl}\right)\\
&=&\widetilde{\xi}^{-1} (n-1)\left(\bl^\top\mathbf {\Sigma}^{-1} \overline {\mathbf x}
+\frac{t_{0}}{\sqrt{n-p+1}}\sqrt{\bl^\top\mathbf{\Sigma}^{-1}\bl} \sqrt{\overline{\mathbf{x}}^\top\mathbf{R}_{\bl}\overline{\mathbf{x}}}\right),
\end{eqnarray*}
where (see, Proposition \ref{th1})
\[
\bl^\top\mathbf {\Sigma}^{-1} \overline{\mathbf{x}}\stackrel{d}{=} \bl^\top\mathbf {\Sigma}^{-1} ( \bmu+\mathbf{B}\bnu)+ \frac{\bl^\top\mathbf {\Sigma}^{-1} \bl}{\sqrt{n}}z'_0,\]
with
$\widetilde{\xi}\sim \chi^{2}_{n-p}$, $z'_0\sim \mathcal{N}(0,1)$, and $t_{0}\sim t(n-p+1,0, 1)$; $\widetilde{\xi}$, $\bnu$, $z'_0$, and $t_{0}$ are mutually independent.

Since $\mathbf{R}_{\bl} \bSigma \mathbf{R}_{\bl}=\mathbf{R}_{\bl}$, $tr(\mathbf{R}_{\bl} \bSigma)=p-1$, and $\mathbf{R}_{\bl} \bSigma \bSigma^{-1} \bl =\mathbf{0}$, the application of Corollary 5.1.3a and Theorem 5.5.1 in \citet{MathaiProvost1992} leads to
\begin{eqnarray}
n \overline{\mathbf{x}}^{\top}\mathbf{R}_{\bl}\overline{\mathbf{x}} |\bnu \sim \chi^2_{p-1}(n\delta^2(\bnu)) ~~ \text{with}~~ \delta^2( \bnu)= (\bmu+\mathbf{B}\bnu)^\top \mathbf{R}_{\bl} (\bmu+\mathbf{B}\bnu)\label{term2}\,
\end{eqnarray}
as well as $\bl^\top\mathbf {\Sigma}^{-1} \overline {\mathbf x}$ and $\overline{\mathbf{x}}^\top\mathbf{R}_{\bl}\overline{\mathbf{x}}$ are independent given $\bnu$. Finally, using the stochastic representation of a $t$-distributed random variable, we get
\begin{eqnarray}
  t_0&\stackrel{d}{=}&\frac{z''_0}{\sqrt{\frac{\zeta}{n-p+1}}}
\end{eqnarray}
with $\zeta\sim\chi^2_{n-p+1}$, $z''_0\sim\mathcal{N}(0, 1)$ and $z'_0\sim\mathcal{N}(0, 1)$ being mutually independent.
Together with \eqref{term2} it yields
\begin{eqnarray}\label{stoch2}
\bl^\top\mathbf{S}^{-1}\overline{\mathbf{x}}
  &\stackrel{d}{=}&{ \widetilde{\xi}^{-1} (n-1)\left(\bl^\top\mathbf {\Sigma}^{-1} (\bmu+\mathbf{B}\bnu) + \frac{z'_0}{\sqrt{n}}\sqrt{\bl^\top\bSigma^{-1}\bl}+ \frac{z''_{0}}{\sqrt{n}} \sqrt{\bl^\top\mathbf{\Sigma}^{-1}\bl} \sqrt{\frac{p-1}{n-p+1}\eta}\right)   }\nonumber\\
&=&\widetilde{\xi}^{-1} (n-1)\left(\bl^\top\mathbf {\Sigma}^{-1} (\bmu+\mathbf{B}\bnu)
+\sqrt{\bl^\top\mathbf{\Sigma}^{-1}\bl} \sqrt{1+\frac{p-1}{n-p+1}\eta}\frac{z_{0}}{\sqrt{n}}\right)\,,
\end{eqnarray}
where the last equality in \eqref{stoch2} follows from the fact that
  $$z'_0+z''_{0}\sqrt{\frac{p-1}{n-p+1}\eta}\stackrel{d}{=} z_0\sqrt{1+\frac{p-1}{n-p+1}\eta}$$ with $\eta=\frac{n\overline{\mathbf{x}}^\top\mathbf{R}_{\bl}\overline{\mathbf{x}}/(p-1)}{\zeta/(n-p+1)}|\bnu\sim F_{p-1,n-p+1}(n\delta^2(\bnu))$ (non-central $F$-distribution with $p-1$ and $n-p+1$ degrees of freedom and non-centrality parameter $n\delta^2(\bnu)$). Moreover, we have that $\widetilde{\xi}\sim \chi^{2}_{n-p}$ and $z_{0}\sim \mathcal{N}(0,1)$; $\widetilde{\xi}$, $z_0$, and $\eta$ are mutually independent.

From Lemma 6.4.(b) in \citet{BodnarHautschParolya2016} we get
\begin{equation*}
\sqrt{n}\left(
\left(
\begin{array}{c}
  {\widetilde{\xi}}/{(n-p)}  \\
  \eta \\
  {z_{0}}/{\sqrt{n}}
\end{array}
\right)
-\left(
\begin{array}{c}
  1 \\
  1+\delta^2(\bnu)/c \\
  0 \\
\end{array}
\right)\right) \Bigg|\bnu
\overset{\mathcal{D}}{\longrightarrow}
\mathcal{N}\left(\mathbf{0},
\left(
\begin{array}{ccc}
  2/(1-c) & 0 & 0\\
  0 &\sigma^2_{\eta} & 0 \\
  0 & 0 & 1 \\
  0 & 0 & 0
\end{array}
\right)
\right)
\end{equation*}
for $p/n=c+o(n^{-1/2})$, $c \in[0, 1)$ as $n\to\infty$ with
\[
\sigma^2_{\eta}=\dfrac{2}{c}\left(1+2\dfrac{\delta^2(\bnu)}{c}\right)+\dfrac{2}{1-c} \left(1+\dfrac{\delta^2(\bnu)}{c}\right)^2
\]

Consequently,
\begin{eqnarray*}
&&\sqrt{n}\left(
\left(
\begin{array}{c}
  {\widetilde{\xi}}/{(n-1)}  \\
  (p-1)\eta/(n-p+1) \\
  {z_{0}}/{\sqrt{n}}
\end{array}
\right)
-\left(
\begin{array}{c}
  (1-c) \\
  (c+\delta^2(\bnu))/(1-c) \\
  0
\end{array}
\right)\right)\Bigg|\bnu
\\
&&\overset{\mathcal{D}}{\longrightarrow}
\mathcal{N}\left(\mathbf{0},
\left(
\begin{array}{ccc}
  2(1-c) & 0 & 0 \\
  0 &c^2\sigma^2_{\eta}/(1-c)^2 & 0 \\
  0 & 0 & 1
 \end{array}
\right)
\right)
\end{eqnarray*}
for $p/n=c+o(n^{-1/2})$, $c \in[0, 1)$ as $n\to\infty$.

Finally, the application of the delta-method (c.f. \citet[Theorem 3.7]{dasGupta}) leads to
\[\sqrt{n}\left(\bl^\top\mathbf{S}^{-1}\overline{\mathbf{x}}-\frac{1}{1-c}\bl^\top\mathbf {\Sigma}^{-1} (\bmu+\mathbf{B}\bnu) \right)\Big|\bnu\overset{\mathcal{D}}{\longrightarrow} \mathcal{N}(0, \tilde{\sigma}^2_{\bnu})\]
for $p/n=c+o(n^{-1/2})$, $c \in[0, 1)$ as $n\to\infty$ with
 \begin{eqnarray*}
   \tilde{\sigma}^2_{\bnu} = \frac{1}{(1-c)^3}\left(2\left(\bl^\top\mathbf {\Sigma}^{-1} (\bmu+\mathbf{B}\bnu)\right)^2
  +\bl^\top\bSigma^{-1}\bl(1+\delta^2(\bnu))\right)\,.
\end{eqnarray*}
Consequently,
\begin{eqnarray*}
  \sqrt{n}\tilde{\sigma}^{-1}_{\bnu}\left(\bl^\top\mathbf{S}^{-1}\overline{\mathbf{x}}-\frac{1}{1-c}\bl^\top\mathbf {\Sigma}^{-1} (\bmu+\mathbf{B}\bnu) \right)\Big|\bnu\overset{\mathcal{D}}{\longrightarrow} \mathcal{N}(0,1)\,,
\end{eqnarray*}
where the asymptotic distribution does not depend on $\bnu$. Hence, it is also the unconditional asymptotic distribution.
\end{proof}

Again, Theorem \ref{th4} shows that the distribution of $\bl^\top\bS^{-1}\overline{\bx}$ can be approximated by a mixture of normal distributions. Indeed,
\begin{eqnarray}
  p^{-1}\bl^\top\bS^{-1}\overline{\bx}|\bnu \approx \mathcal{CN}\left(\frac{p^{-1}}{1-c}\bl^\top\bSigma^{-1}\bmu_{\bnu}, \frac{p^{-2}\tilde{\sigma}^2_{\bnu}}{n}\right)\,.
\end{eqnarray}
In the proof of Theorem \ref{th4} we can read out that the stochastic representation for the product of the inverse sample covariance matrix and the sample mean
vector is presented by using a $\chi^{2}$ distributed random variable, a general skew normally distributed random vector and a standard $t$-distributed random variable. This result is itself very useful and allows to generate the values of $\bl^\top\bS^{-1}\overline{\bx}$ by just simulating three random variables from the standard univariate distributions and a random vector $\bnu$ which determines the family of the matrix-variate location mixture of normal distributions. The assumptions about the boundedness of the quadratic and bilinear forms involving $\bSigma^{-1}$ plays here the same role as in Theorem \ref{th2}. Note that in this case we need no assumption either on the Frobenius norm of the covariance matrix or its inverse.

Finally, in Corollary \ref{cor2} we formulate the CLT for the product of the inverse sample covariance matrix and the sample mean vector. 
\begin{corollary}[CLT]\label{cor2}
   Under the assumptions of Theorem \ref{th4}, assume $\mathbf{(A3)}$ and $\mathbf{(A4)}$. Let $\mathbf{l}^\top \bSigma^{-1} \mathbf{l}=O(p)$ and $\bmu^\top \bSigma^{-1} \bmu=O(p)$.  Then it holds that
\begin{eqnarray*}
 \sqrt{n}\tilde{\sigma}^{-1}\left(\bl^\top\mathbf{S}^{-1}\overline{\mathbf{x}}-\frac{1}{1-c}\bl^\top\mathbf {\Sigma}^{-1} (\bmu+\bB\bomega) \right)\overset{\mathcal{D}}{\longrightarrow} \mathcal{N}(0,1)
\end{eqnarray*}
  for $p/n=c+o\left(n^{-1/2}\right)$ with $c\in[0, 1)$ and $q/n\to \gamma$ with $\gamma> 0$ as $n\to\infty$ where
  \begin{eqnarray*}
    \tilde{\sigma}^2&=&
     \frac{1}{(1-c)^3}\left(\left(\bl^\top\mathbf{\Sigma}^{-1}(\bmu+\bB\bomega)\right)^2 +\bl^\top\bSigma^{-1}\bl(1+(\bmu+\bB\bomega)^\top\mathbf{\Sigma}^{-1} (\bmu+\bB\bomega))\right)\,.
  \end{eqnarray*}
\end{corollary}

\begin{proof}
From the proof of Theorem \ref{th4} and assumption $\mathbf{(A3)}$ it holds that 
  \begin{eqnarray*}
     \sqrt{n}\tilde{\sigma}^{-1}\left(\bl^\top\mathbf{S}^{-1}\overline{\mathbf{x}}-\frac{1}{1-c}\bl^\top\mathbf {\Sigma}^{-1} (\bmu+\bB\bomega) \right)
    =\frac{\tilde{\sigma}_{\bnu}}{\tilde{\sigma}}z_0+\frac{1}{1-c}\frac{\sqrt{n}}{\sqrt{q}}\frac{\sqrt{\bl^\top\bSigma^{-1} \bB\bOmega\bB^\top \bSigma^{-1} \bl}}{\tilde{\sigma}} \tilde{z}_0+o_P(1)\,,
  \end{eqnarray*}
where $z_0\sim\mathcal{N}(0, 1)$ and $\tilde{z}_0$ are independently distributed.

Finally, in a similar way like in Corollary 1 we get from $\mathbf{(A4)}$ that
\begin{eqnarray}
  (1-c)^{-2}\gamma^{-1}\frac{\bl^\top\bSigma^{-1}\bB\bOmega\bB^\top\bSigma^{-1}\bl}{\tilde{\sigma}^2}&\to& 0,\\
  \frac{\tilde{\sigma}_{\bnu}^2}{\tilde{\sigma}^2} &\overset{a.s.}{\to}& 1
\end{eqnarray}
for $p/n=c+o\left(n^{-1/2}\right)$ with $c\in[0,1)$ and $q/n\to \gamma$ with $\gamma> 0$ as $n\to\infty$.
\end{proof}

\section{Numerical study}

In this section we provide a Monte Carlo simulation study to investigate the performance of
the suggested CLTs for the products of the (inverse) sample covariance matrix and the sample mean vector.

In our simulations we put $\mathbf l = \mathbf 1_p$, each element of the vector $\bmu$ is uniformly distributed on $[-1,1]$ while each element of the matrix $\mathbf B$ is uniformly distributed on $[0,1]$.
Also, we take $\mathbf \Sigma$ as a diagonal matrix where each diagonal element is uniformly distributed on $[0,1]$. It can be checked that in such a setting the assumptions $\mathbf{(A1)}$ and $\mathbf{(A2)}$ are satisfied. Indeed, the population covariance matrix satisfies the condition $\mathbf{(A1)}$ because the probability of getting exactly zero eigenvalue equals to zero. On the other hand, the condition $\mathbf{(A2)}$ is obviously valid too because the $i$th eigenvector of $\bSigma$ is $\bu_i=\mathbf{e}_i=(0,\ldots,\underset{\text{ith place}}{1},0,\ldots,0)^\prime$.

In order to define the distribution for the random vector $\bnu$, we consider two special cases.
In the first case we take $\bnu = |\bpsi|$, where $\bpsi \sim \mathcal {N}_q ( \mathbf 0, \mathbf I_q)$, i.e., $\bnu$ has a $q$-variate truncated normal distribution.
In the second case we put $\bnu \sim \mathcal{GAL}_q (\mathbf I_q, \mathbf 1_q, 10)$, i.e., $\bnu$ has a $q$-variate generalized asymmetric Laplace distribution (c.f., \citet{KPR2013}).
Also, we put $q=10$.

We compare the results for several values of $c \in \{0.1, 0.5, 0.8, 0.95\}$.
The simulated data consist of $N=10^5$ independent realizations which are used to fit the corresponding kernel density estimators with Epanechnikov kernel.
The bandwith parameters are determined via cross-validation for every sample.
The asymptotic distributions are simulated using the results of Theorems \ref{th2} and \ref{th4}.
The corresponding algorithm is given next:

\begin{itemize}
\item[a)]
generate $\bnu = |\bpsi|$, where $\bpsi \sim \mathcal {N}_q ( \mathbf 0_q, \mathbf I_q)$, or generate $\bnu \sim \mathcal {GAL} _q (\mathbf I_q, \mathbf 1_q, 10)$;

\item[b)] generate $\mathbf l ^\top \mathbf S \overline {\mathbf x}$ by using the stochastic representation \eqref{stoch1} obtained in the proof of Theorem \ref{th2}, namely
\begin{eqnarray*}
\bl^\top\mathbf{S}\overline{\mathbf{x}}
\stackrel{d}{=}
{\dfrac{\xi}{n-1}}\bl^\top\mathbf{\Sigma}(\mathbf{y}+\mathbf B\bnu)
+
\dfrac{\sqrt{\xi}}{n-1}
((\mathbf{y}+\mathbf B\bnu)^\top\mathbf{\Sigma}(\mathbf{y}+\mathbf B\bnu)\bl^\top\mathbf{\Sigma}\bl-(\bl^\top\mathbf{\Sigma}(\mathbf{y}+\mathbf B\bnu))^2)^{1/2}z_{0},
\end{eqnarray*}
where $\xi \sim  \chi^{2}_{{n-1}}$, $ z_0\sim \mathcal{N} (0, 1)$, $\mathbf{y}\sim \mathcal{N}_p(\bmu,\frac{1}{n}\mathbf{\Sigma})$; $\xi$, $z_{0}$, $\mathbf{y}$, and $\bnu$ are mutually independent

\item[b')] generate $\mathbf l ^\top \mathbf S^{-1} \overline {\mathbf x}$ by using the stochastic representation \eqref{stoch2} obtained in the proof of Theorem \ref{th4}, namely
\begin{eqnarray*}
\bl^\top\mathbf{S}^{-1}\overline{\mathbf{x}}
&\stackrel{d}{=}&
\widetilde{\xi}^{-1} (n-1)\left(\bl^\top\mathbf {\Sigma}^{-1} (\bmu+\mathbf{B}\bnu)
+\sqrt{\bl^\top\mathbf{\Sigma}^{-1}\bl} \sqrt{1+\frac{p-1}{n-p+1}\eta}\frac{z_{0}}{\sqrt{n}}\right),
\end{eqnarray*}
where $\widetilde{\xi}\sim \chi^{2}_{n-p}$, $z_{0}\sim \mathcal{N}(0,1)$, and $\eta\sim F_{p-1,n-p+1}(n\delta^2(\bnu))$ with $\delta^2(\bnu)=(\bmu+\mathbf{B}\bnu)^\top \mathbf{R}_{\bl} (\bmu+\mathbf{B}\bnu)$, $\mathbf{R}_{\mathbf{l}}=\mathbf{\Sigma}^{-1}-\mathbf{\Sigma}^{-1}\mathbf{l}\mathbf{l}^\top\mathbf{\Sigma}^{-1}/\mathbf{l}^\top\mathbf{\Sigma}^{-1}\mathbf{l}$; $\widetilde{\xi}$, $z_0$ and $(\eta,\bnu)$ are mutually independent.


\item[c)]
compute
$$
\sqrt{n} \sigma^{-1}_{\bnu}  \left( \mathbf l ^\top \mathbf S \overline {\mathbf x} - \mathbf l^\top \bSigma \bmu_{\bnu} \right)
$$
and
$$
\sqrt{n} \tilde{\sigma}^{-1}_{\bnu}  \left( \mathbf l ^\top \mathbf S^{-1} \overline {\mathbf x} - \frac{1}{1-c}\mathbf l^\top \bSigma^{-1} \bmu_{\bnu} \right)
$$
where
\begin{eqnarray*}
\bmu_{\bnu} &=& \bmu + \mathbf B \bnu
\\
\sigma^2 _{\bnu} &=& \left[\bmu^\top_{\bnu}\bSigma\bmu_{\bnu}+c||\bSigma||^2_F\right]\bl^\top\bSigma\bl+(\bl^\top\bSigma\bmu_{\bnu})^2 +\bl^\top \bSigma^3 \bl
\\
\tilde{\sigma}^{2}_{\bnu}&=&\frac{1}{(1-c)^3}\left(2\left(\bl^\top\mathbf{\Sigma}^{-1} \bmu_{\bnu}\right)^2
  +\bl^\top\bSigma^{-1}\bl(1+\delta^2(\bnu))\right)
\end{eqnarray*}
with $\delta^2(\bnu)=\bmu_{\bnu}^\top \mathbf{R}_{\bl} \bmu_{\bnu}$, $\mathbf{R}_{\bl}=\mathbf{\Sigma}^{-1}-\mathbf{\Sigma}^{-1}\mathbf{l}\mathbf{l}^\top\mathbf{\Sigma}^{-1}/\mathbf{l}^\top\mathbf{\Sigma}^{-1}\mathbf{l}$.
\item[d)]
repeat a)-c) $N$ times.
\end{itemize}

It is remarkable that for generating $\mathbf l ^\top \mathbf S \overline {\mathbf x}$ and $\mathbf l ^\top \mathbf S^{-1} \overline {\mathbf x}$ only random variables from the standard distributions are need. Neither the data matrix $\mathbf{X}$ nor the sample covariance matrix $\mathbf{S}$ are used.

$$
[\ Figures\ \  \ref{fig1}-\ref{fig8}\ ]
$$

In Figures \ref{fig1}-\ref{fig4} we present the results of simulations for the asymptotic distribution that is given in Theorem \ref{th2} while the asymptotic distribution as given in Theorem \ref{th4} is presented in Figures \ref{fig5}-\ref{fig8} for different values of $c=\{0.1,0.5,0.8,0.95\}$.
The suggested asymptotic distributions are shown as a dashed black line, while the standard normal distribution is a solid black line.
All results demonstrate a good performance of both asymptotic distributions for all considered values of $c$. Even in the extreme case $c=0.95$ our asymptotic results seem to produce a quite reasonable approximation.
Moreover, we observe a good robustness of our theoretical results for different distributions of $\bnu$.
Also, we observe that all asymptotic distributions are slightly skewed to the right for the finite dimensions. This effect is even more significant in the case of  the  generalized asymmetric Laplace distribution.
 Nevertheless, the skewness disappears with growing dimension and sample size, i.e., the distribution becomes symmetric one and converges to its asymptotic counterpart.

\section{Summary}

In this paper we introduce the family of the matrix-variate location mixture of normal distributions that generalizes a large number of
the existing skew normal models.
Under the MVLMN we derive the distributions of the sample mean vector and the sample covariance matrix.
Moreover, we show that they are independently distributed.
Furthermore, we derive the CLTs under the high-dimensional asymptotic regime for the products of the (inverse) sample covariance matrix and the sample mean vector.
In the numerical study, the good finite sample performance of both asymptotic distributions is documented.

\subsection*{Acknowledgement}
The authors are thankful to Professor Niels Richard Hansen, the Associate Editor, and two anonymous Reviewers for careful reading of the manuscript and for their suggestions which have improved an earlier version of this paper.

\section{Appendix}

\noindent {\bf Proof of Proposition \ref{kapec}.}
\begin{proof} Straightforward but tedious calculations give
{\small
\begin{eqnarray*}
f_{\mathbf{X}}(\mathbf{Z})
&=&
C^{-1}\int_{\mathbb{R}^{q}_{+}}
f_{N_{p, n}(\bmu\mathbf{1}_{n}^\top,\mathbf{\Sigma}\otimes\mathbf{I}_{n})}(\mathbf{Z}-\mathbf{B}\bnu^*\mathbf{1}_{n}^\top)
f_{N_{q}(\mathbf{0},\mathbf{\Omega})}(\bnu^*)
\mathbf{d\bnu}^*
\\
&=&
C^{-1} \frac{(2\pi)^{-(np+q)/2}}{|\mathbf{\Omega}|^{1/2}|\mathbf{\Sigma}|^{n/2}}
\int_{\mathbb{R}^{q}_{+}}
\exp\left\{-\frac{1}{2}\bnu^{*\top}\mathbf{\Omega}^{-1}\bnu^*\right\}
\\
&\times&
\exp\left\{-\frac{1}{2}\mathbf{vec}\left(\mathbf{Z}-\bmu\mathbf{1}_{n}^\top-\mathbf{B}\bnu^*\mathbf{1}_n^\top\right)^\top(\mathbf{I}_n\otimes\mathbf{\Sigma})^{-1}\mathbf{vec}\left(\mathbf{Z}-\bmu\mathbf{1}_{n}^\top-\mathbf{B}\bnu^*\mathbf{1}_n^\top\right)\right\}
\mathbf{d}\bnu^*
\\
&=&
C^{-1}\frac{(2\pi)^{-(np+q)/2}}{|\mathbf{\Omega}|^{1/2}|\mathbf{\Sigma}|^{n/2}}\exp\left\{-\frac{1}{2}\mathbf{vec}(\mathbf{Z}-\bmu\mathbf{1}_n^\top)^\top
(\mathbf{I}_n\otimes\mathbf{\Sigma})^{-1}\mathbf{vec}(\mathbf{Z}-\bmu\mathbf{1}_n^\top)\right\}
\\
&\times&
\exp\left\{\frac{1}{2}\mathbf{vec}(\mathbf{Z}-\bmu\mathbf{1}_n^\top)^\top \mathbf{E}^\top
\mathbf{D}\mathbf{Evec}(\mathbf{Z}-\bmu\mathbf{1}_n^\top)\right\}
\\
&\times&
\int_{\mathbb{R}^{q}_{+}}\exp\left\{-\frac{1}{2}\left[\left(\bnu^*-\mathbf{D}\mathbf{Evec}(\mathbf{Z}-\bmu\mathbf{1}_n^\top)\right)^\top
\mathbf{D}^{-1}\left(\bnu^*-\mathbf{D}\mathbf{Evec}(\mathbf{Z}-\bmu\mathbf{1}_n^\top)\right)\right]\right\}\mathbf{d}\bnu^*
\\
&=&
C^{-1}\frac{|\mathbf{F}|^{1/2}|\mathbf{D}|^{1/2}}{|\mathbf{\Omega}|^{1/2}|\mathbf{\Sigma}|^{n/2}}
\Phi_{q}\left(\mathbf{0};-\mathbf{D}\mathbf{E}\mathbf{vec}(\mathbf{Z}-\bmu\mathbf{1}_n^\top),\mathbf{D}\right)
\phi_{pn}\left(\mathbf{vec}(\mathbf{Z}-\bmu\mathbf{1}_n^\top);\mathbf{0},\mathbf{F}\right)
\\
&=&
\widetilde{C}^{-1}\Phi_{q}\left(\mathbf{0};-\mathbf{D}\mathbf{E}\mathbf{vec}(\mathbf{Z}-\bmu\mathbf{1}_n^\top),\mathbf{D}\right)
\phi_{pn}\left(\mathbf{vec}(\mathbf{Z}-\bmu\mathbf{1}_n^\top);\mathbf{0},\mathbf{F}\right)
\end{eqnarray*}
where
$\mathbf{D} = (n\mathbf{B}^\top\mathbf{\Sigma}^{-1}\mathbf{B}+\mathbf{\Omega}^{-1})^{-1}$,
$\mathbf{E}=\mathbf{1}_n^\top\otimes\mathbf{B}^\top\mathbf{\Sigma}^{-1}$, $\mathbf{F}=(\mathbf{I}_n\otimes\mathbf{\Sigma}^{-1}-\mathbf{E}^\top\mathbf{D}\mathbf{E})^{-1}$, and
$$\widetilde{C}^{-1}=C^{-1}\frac{|\mathbf{F}|^{1/2}|\mathbf{D}|^{1/2}}{|\mathbf{\Omega}|^{1/2}|\mathbf{\Sigma}|^{n/2}}.$$
}
\end{proof}
\bibliographystyle{apalike}
\bibliography{SkewNormal}


\begin{figure}
  \includegraphics[width=8cm]{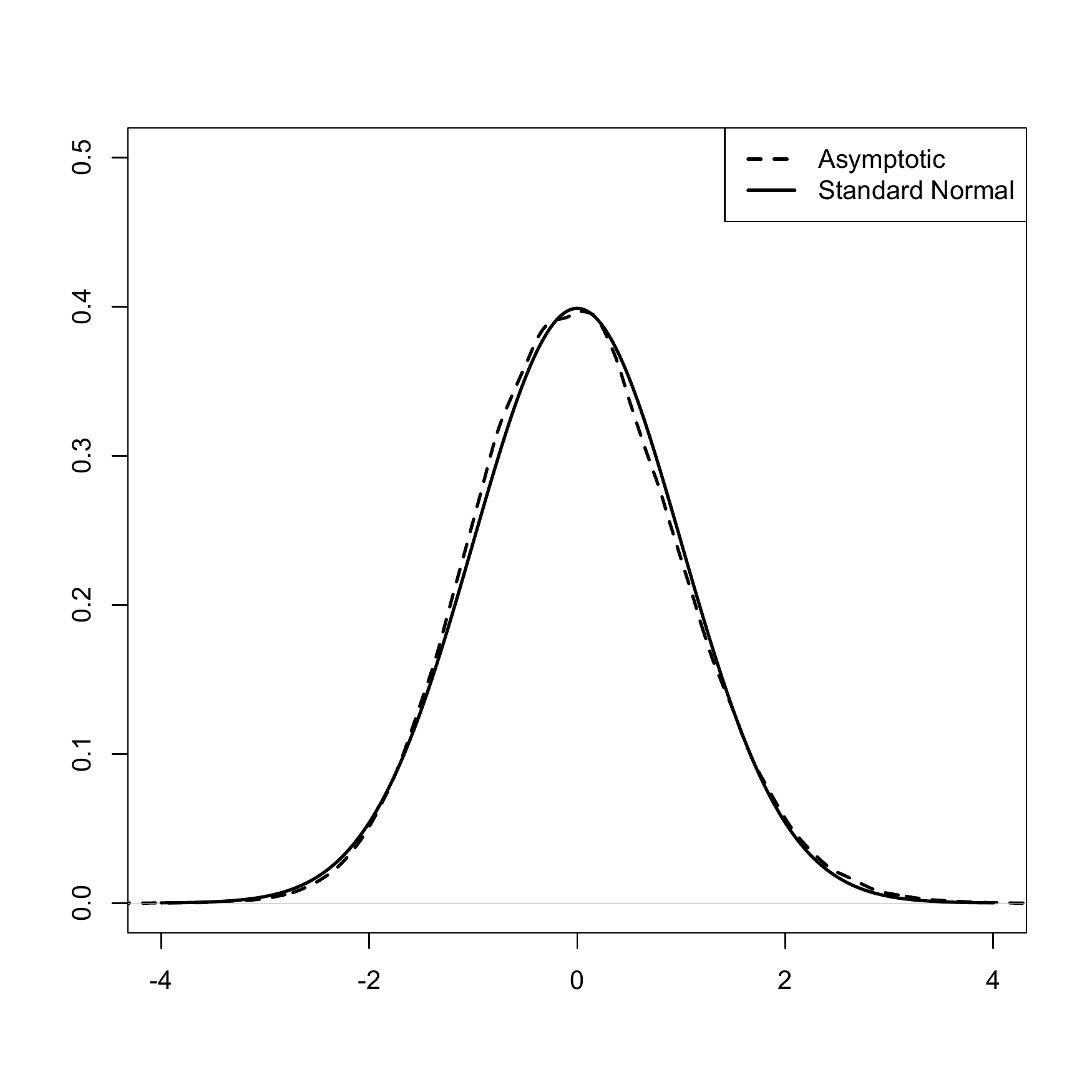}
  \includegraphics[width=8cm]{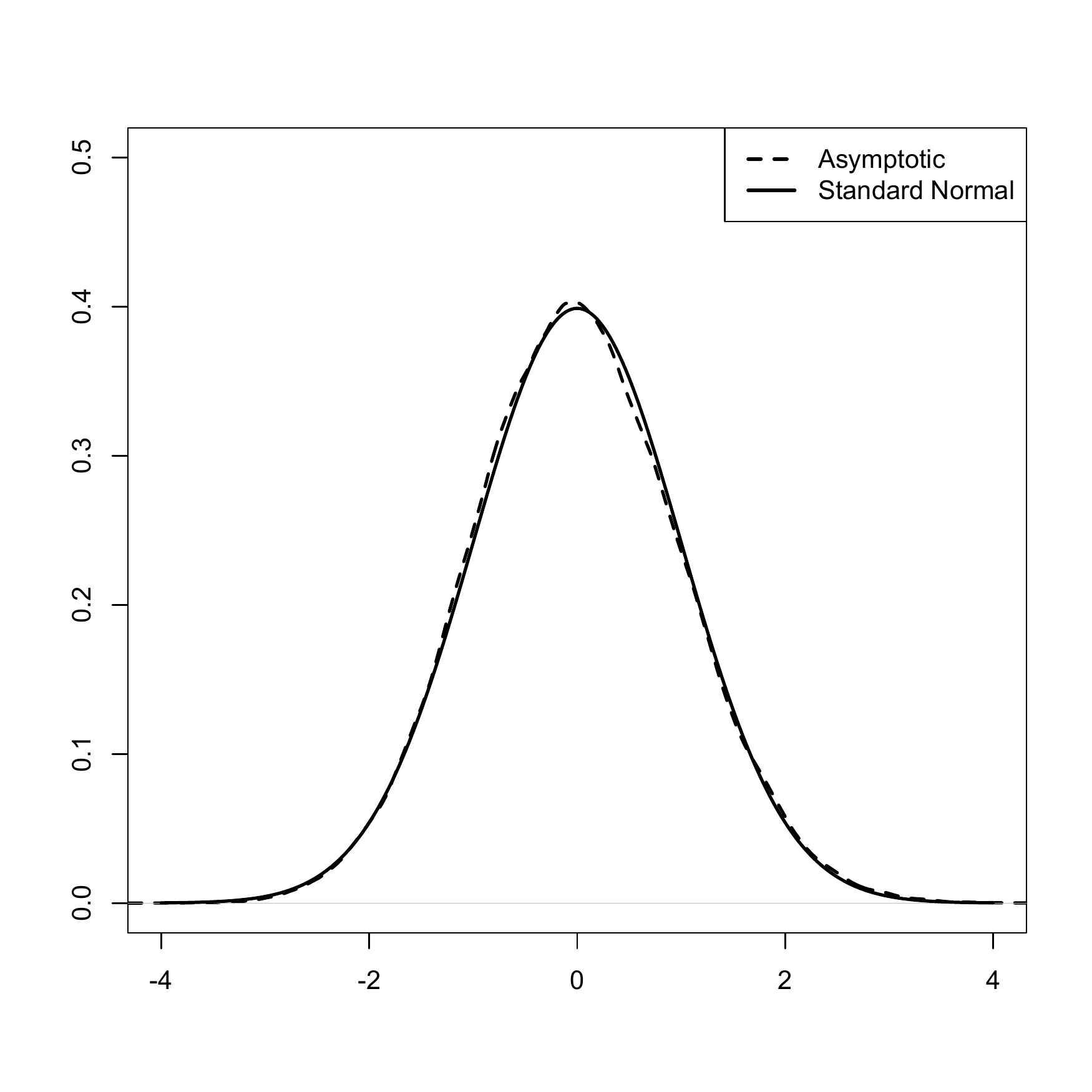}
  \begin{tabular}{lcrp{1.5in}}
(a)& $p=50$, $n=500$, $\bnu \sim\mathcal{TN} _q (\mathbf 0, \mathbf I_q) $.\\
\end{tabular}
\begin{tabular}{lcrp{1.5in}}
(b)&  $p=100$, $n=1000$, $\bnu \sim\mathcal{TN} _q (\mathbf 0, \mathbf I_q) $.\\
\end{tabular}
  \includegraphics[width=8cm]{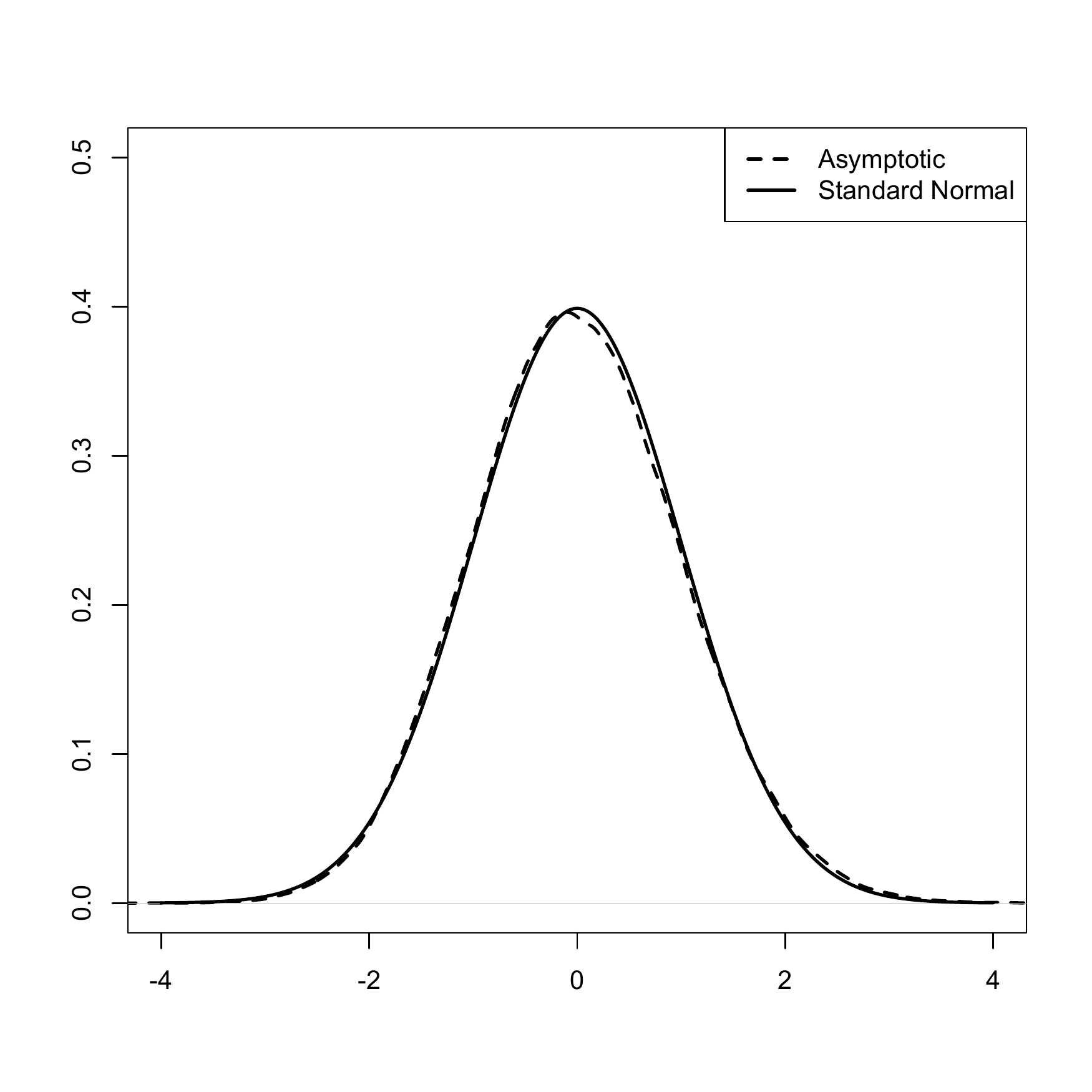}
  \includegraphics[width=8cm]{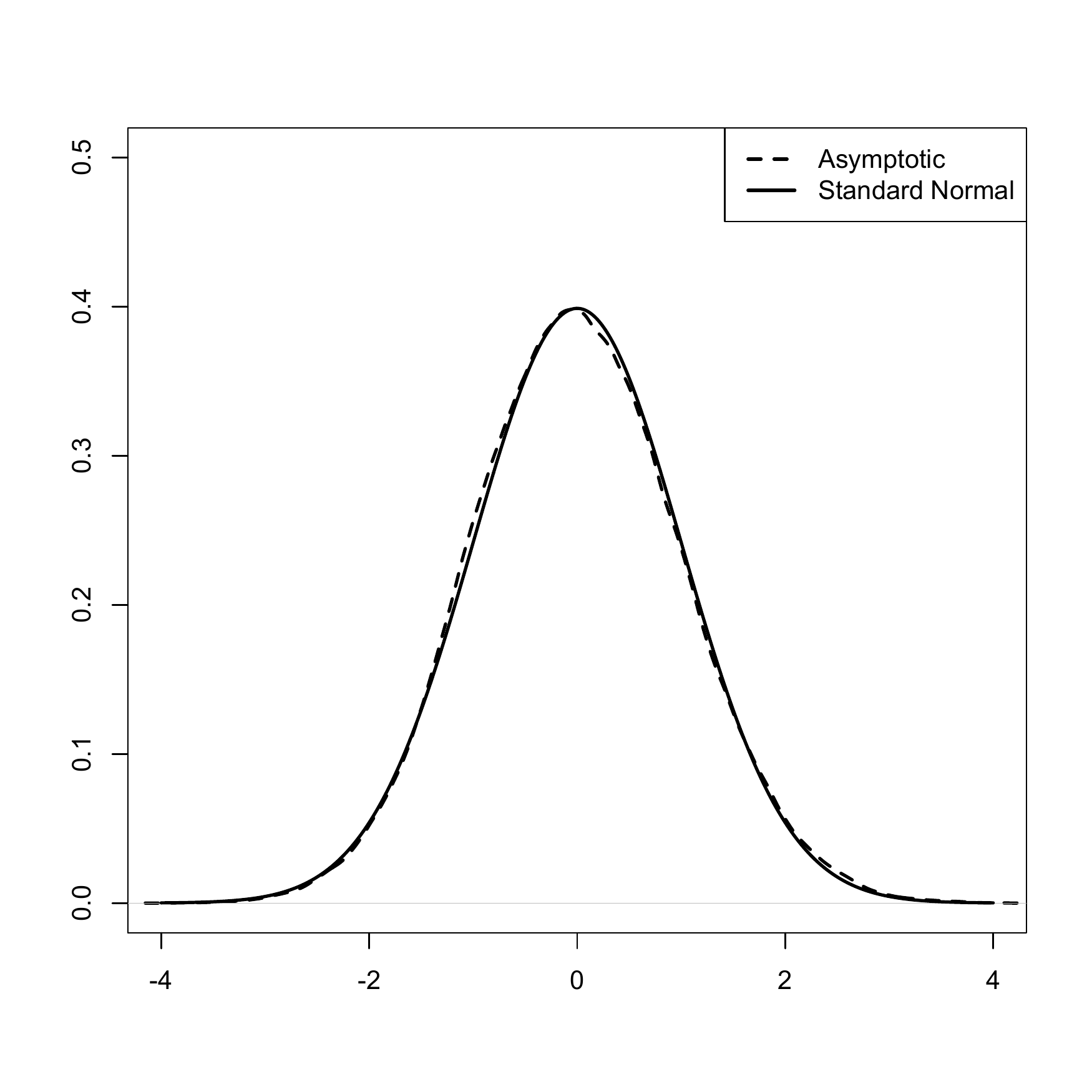}
  \begin{tabular}{lcrp{1.5in}}
(c)& $p=50$, $n=500$, $\bnu \sim\mathcal{GAL} _q (\mathbf 1_q, \mathbf I_q, 10).$\\
\end{tabular}
\begin{tabular}{lcrp{1.5in}}
(d)& $p=100$, $n=1000$, $\bnu \sim\mathcal{GAL} _q (\mathbf 1_q, \mathbf I_q, 10).$ \\
\end{tabular}
  \caption{The kernel density estimator of the asymptotic distribution as given in Theorem \ref{th2} for $c=0.1$.}
  \label{fig1}
\end{figure}


\begin{figure}
  \includegraphics[width=8cm]{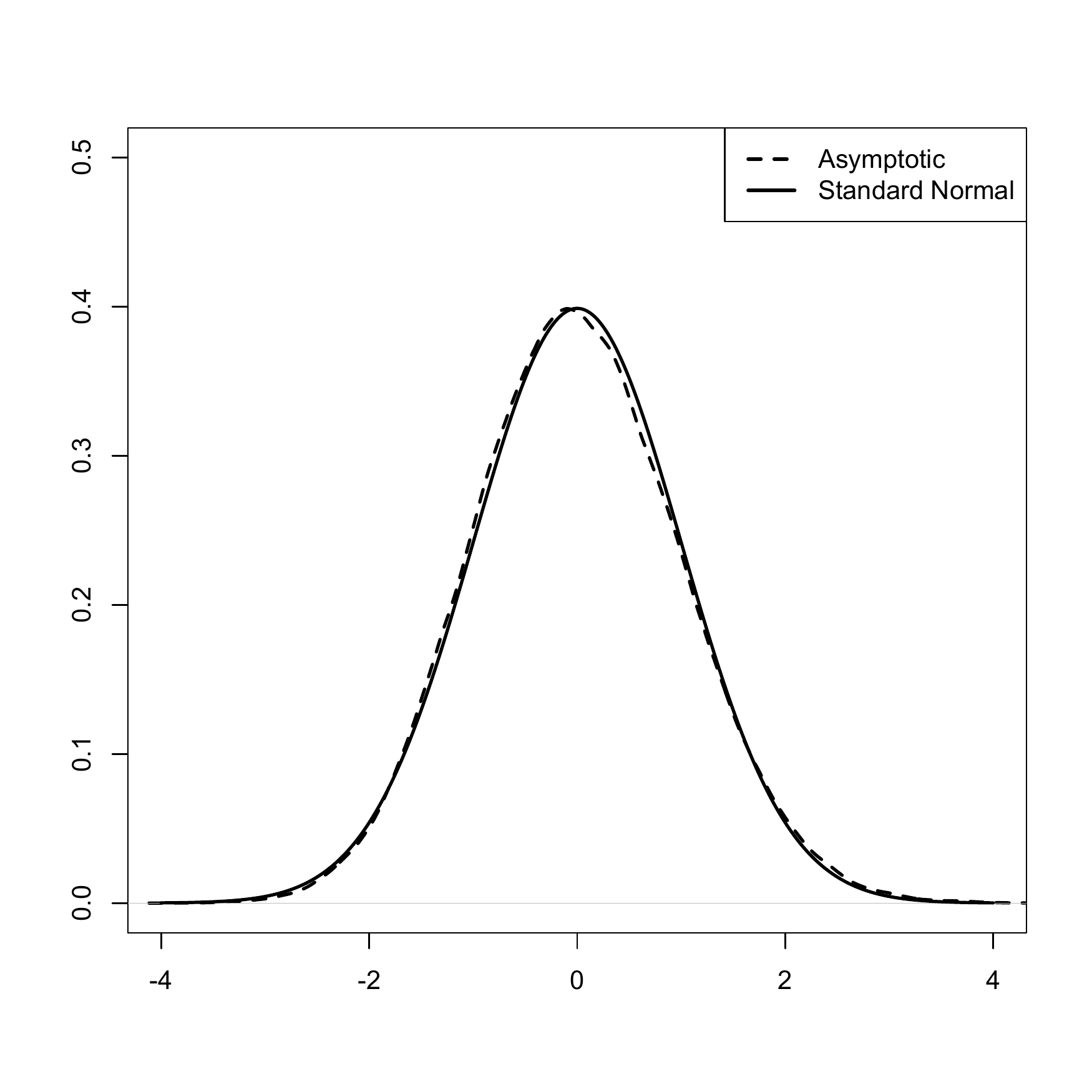}
  \includegraphics[width=8cm]{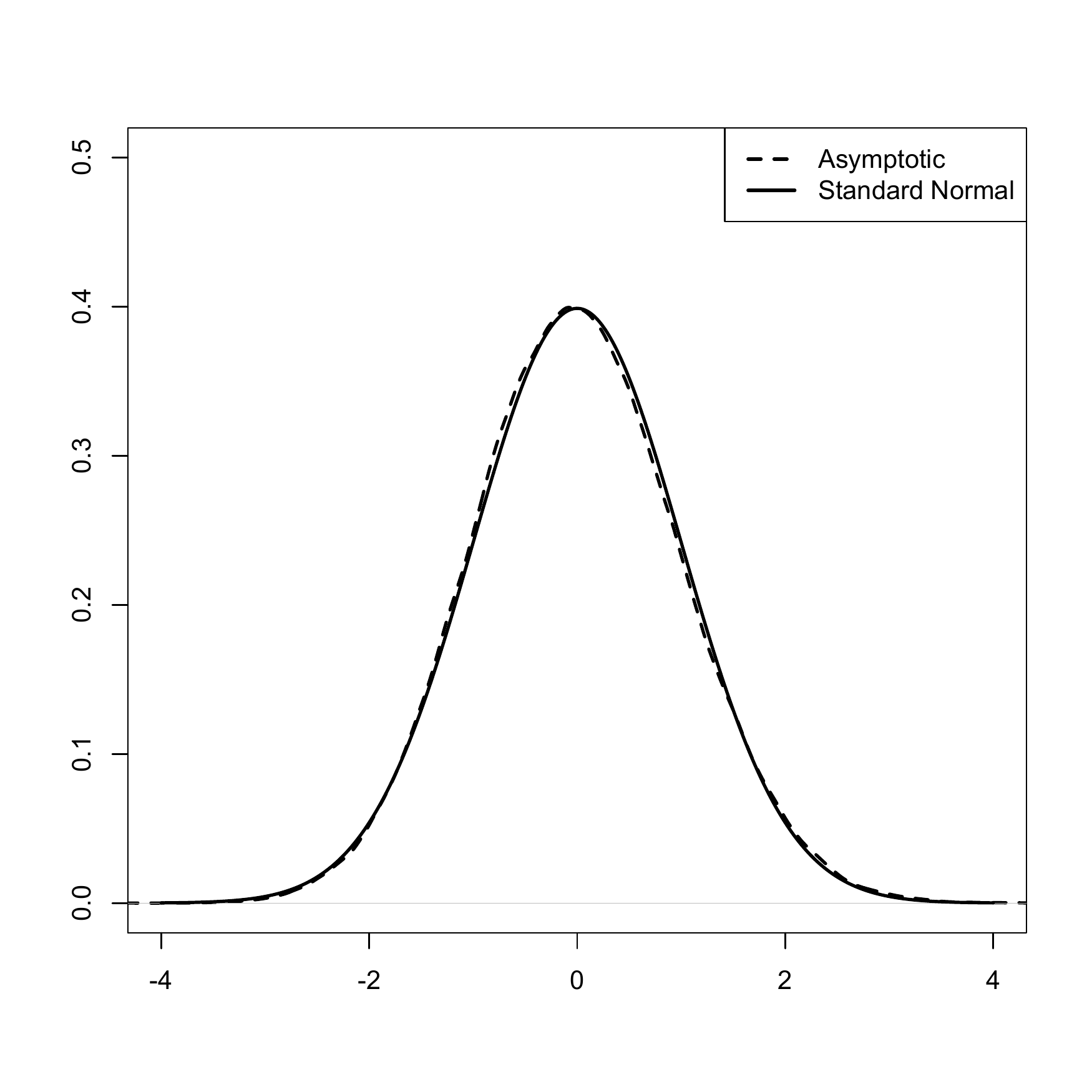}
  \begin{tabular}{lcrp{1.5in}}
(a)& $p=250$, $n=500$, $\bnu \sim\mathcal{TN} _q (\mathbf 0, \mathbf I_q) $.\\
\end{tabular}
\begin{tabular}{lcrp{1.5in}}
(b)&  $p=500$, $n=1000$, $\bnu \sim\mathcal{TN} _q (\mathbf 0, \mathbf I_q) $.\\
\end{tabular}
  \includegraphics[width=8cm]{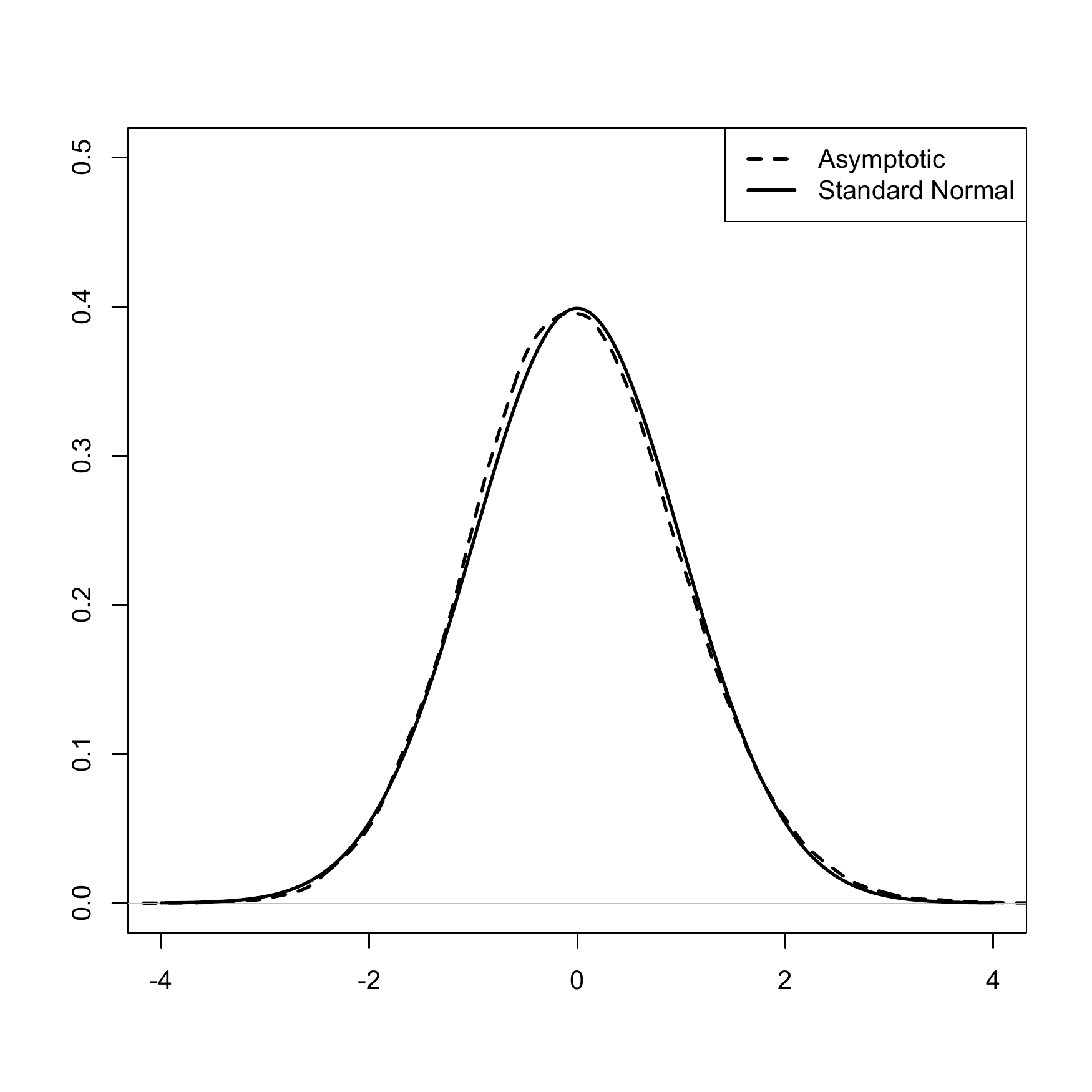}
  \includegraphics[width=8cm]{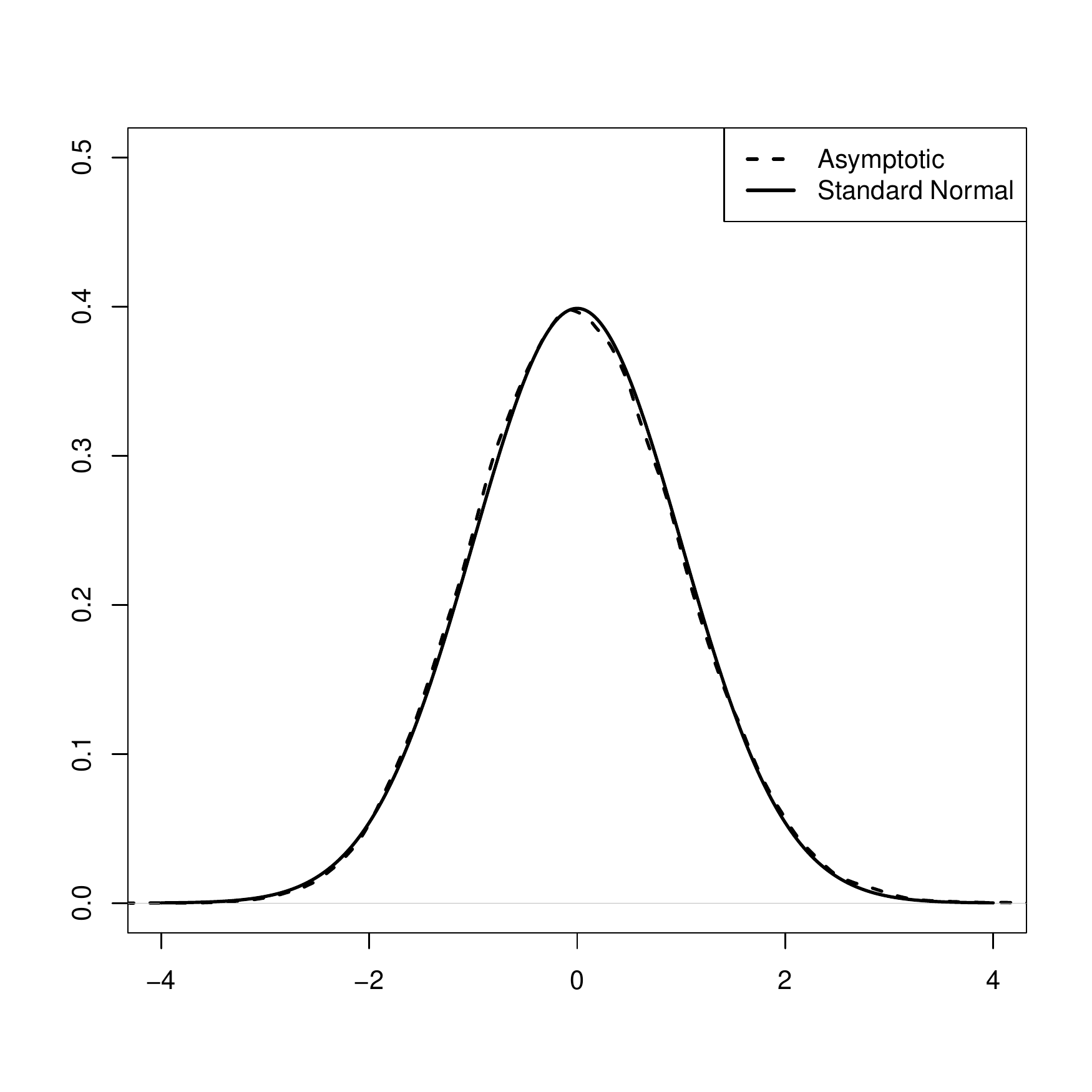}
  \begin{tabular}{lcrp{1.5in}}
(c)& $p=250$, $n=500$, $\bnu \sim\mathcal{GAL} _q (\mathbf 1_q, \mathbf I_q, 10).$\\
\end{tabular}
\begin{tabular}{lcrp{1.5in}}
(d)& $p=500$, $n=1000$, $\bnu \sim\mathcal{GAL} _q (\mathbf 1_q, \mathbf I_q, 10).$ \\
\end{tabular}
  \caption{The kernel density estimator of the asymptotic distribution as given in Theorem \ref{th2} for $c=0.5$.}
  \label{fig2}
\end{figure}


\begin{figure}
  \includegraphics[width=8cm]{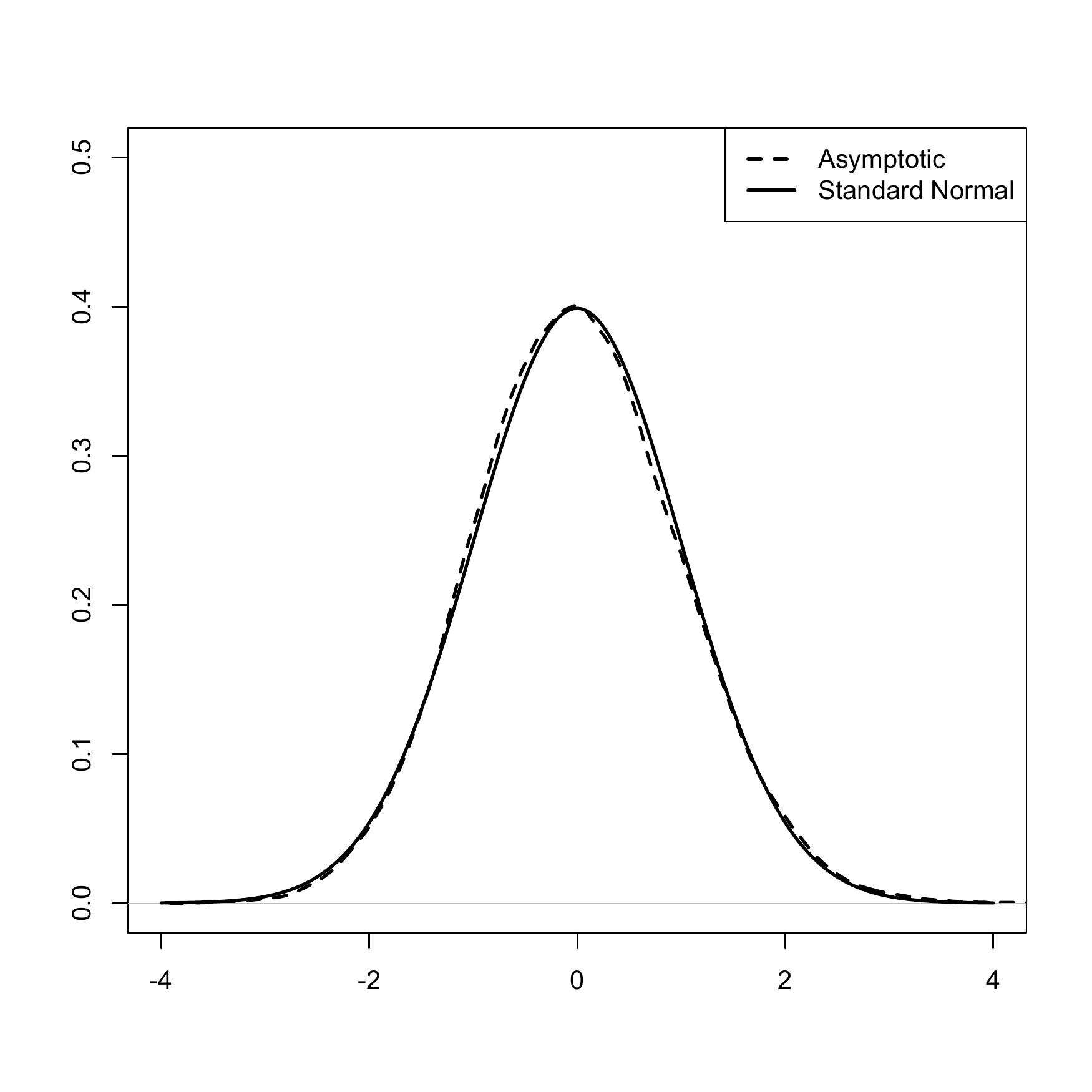}
  \includegraphics[width=8cm]{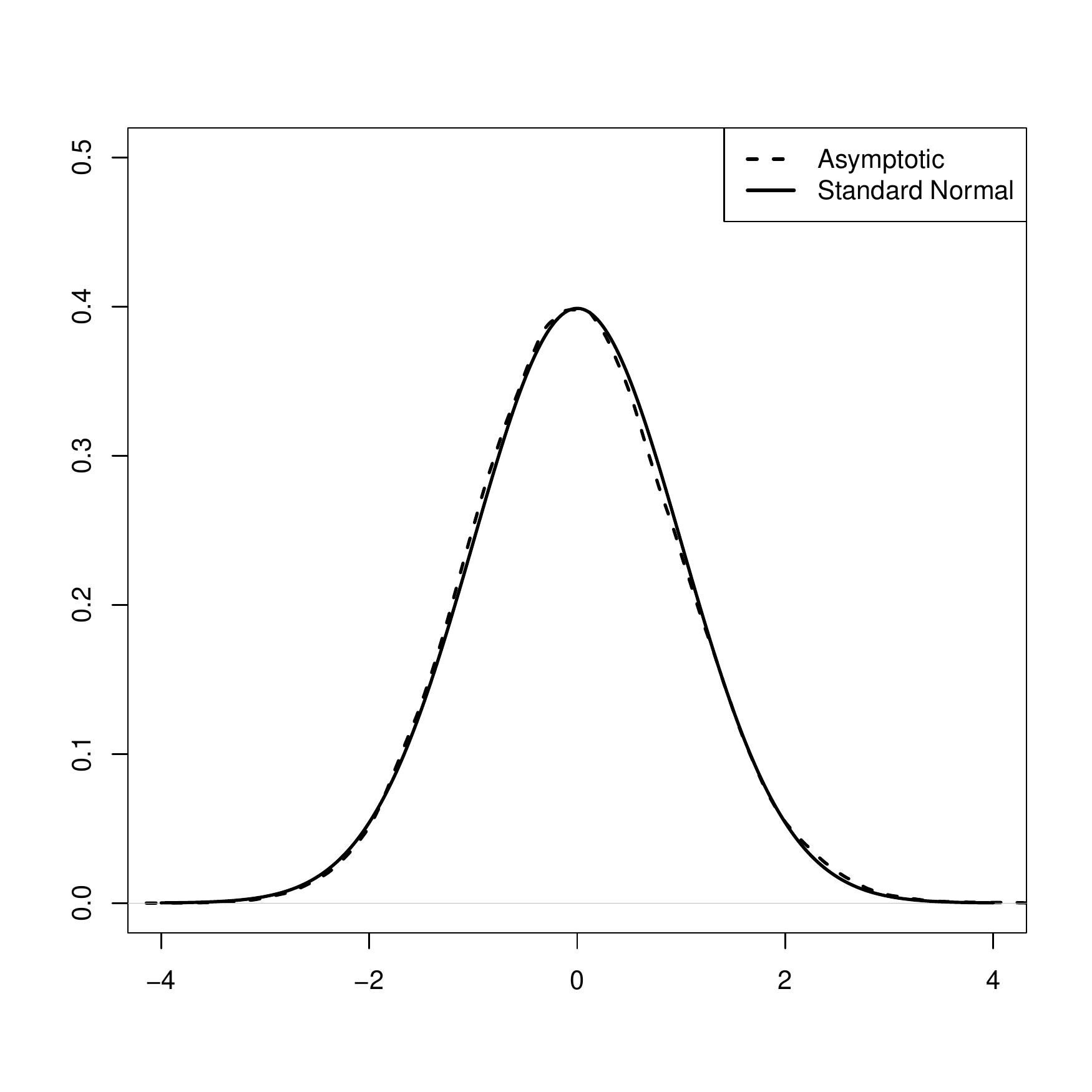}
  \begin{tabular}{lcrp{1.5in}}
(a)& $p=400$, $n=500$, $\bnu \sim\mathcal{TN} _q (\mathbf 0, \mathbf I_q) $.\\
\end{tabular}
\begin{tabular}{lcrp{1.5in}}
(b)&  $p=800$, $n=1000$, $\bnu \sim\mathcal{TN} _q (\mathbf 0, \mathbf I_q) $.\\
\end{tabular}
  \includegraphics[width=8cm]{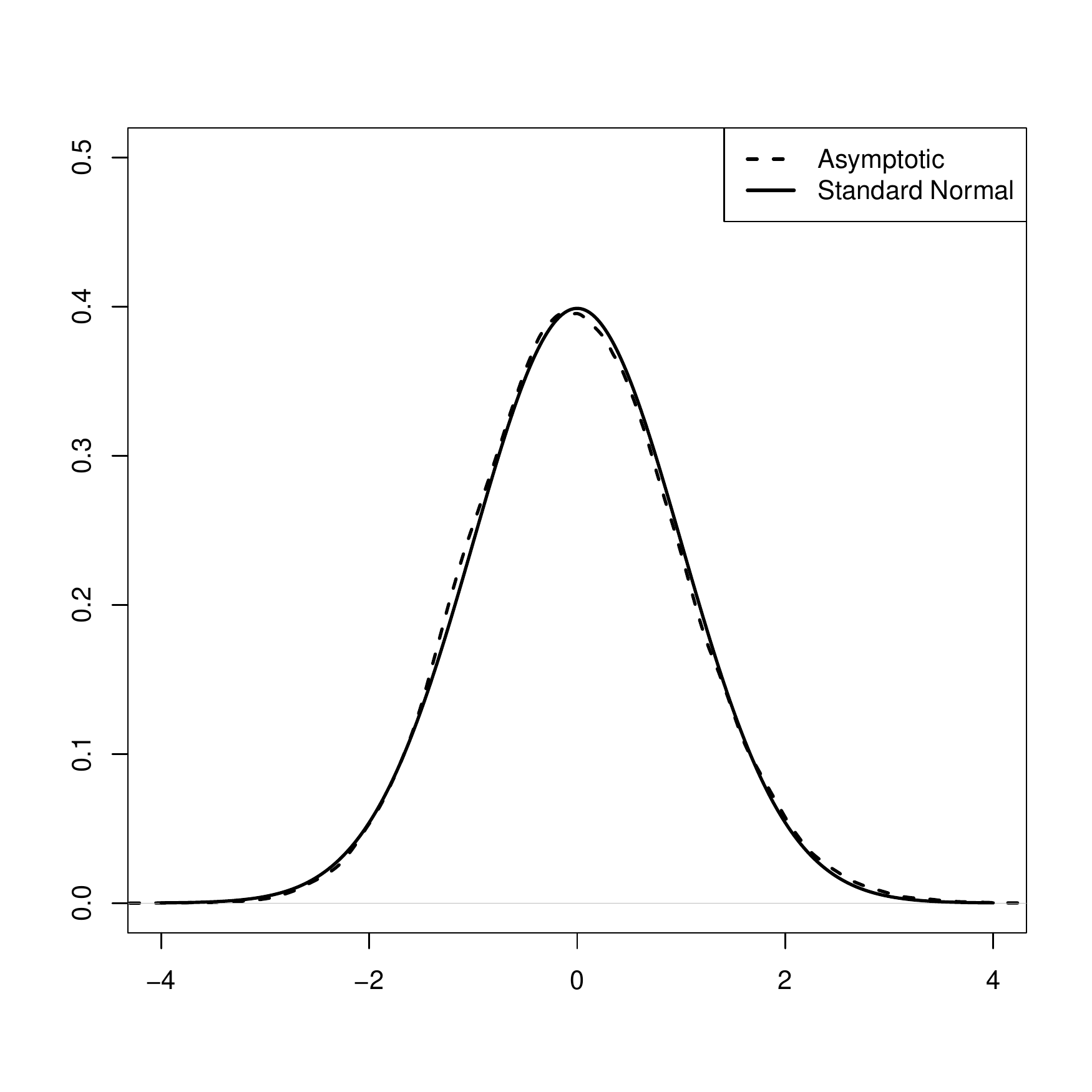}
  \includegraphics[width=8cm]{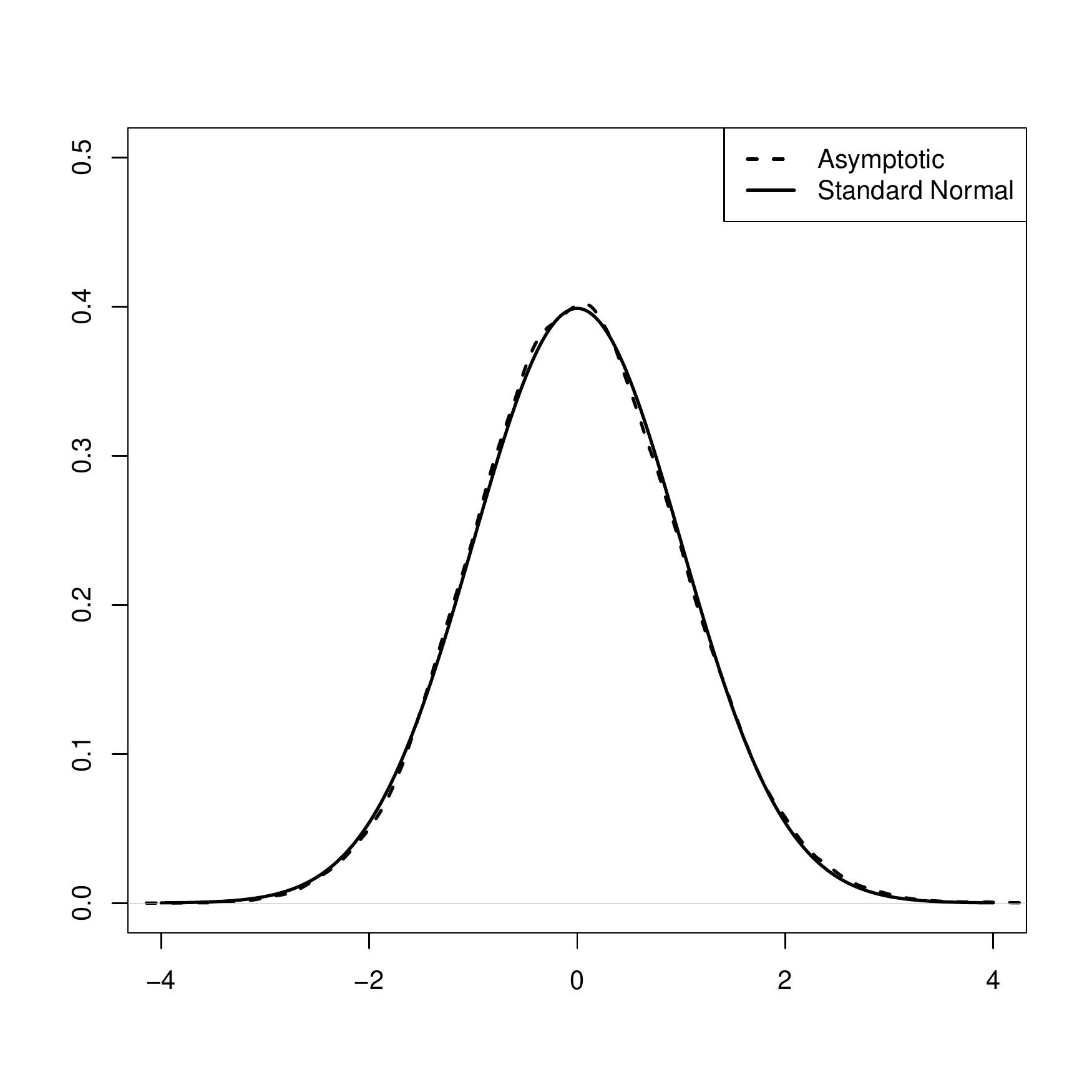}
  \begin{tabular}{lcrp{1.5in}}
(c)& $p=400$, $n=500$, $\bnu \sim\mathcal{GAL} _q (\mathbf 1_q, \mathbf I_q, 10).$\\
\end{tabular}
\begin{tabular}{lcrp{1.5in}}
(d)& $p=800$, $n=1000$, $\bnu \sim\mathcal{GAL} _q (\mathbf 1_q, \mathbf I_q, 10).$ \\
\end{tabular}
  \caption{The kernel density estimator of the asymptotic distribution as given in Theorem \ref{th2} for $c=0.8$.}
  \label{fig3}
\end{figure}


\begin{figure}
  \includegraphics[width=8cm]{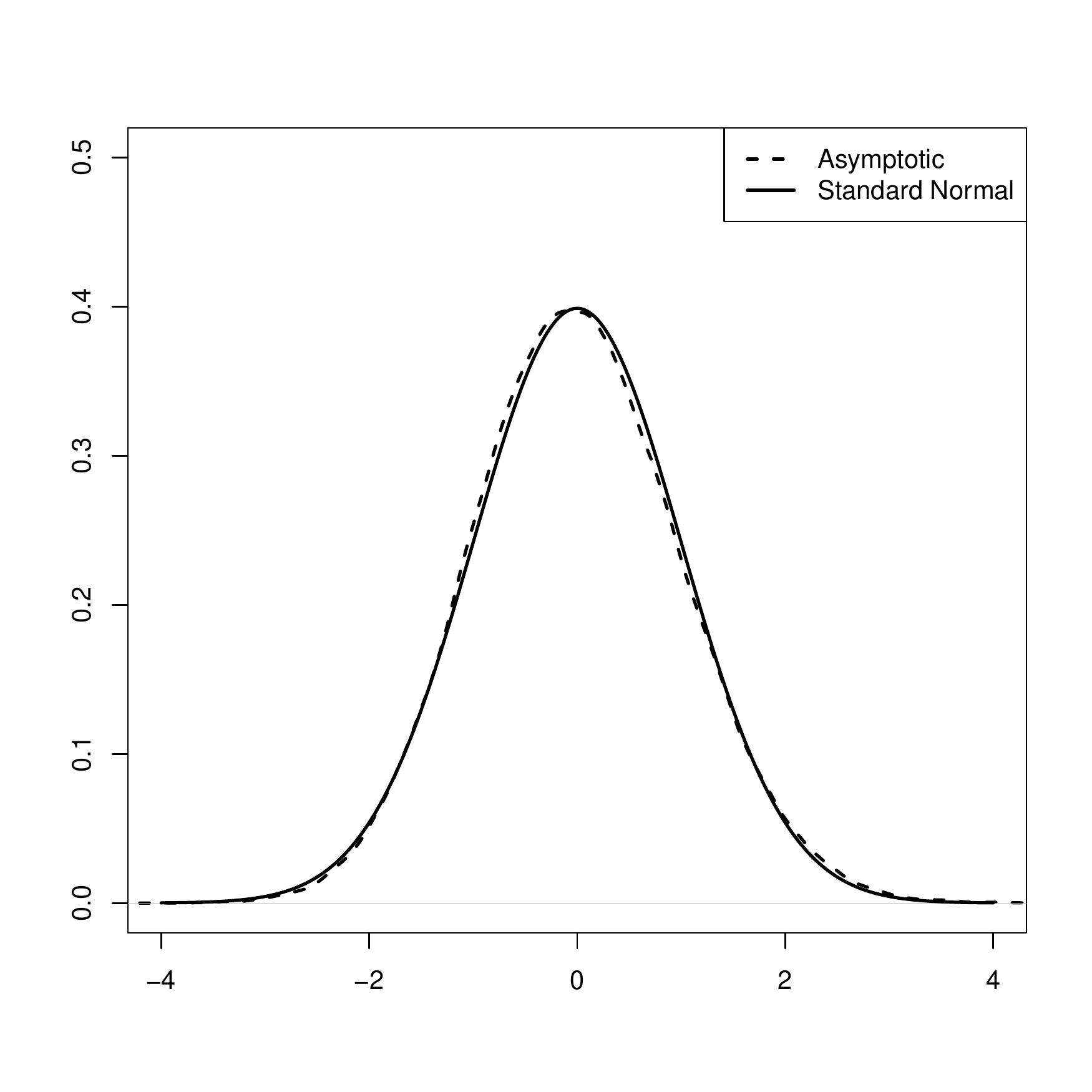}
  \includegraphics[width=8cm]{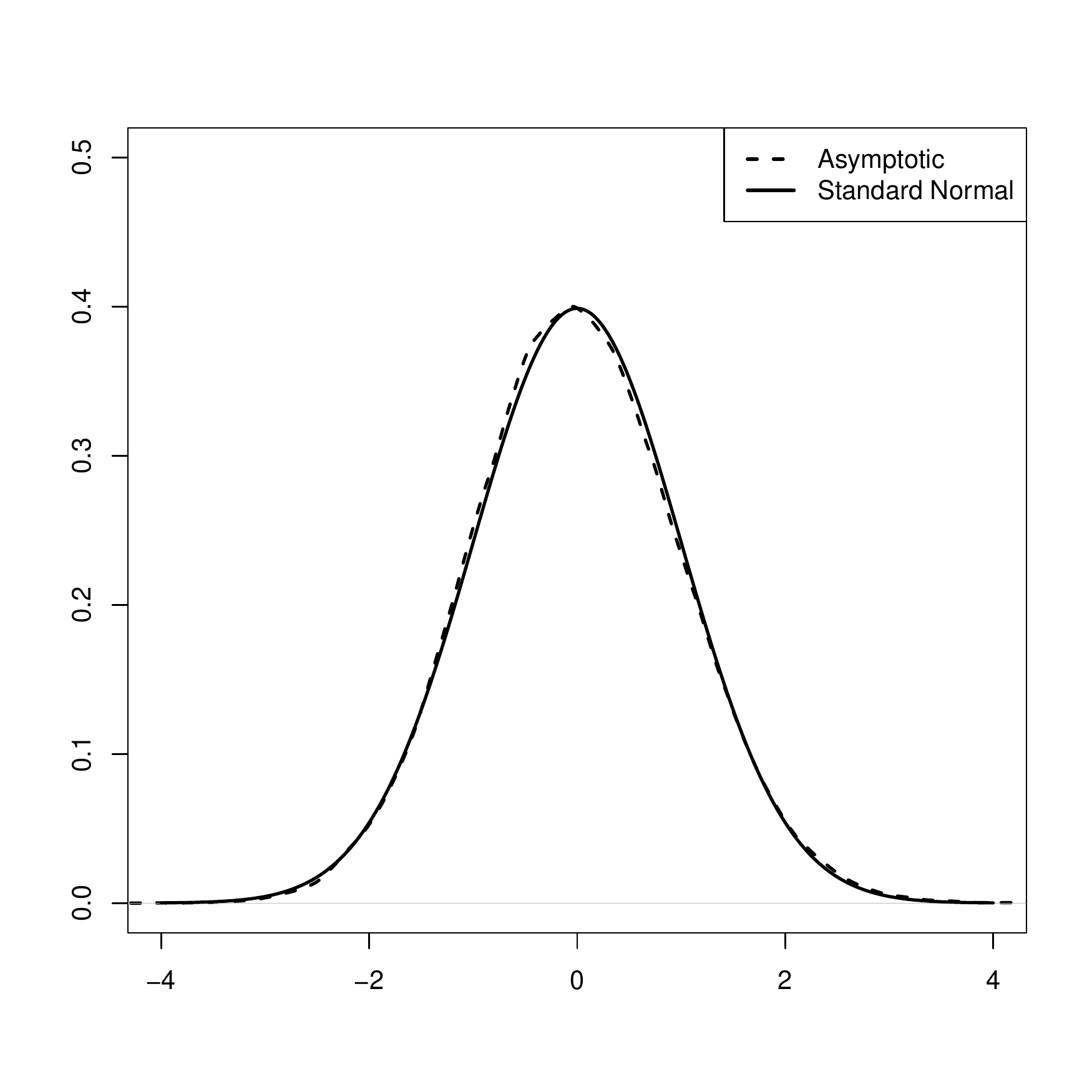}
  \begin{tabular}{lcrp{1.5in}}
(a)& $p=475$, $n=500$, $\bnu \sim\mathcal{TN} _q (\mathbf 0, \mathbf I_q) $.\\
\end{tabular}
\begin{tabular}{lcrp{1.5in}}
(b)&  $p=950$, $n=1000$, $\bnu \sim\mathcal{TN} _q (\mathbf 0, \mathbf I_q) $.\\
\end{tabular}
  \includegraphics[width=8cm]{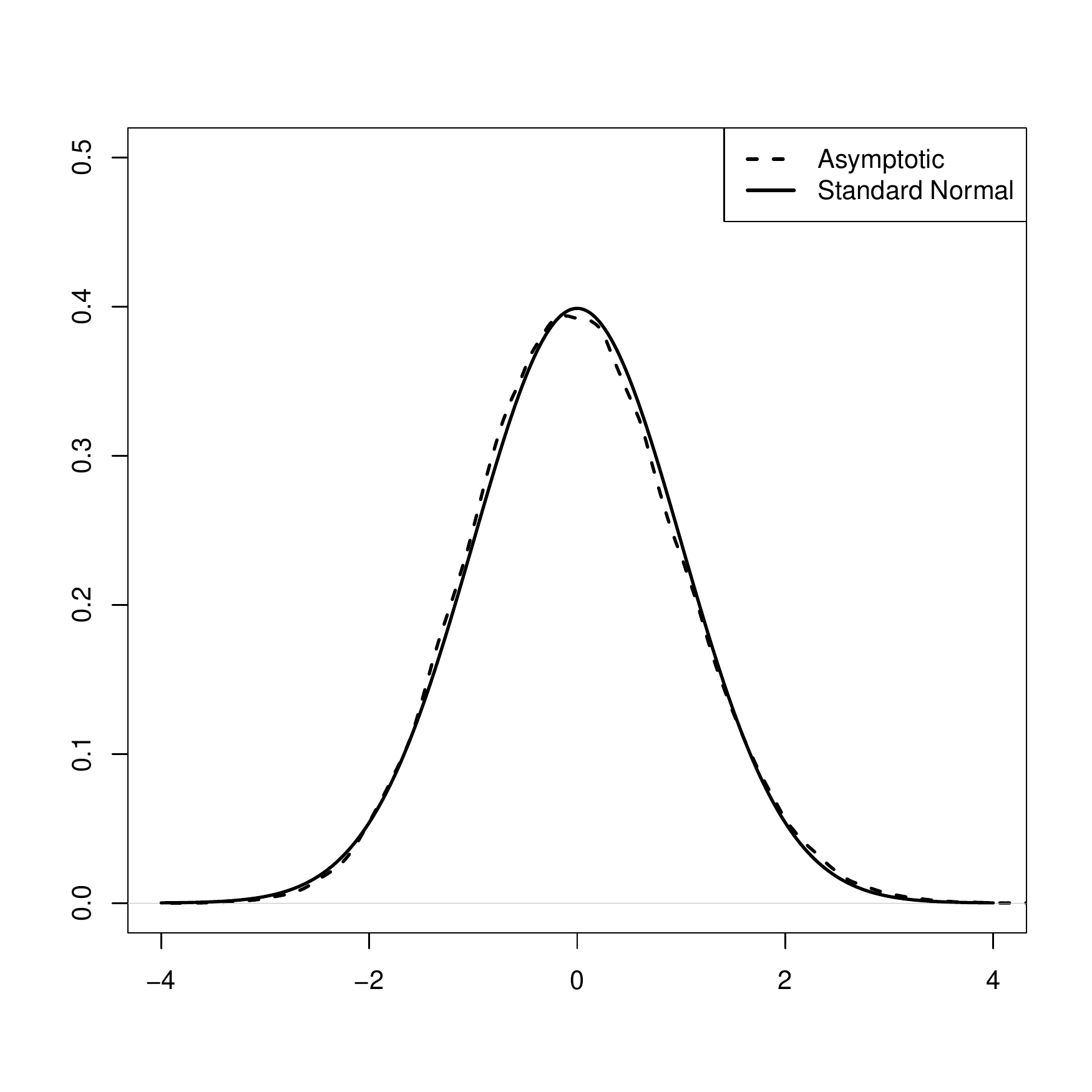}
  \includegraphics[width=8cm]{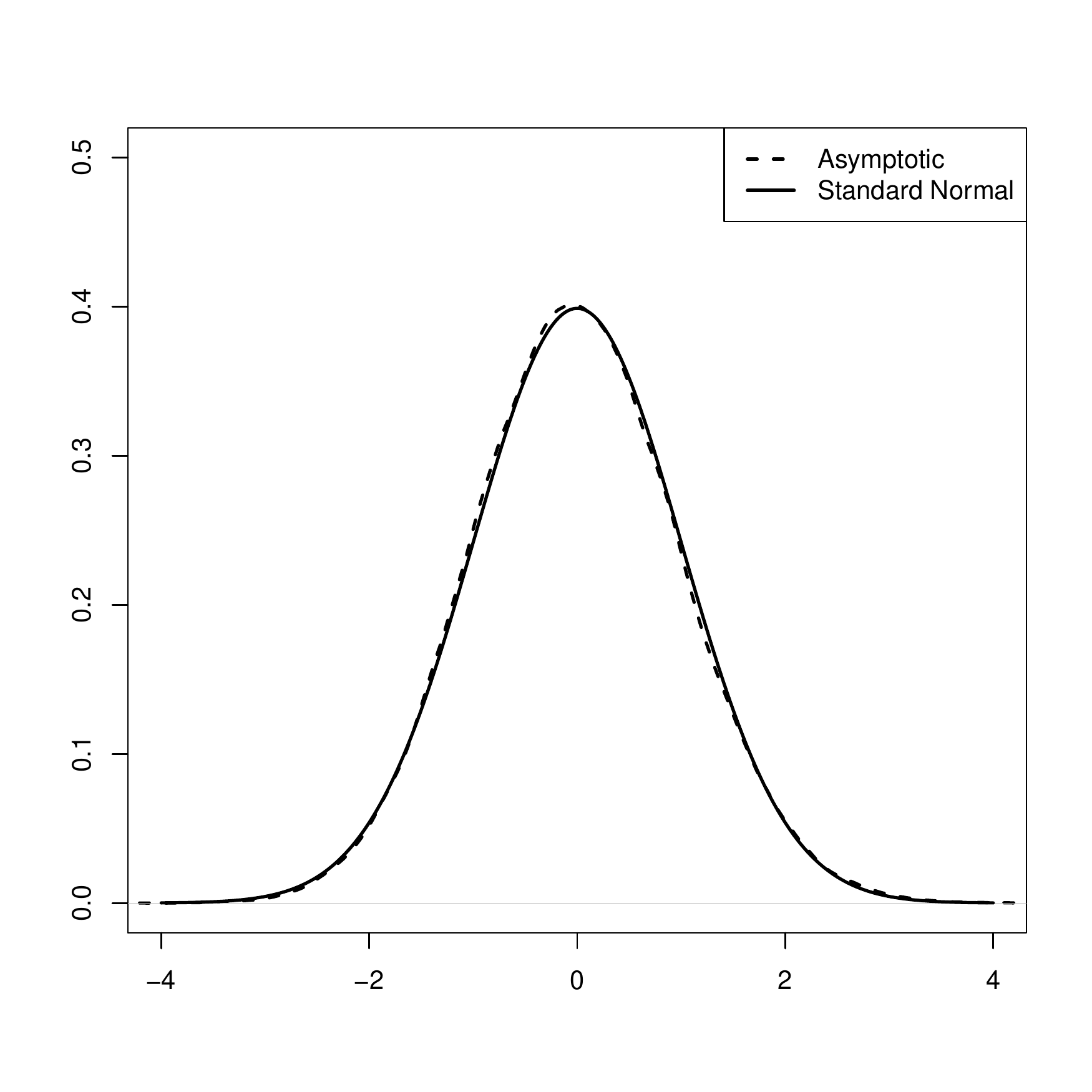}
  \begin{tabular}{lcrp{1.5in}}
(c)& $p=475$, $n=500$, $\bnu \sim\mathcal{GAL} _q (\mathbf 1_q, \mathbf I_q, 10).$\\
\end{tabular}
\begin{tabular}{lcrp{1.5in}}
(d)& $p=950$, $n=1000$, $\bnu \sim\mathcal{GAL} _q (\mathbf 1_q, \mathbf I_q, 10).$ \\
\end{tabular}
  \caption{The kernel density estimator of the asymptotic distribution as given in Theorem \ref{th2} for $c=0.95$.}
  \label{fig4}
\end{figure}


\begin{figure}
  \includegraphics[width=8cm]{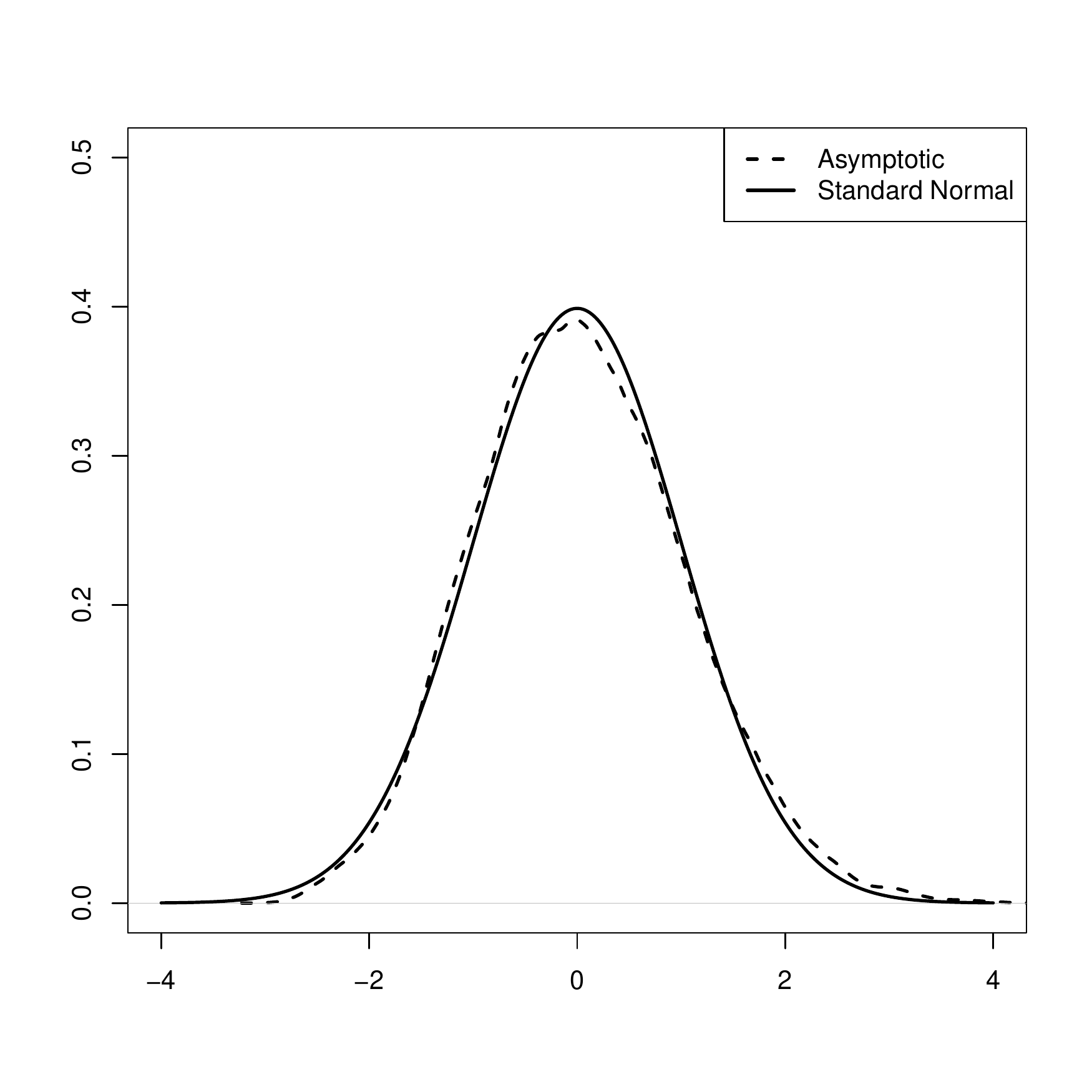}
  \includegraphics[width=8cm]{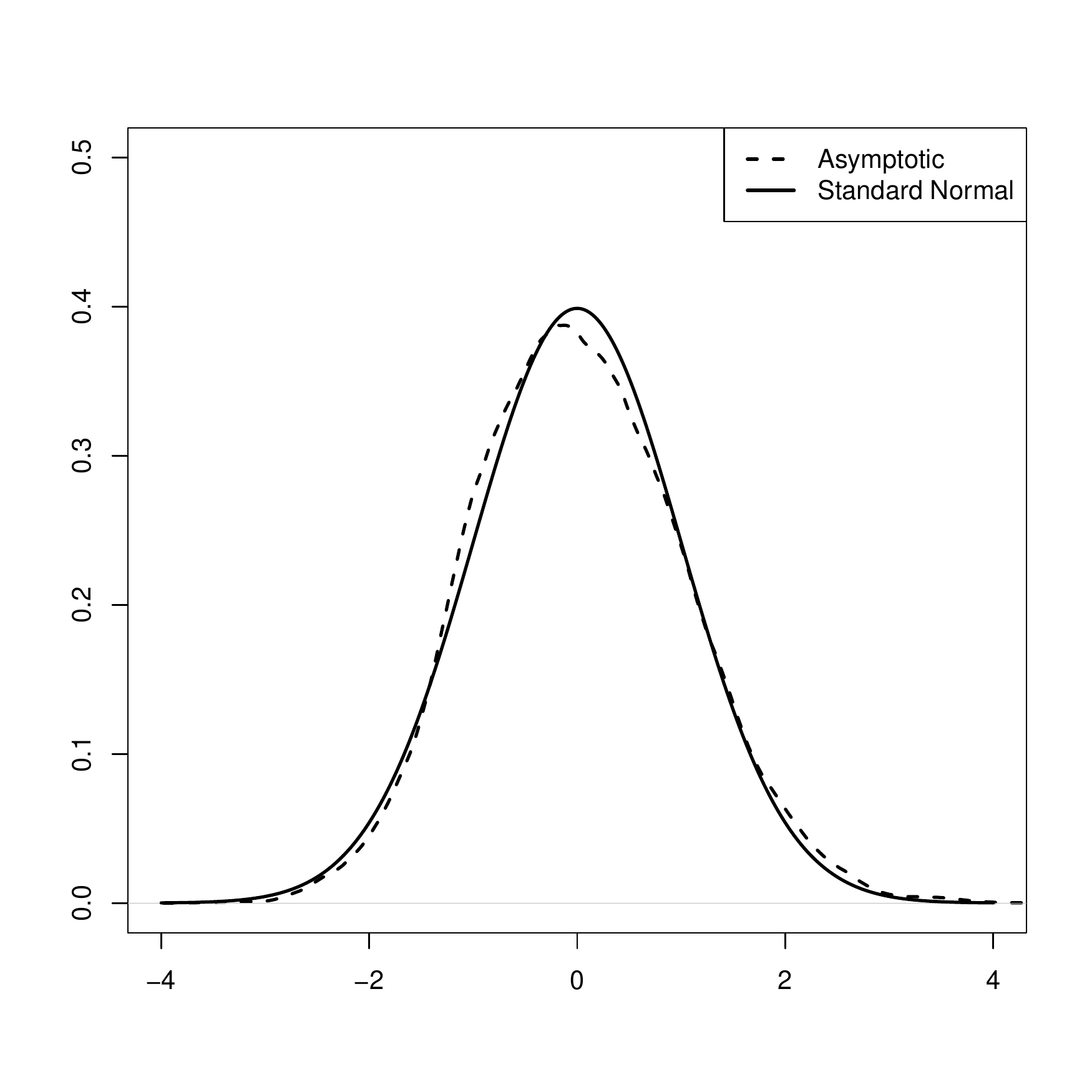}
  \begin{tabular}{lcrp{1.5in}}
(a)& $p=50$, $n=500$, $\bnu \sim\mathcal{TN} _q (\mathbf 0, \mathbf I_q) $.\\
\end{tabular}
\begin{tabular}{lcrp{1.5in}}
(b)&  $p=100$, $n=1000$, $\bnu \sim\mathcal{TN} _q (\mathbf 0, \mathbf I_q) $.\\
\end{tabular}
  \includegraphics[width=8cm]{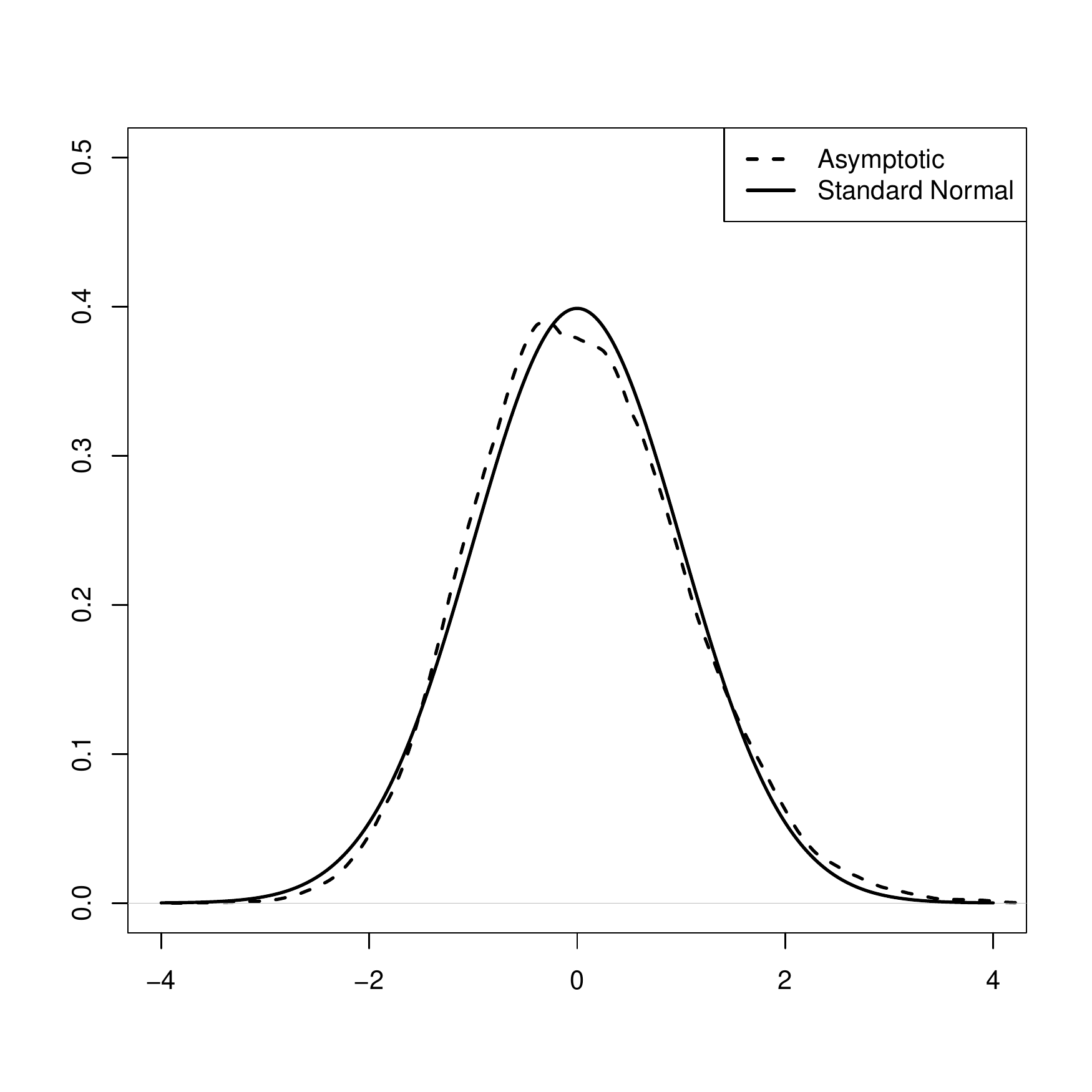}
  \includegraphics[width=8cm]{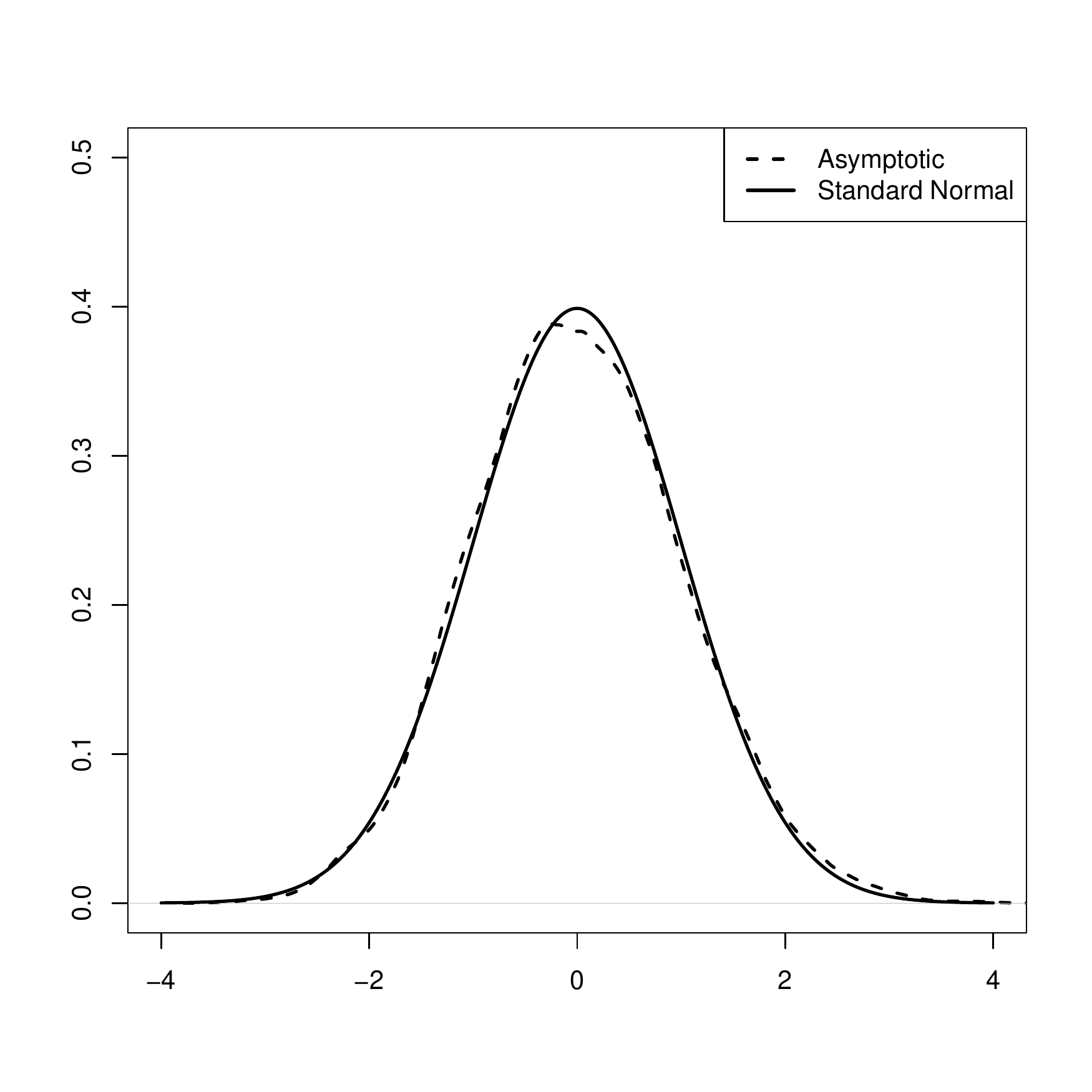}
  \begin{tabular}{lcrp{1.5in}}
(c)& $p=50$, $n=500$, $\bnu \sim\mathcal{GAL} _q (\mathbf 1_q, \mathbf I_q, 10).$\\
\end{tabular}
\begin{tabular}{lcrp{1.5in}}
(d)& $p=100$, $n=1000$, $\bnu \sim\mathcal{GAL} _q (\mathbf 1_q, \mathbf I_q, 10).$ \\
\end{tabular}
  \caption{The kernel density estimator of the asymptotic distribution as given in Theorem \ref{th4} for $c=0.1$.}
  \label{fig5}
\end{figure}


\begin{figure}
  \includegraphics[width=8cm]{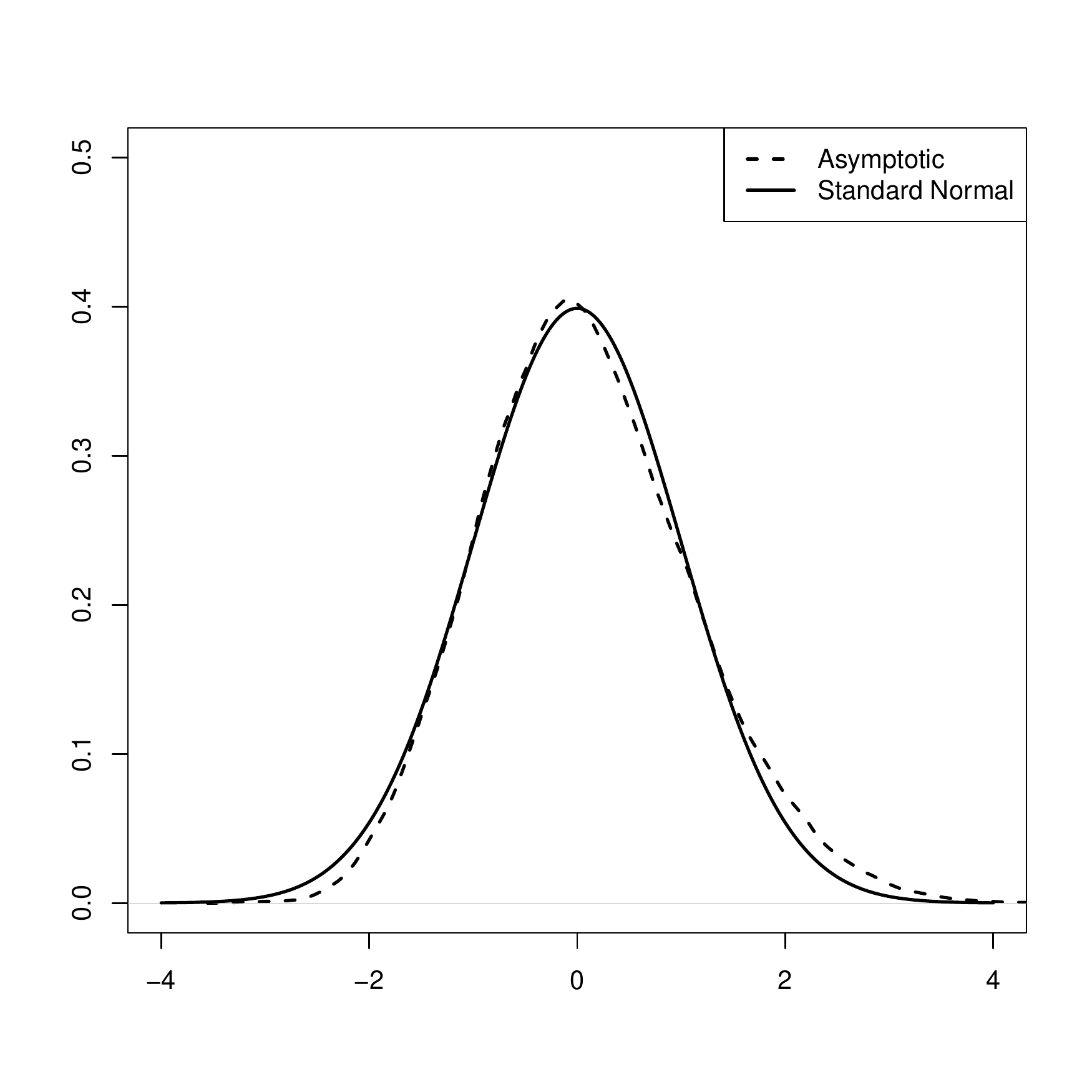}
  \includegraphics[width=8cm]{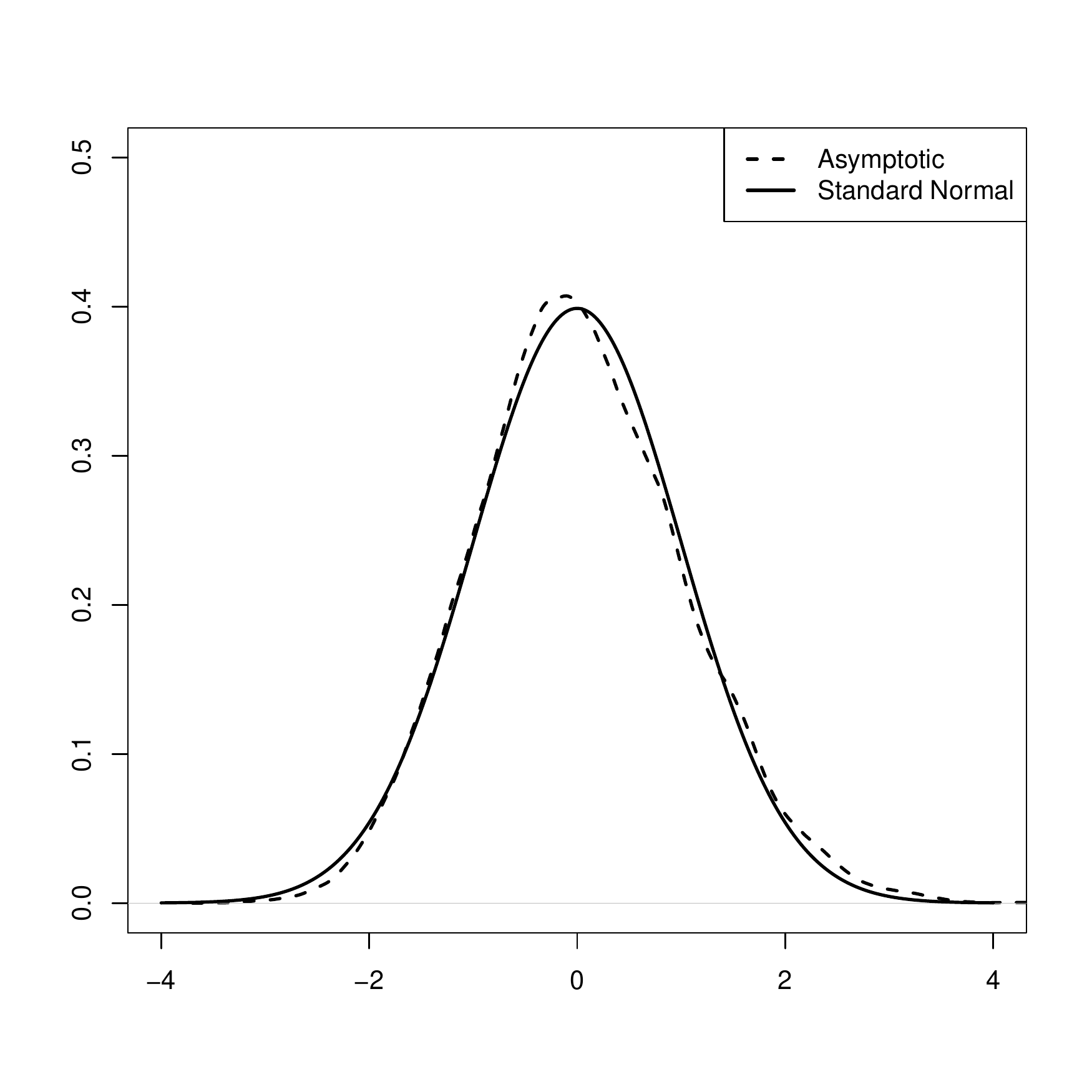}
  \begin{tabular}{lcrp{1.5in}}
(a)& $p=250$, $n=500$, $\bnu \sim\mathcal{TN} _q (\mathbf 0, \mathbf I_q) $.\\
\end{tabular}
\begin{tabular}{lcrp{1.5in}}
(b)&  $p=500$, $n=1000$, $\bnu \sim\mathcal{TN} _q (\mathbf 0, \mathbf I_q) $.\\
\end{tabular}
  \includegraphics[width=8cm]{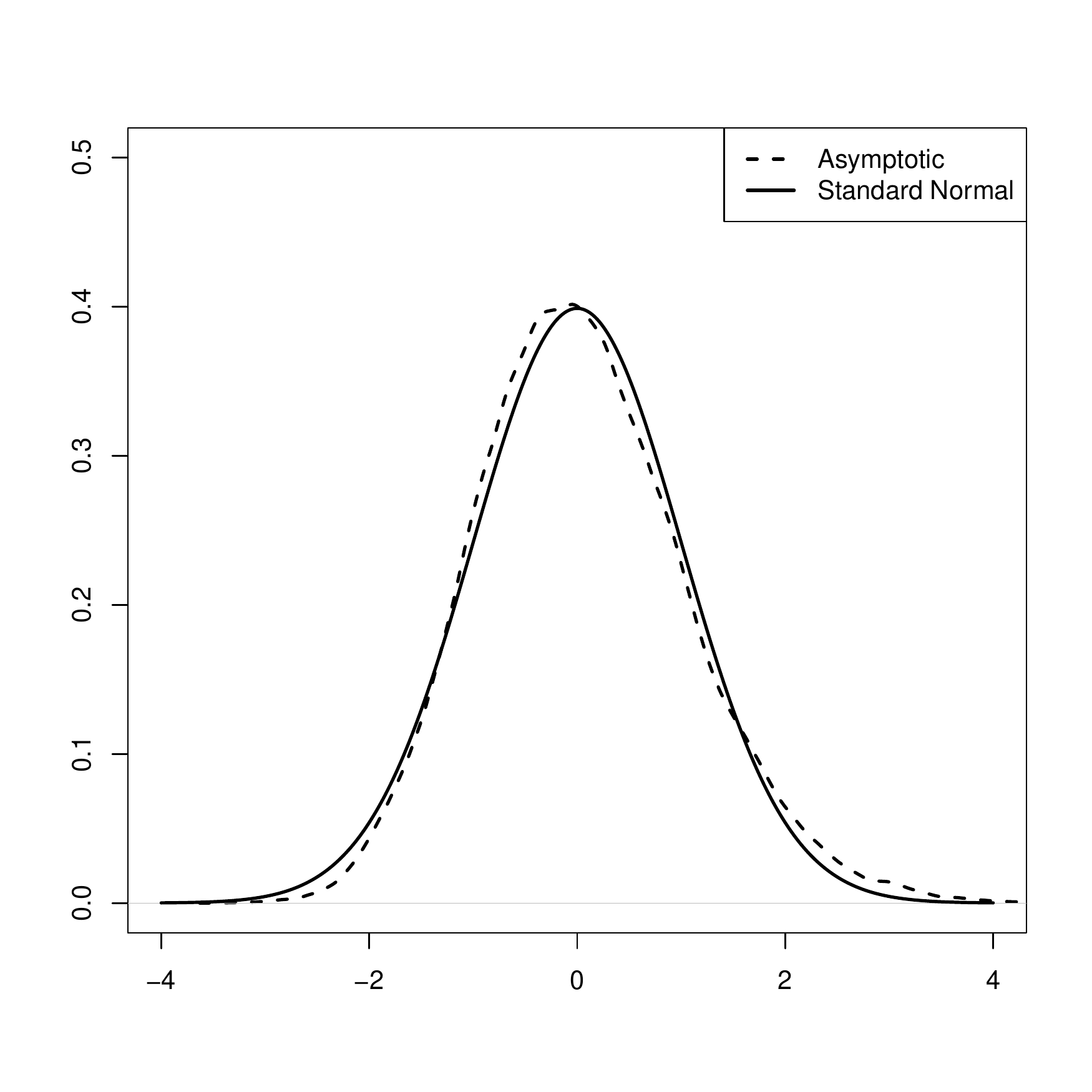}
  \includegraphics[width=8cm]{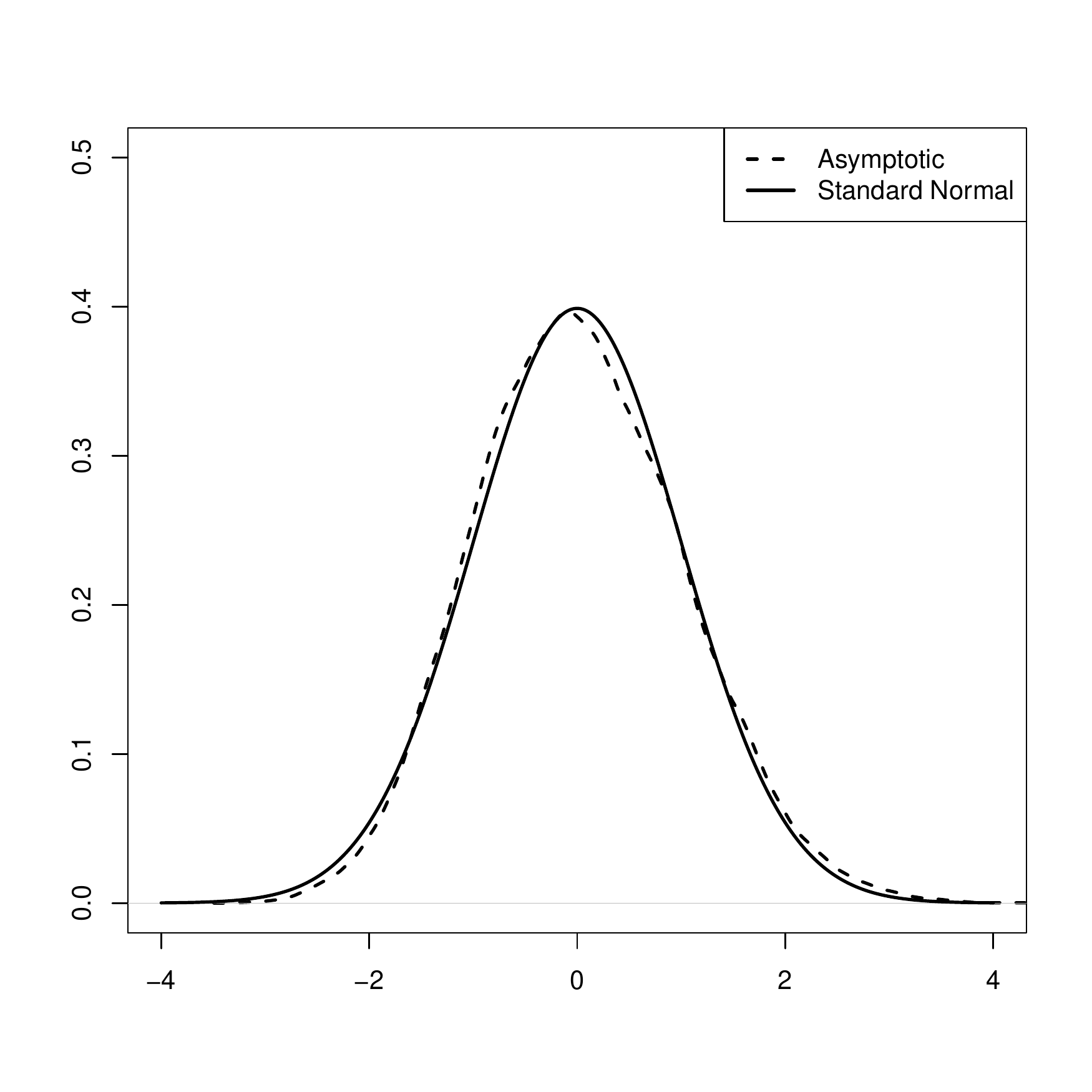}
  \begin{tabular}{lcrp{1.5in}}
(c)& $p=250$, $n=500$, $\bnu \sim\mathcal{GAL} _q (\mathbf 1_q, \mathbf I_q, 10).$\\
\end{tabular}
\begin{tabular}{lcrp{1.5in}}
(d)& $p=500$, $n=1000$, $\bnu \sim\mathcal{GAL} _q (\mathbf 1_q, \mathbf I_q, 10).$ \\
\end{tabular}
  \caption{The kernel density estimator of the asymptotic distribution as given in Theorem \ref{th4} for $c=0.5$.}
  \label{fig6}
\end{figure}


\begin{figure}
  \includegraphics[width=8cm]{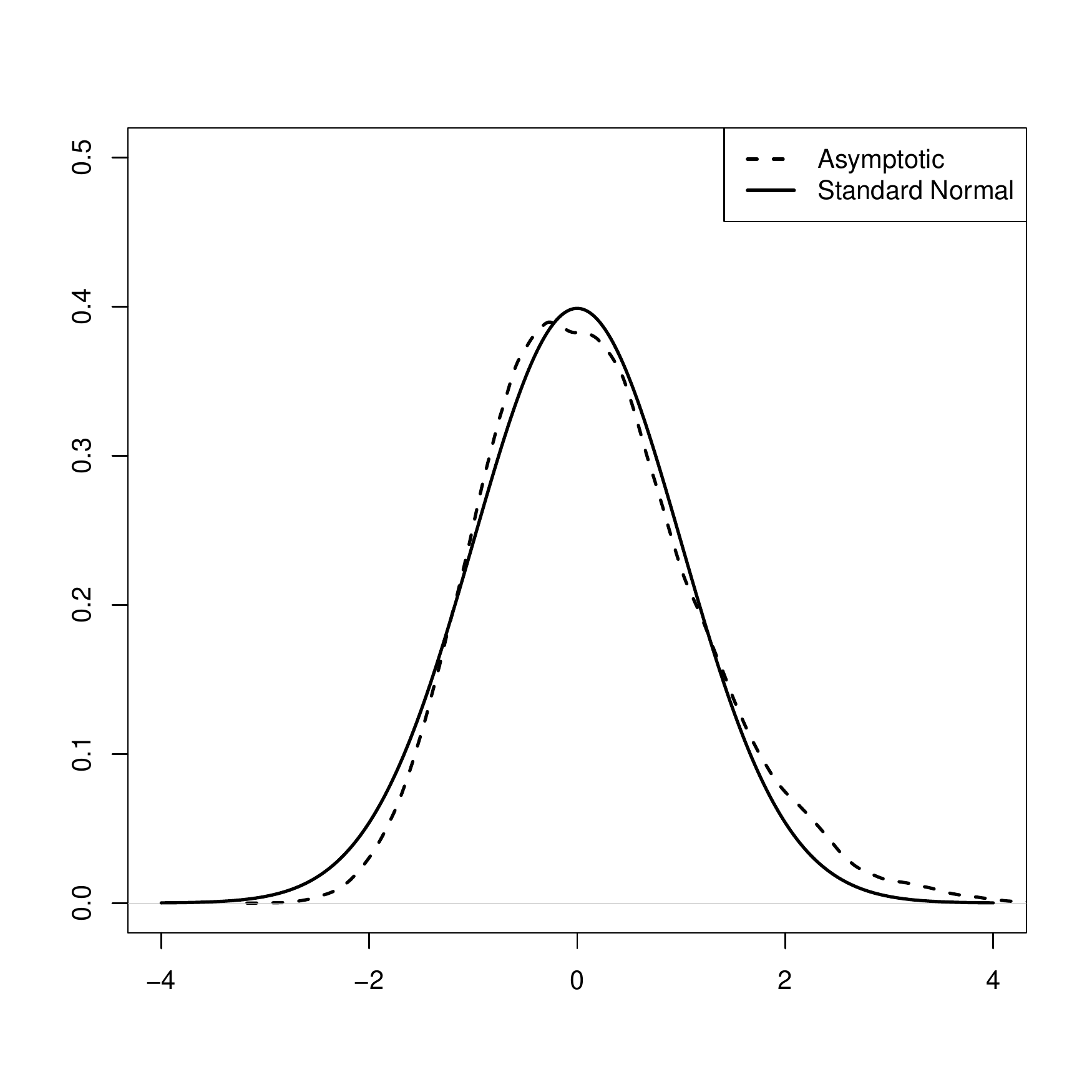}
  \includegraphics[width=8cm]{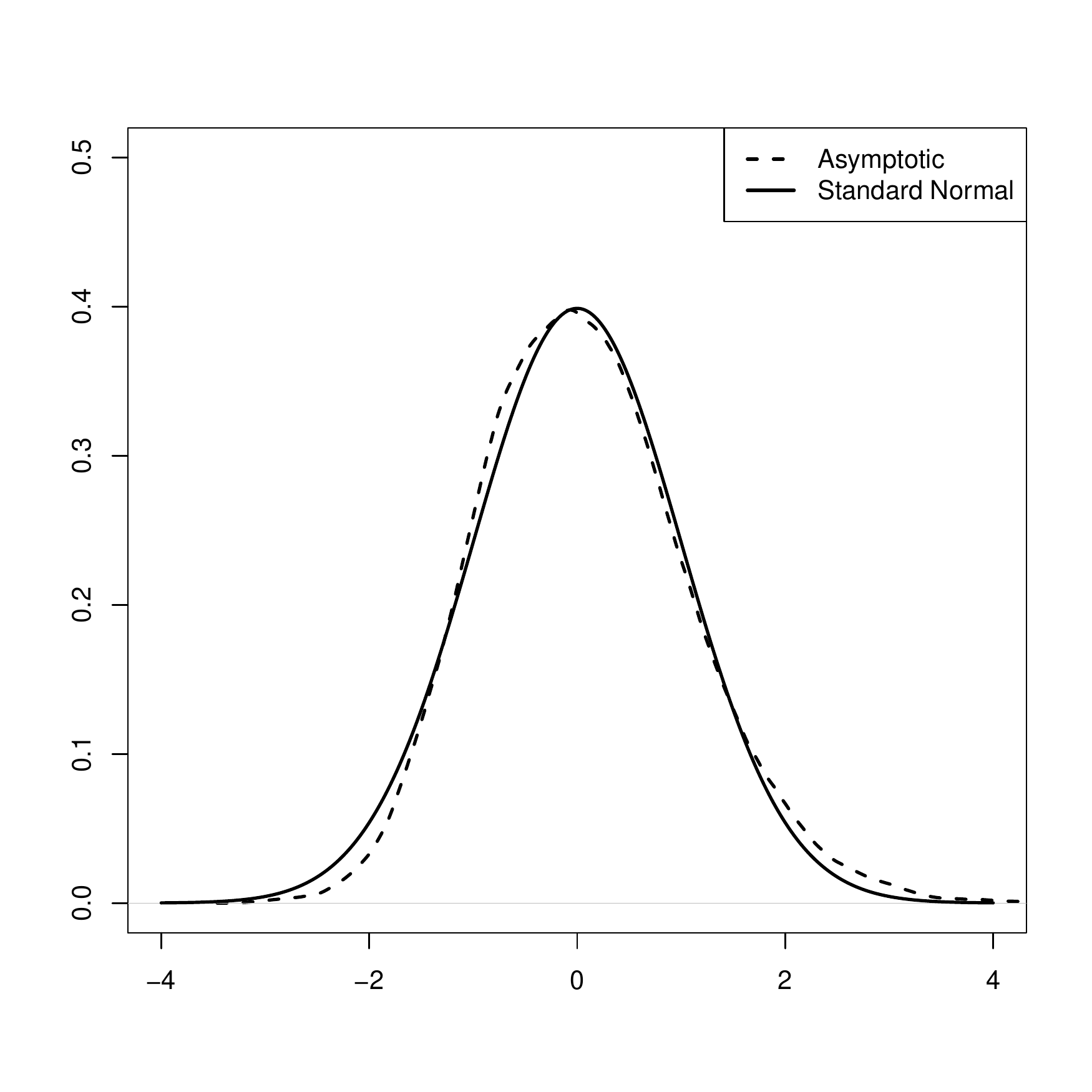}
  \begin{tabular}{lcrp{1.5in}}
(a)& $p=400$, $n=500$, $\bnu \sim\mathcal{TN} _q (\mathbf 0, \mathbf I_q) $.\\
\end{tabular}
\begin{tabular}{lcrp{1.5in}}
(b)&  $p=800$, $n=1000$, $\bnu \sim\mathcal{TN} _q (\mathbf 0, \mathbf I_q) $.\\
\end{tabular}
  \includegraphics[width=8cm]{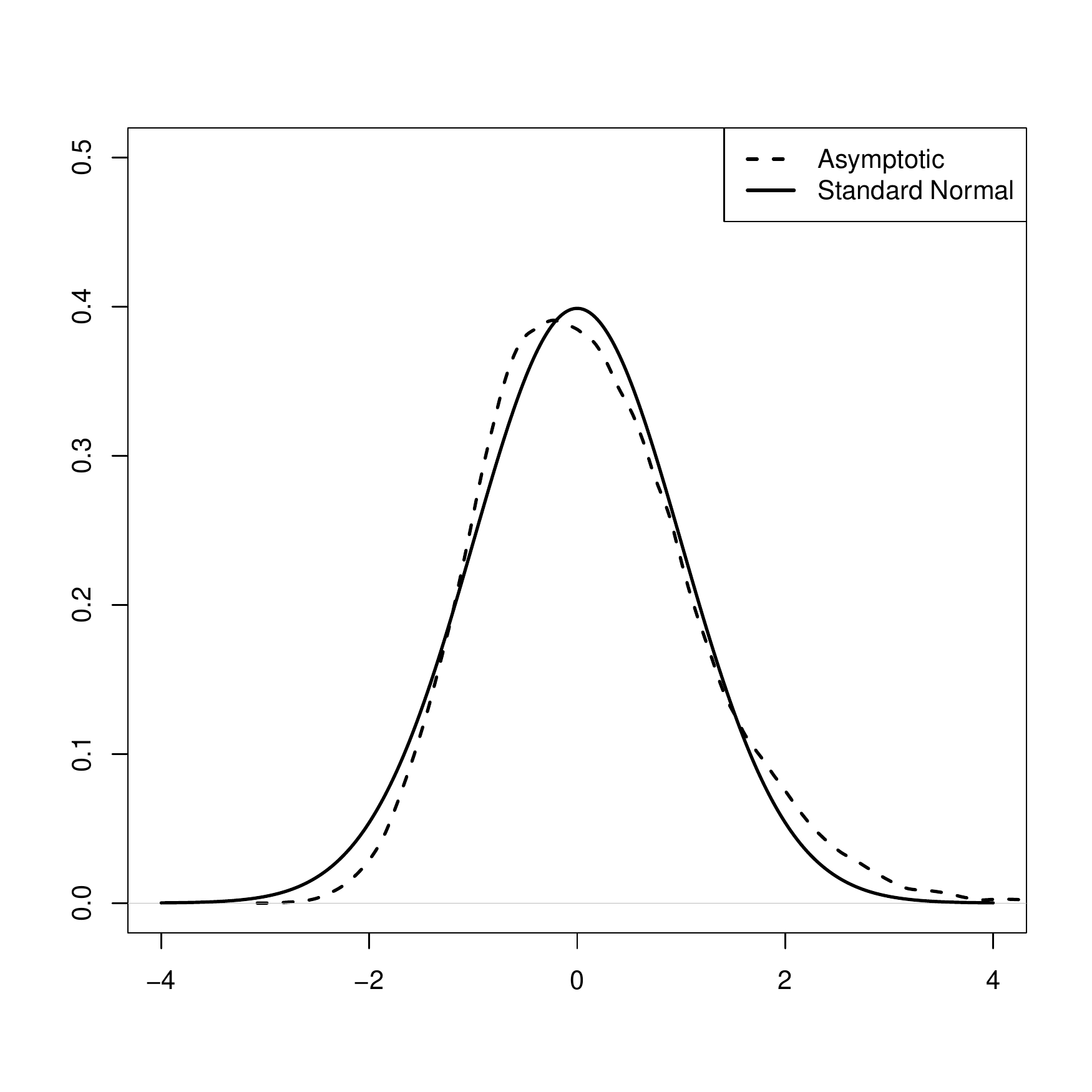}
  \includegraphics[width=8cm]{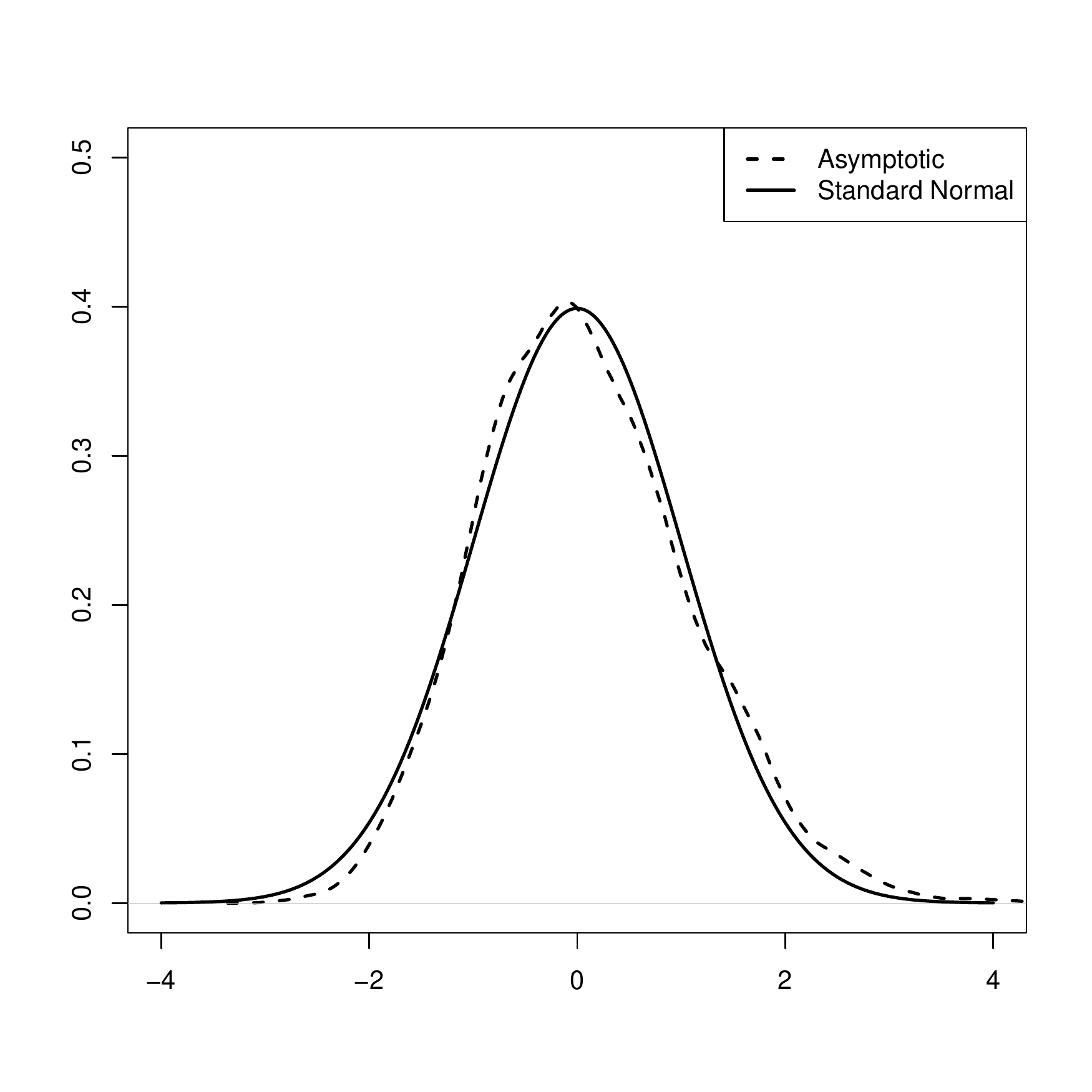}
  \begin{tabular}{lcrp{1.5in}}
(c)& $p=400$, $n=500$, $\bnu \sim\mathcal{GAL} _q (\mathbf 1_q, \mathbf I_q, 10).$\\
\end{tabular}
\begin{tabular}{lcrp{1.5in}}
(d)& $p=800$, $n=1000$, $\bnu \sim\mathcal{GAL} _q (\mathbf 1_q, \mathbf I_q, 10).$ \\
\end{tabular}
  \caption{The kernel density estimator of the asymptotic distribution as given in Theorem \ref{th4} for $c=0.8$.}
  \label{fig7}
\end{figure}


\begin{figure}
  \includegraphics[width=8cm]{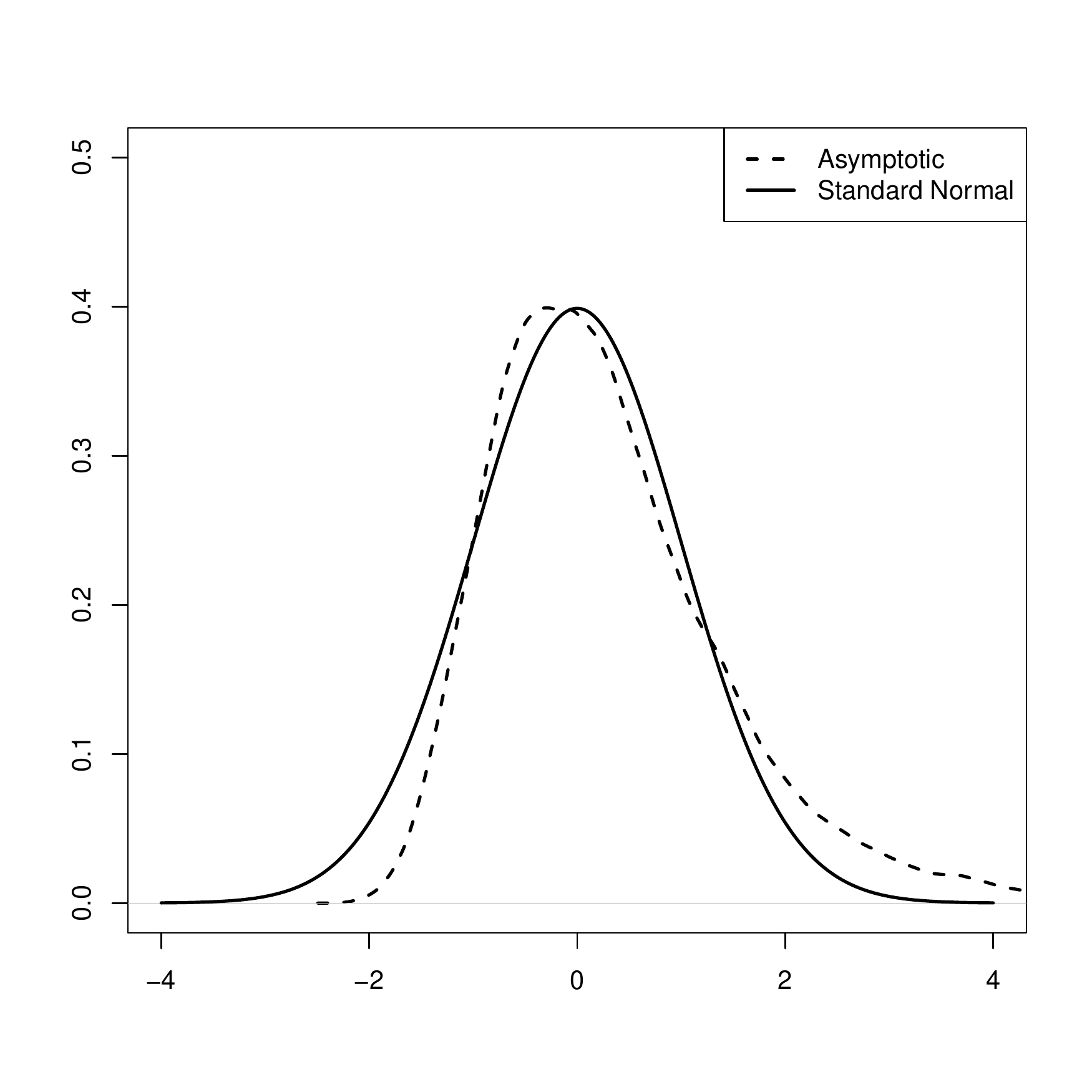}
  \includegraphics[width=8cm]{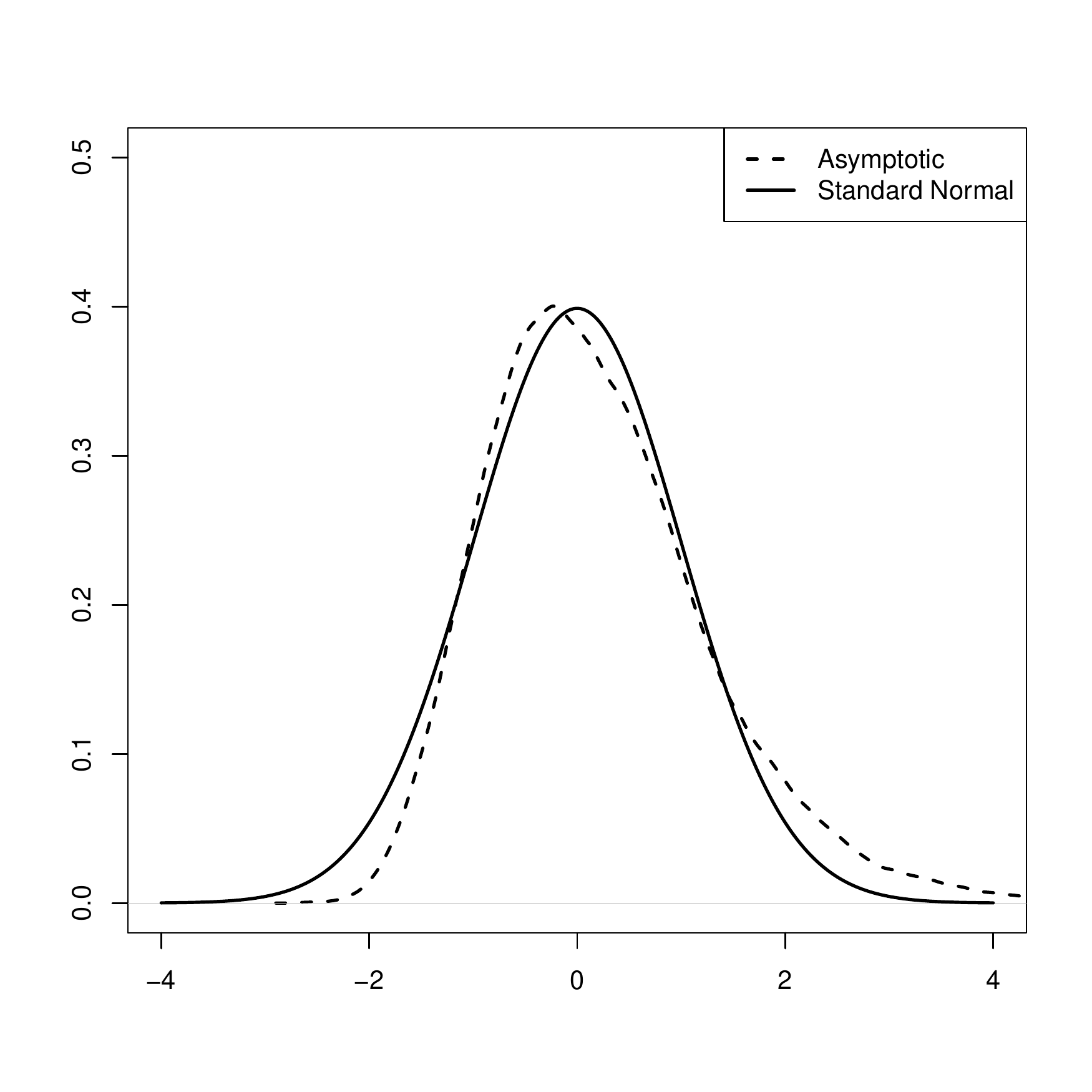}
  \begin{tabular}{lcrp{1.5in}}
(a)& $p=475$, $n=500$, $\bnu \sim\mathcal{TN} _q (\mathbf 0, \mathbf I_q) $.\\
\end{tabular}
\begin{tabular}{lcrp{1.5in}}
(b)&  $p=950$, $n=1000$, $\bnu \sim\mathcal{TN} _q (\mathbf 0, \mathbf I_q) $.\\
\end{tabular}
  \includegraphics[width=8cm]{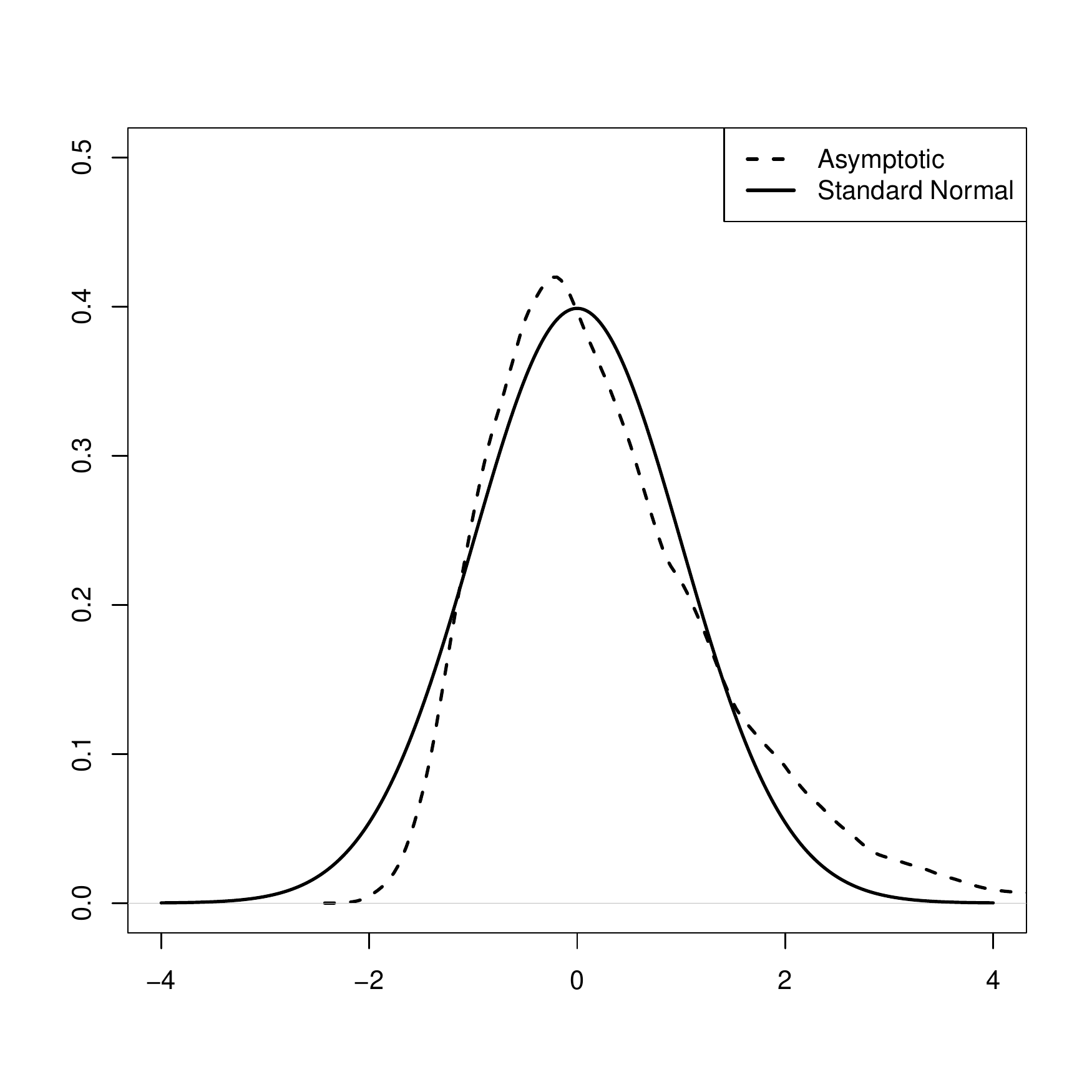}
  \includegraphics[width=8cm]{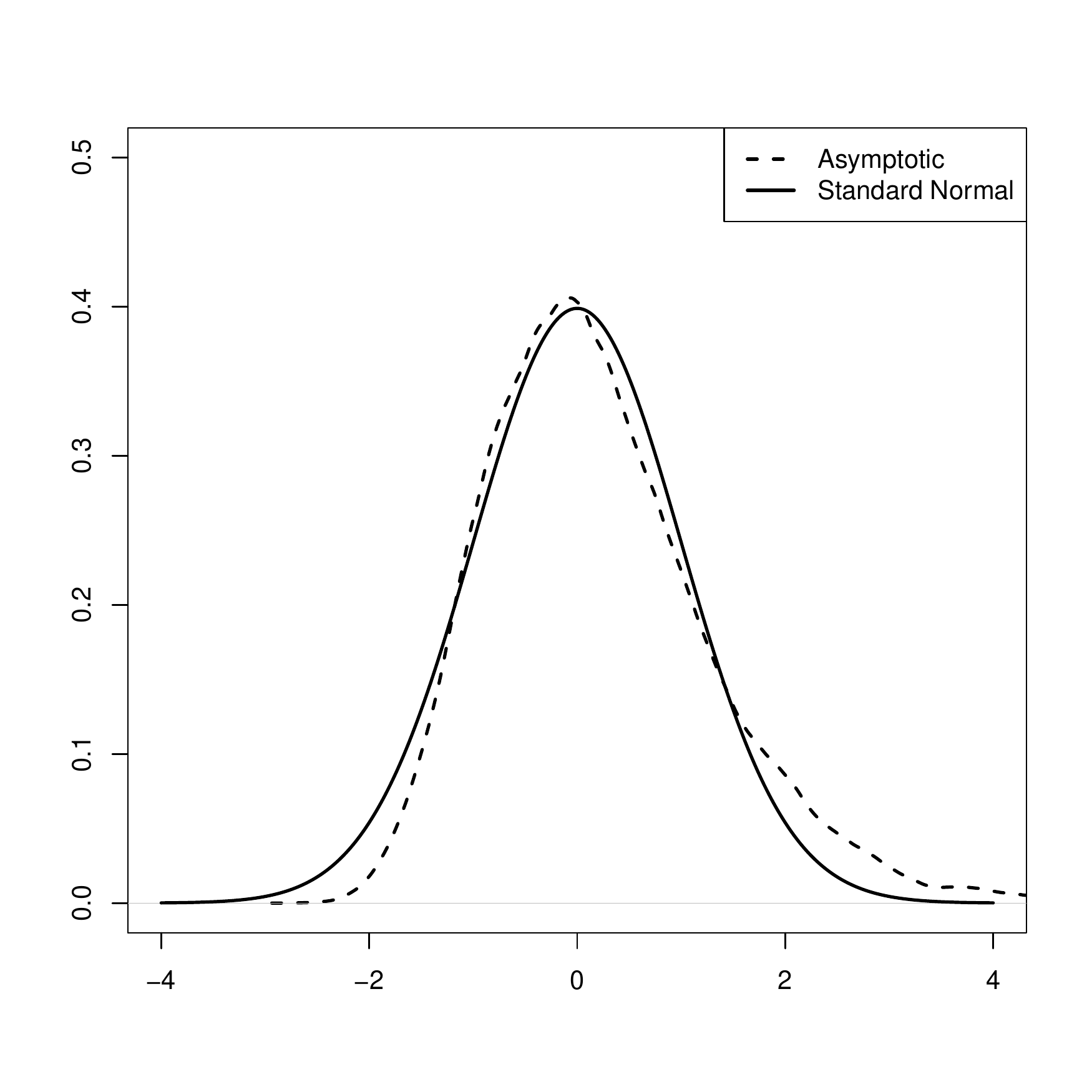}
  \begin{tabular}{lcrp{1.5in}}
(c)& $p=475$, $n=500$, $\bnu \sim\mathcal{GAL} _q (\mathbf 1_q, \mathbf I_q, 10).$\\
\end{tabular}
\begin{tabular}{lcrp{1.5in}}
(d)& $p=950$, $n=1000$, $\bnu \sim\mathcal{GAL} _q (\mathbf 1_q, \mathbf I_q, 10).$ \\
\end{tabular}
  \caption{The kernel density estimator of the asymptotic distribution as given in Theorem \ref{th4} for $c=0.95$.}
  \label{fig8}
\end{figure}

\end{document}